\documentclass[
a4paper,
11pt,
twoside,
]{article}
\usepackage[utf8]{inputenc}
\usepackage[T1]{fontenc}
\usepackage[sc]{mathpazo}
\usepackage[english]{babel}
\usepackage{geometry}
\usepackage{amsmath,amsfonts,amssymb,amsthm}
\usepackage[final]{hyperref}
\usepackage{graphicx}
\usepackage{mathtools}
\usepackage[dvipsnames]{xcolor}
\usepackage{pdflscape}
\usepackage{etoolbox}
\usepackage{fancyhdr}
\usepackage{lastpage}
\usepackage{ifdraft}
\usepackage{hyphenat}
\ifdraft{\usepackage{showlabels}}{}
\usepackage{enumitem}
\usepackage{siunitx}
\usepackage[mathlines,pagewise]{lineno}

\DeclareMathOperator{\Span}{span}
\newcommand{\dd}{\mathop{}\!\mathrm{d}}
\theoremstyle{plain}
\newtheorem{theorem}{Theorem}[section]
\newtheorem{lemma}[theorem]{Lemma}
\newtheorem{proposition}[theorem]{Proposition}
\newtheorem{corollary}[theorem]{Corollary}
\newtheorem{remark}[theorem]{Remark}

\numberwithin{equation}{section}
\numberwithin{table}{section}
\numberwithin{figure}{section}

\fancypagestyle{mypagestyle}{%
  \fancyhf{}%
  \fancyhead[LO,RE]{\scriptsize Convergence of collocation methods for periodic solutions of RFDEs}%
  \fancyhead[RO,LE]{\scriptsize A. And\`o and D. Breda}%
  \fancyfoot[LE,RO]{\scriptsize \thepage\,/\,\pageref*{LastPage}}%
}
\title{Convergence analysis of collocation methods for computing periodic solutions of\\ retarded functional differential equations}
\author{
Alessia And\`o$^{1,2}$ and Dimitri Breda$^{1,3}$\\[.5em]
\small $^{1}$CDLab -- Computational Dynamics Laboratory\\[-.2em]
\small Department of Mathematics, Computer Science and Physics -- University of Udine\\[-.2em]
\small via delle scienze 206, 33100 Udine, Italy\\[.5em]
\small $^{2}$\texttt{ando.alessia@spes.uniud.it}\\[.5em]
\small $^{3}$\texttt{dimitri.breda@uniud.it}
}
\date{\today}
\begin{document}
\clearpage
\maketitle
\thispagestyle{empty}
\begin{abstract}
We analyze the convergence of piecewise collocation methods for computing periodic solutions of general retarded functional differential equations under the abstract framework recently developed in [S. Maset, {\it Numer. Math.} (2016) 133(3):525-555], [S. Maset, {\it SIAM J. Numer. Anal.} (2015) 53(6):2771--2793] and [S. Maset, {\it SIAM J. Numer. Anal.} (2015) 53(6):2794--2821]. We rigorously show that a reformulation as a boundary value problem requires a proper {\it infinite-dimensional} boundary periodic condition in order to be amenable to such analysis. In this regard, we also highlight the role of the {\it period} acting as an unknown parameter, which is critical since it is directly linked to the course of time. Finally, we prove that the {\it finite element method} is convergent, {\color{black}while we limit ourselves to commenting on the infeasibility of this approach as far as the {\it spectral element method} is concerned.}
\end{abstract}


\section{Introduction}
\label{s_introduction}
Periodic behaviors emerge quite often in the dynamical analysis of systems. Their importance is even greater when dealing with complex and realistic models portraying natural phenomena, such as, e.g., the evolution of epidemics or population dynamics. Some form of delay is usually intrinsic in their description, and this is definitely the case we are focused in.

While the subject of periodic solutions is well settled for ordinary differential equations as far as computation, continuation and bifurcation are considered (see, e.g., the package \texttt{MatCont} \cite{matcont} as a representative of the state-of-the-art), relevant theory and computational tools have not yet reached a full maturity for delay equations. Among the main references for delay differential equations is \texttt{DDE-Biftool} \cite{ddebiftool,engacm02}, where the computation of periodic solutions is based on the work \cite{elir00}, extending the classic piecewise orthogonal collation methods already used for the {\color{black} case of ordinary differential equations} (see, e.g., \cite{acr81,amr88}). But when it comes to dealing with more complicated systems, involving also renewal or Volterra integral and integro-differential equations, the lack is evident \cite{bdgsv16,bdls16}.

The present work was originally guided by the need to fill this gap, trying to extend the numerical collocation \cite{elir00} to {\it Renewal Equations (REs)}. Besides the basic aspects concerning implementation and computation, effort was initially devoted to providing sources from the literature for the analysis of the error and the relevant convergence. In realizing that even these sources are lacking or at least not general (see Section \ref{s_literature} below), we decided to tackle a full investigation starting from the basic case of {\it Retarded Functional Differential Equations (RFDEs)}, mainly inspired by the recent ``trilogy'' of papers \cite{mas15I,mas15II,mas15NM}, which deals with the numerical solution of {\it Boundary Value Problems (BVPs)}.

The outcome, to the best of the authors' knowledge, is the first rigorous and fully-detailed analysis of error and convergence of piecewise collocation methods for the computation of periodic solutions of general RFDEs. {\color{black}Let us anticipate that the proposed approach is based on collocating the derivative of the solution following \cite{mas15NM} and in view of extension to REs as discussed in Section \ref{s_validation_t}.}

\bigskip
In this introduction we start in Section \ref{s_boundary} by deriving two equivalent BVP formulations for general RFDEs in view of computing periodic solutions. A discussion of the relevant literature is presented in Section \ref{s_literature}. Aims, contributions and results of the analysis we propose are summarized in Section \ref{s_aims}. Finally, some notations on relevant function spaces are introduced and suitably discussed in Section \ref{s_notation}.

The rest of the paper is organized in three main parts, namely Section \ref{s_abstract}, dealing with the validation of the required {\it theoretical} assumptions; Section \ref{s_discretization} presenting the discretization and validating the required {\it numerical} assumptions; Section \ref{s_convergence} concerning the final {\it convergence analysis}. Eventually, some closing remarks are given in Section \ref{s_concluding}, while Appendix \ref{s_appendix} collects a bunch of results used in the proofs developed in the above mentioned main sections.
\subsection{Boundary value problems}
\label{s_boundary}
Let $d$ be a positive integer, $a,b\in\mathbb{R}$ with $a<b$ and $\mathbb{F}([a,b],\mathbb{R}^{d}):=\{f:[a,b]\rightarrow\mathbb{R}^{d}\}$.

\bigskip
Let us consider the RFDE
\begin{equation}\label{rfdes}
\mathtt{y}'(\mathtt{t})=G(\mathtt{y}_{\mathtt{t}}),
\end{equation}
where $G:\mathtt{Y}\rightarrow\mathbb{R}^{d}$ is a function defined on a {\it state space} $\mathtt{Y}\subseteq\mathbb{F}([-\tau,0],\mathbb{R}^{d})$ for $\tau>0$ a given maximum delay. As usual \cite{diekmann95,hale77}, the {\it state} $\mathtt{y}_{\mathtt{t}}\in\mathtt{Y}$ is defined as 
\begin{equation}\label{states}
\mathtt{y}_{\mathtt{t}}(\sigma):=\mathtt{y}(\mathtt{t}+\sigma),\quad\sigma\in[-\tau,0],
\end{equation}
and the time derivative in \eqref{rfdes} is intended from the right.

\bigskip
The goal is to compute a periodic solution of \eqref{rfdes}, assuming its existence. As this solution is unknown, so is its period, say $\omega>0$. To deal with this lack of {\color{black}information} one usually resorts to a scaling of time, see, e.g., \cite{elir00}. Although numerically convenient, this scaling plays an essential role in the analysis of convergence, a role that to the best of our knowledge has not received the deserved attention in the literature, possibly because not even the general form \eqref{rfdes} has been adequately considered (in favor of maybe more {\it practical} instances like $g(\mathtt{y}(\mathtt{t}),\mathtt{y}(\mathtt{t}-\tau))$ or similar ones). Let us then define {\color{black}$s_{\omega}:[-\tau,0]\rightarrow\mathbb{R}$} as $t=s_{\omega}(\mathtt{t}):=\mathtt{t}/\omega$, which transforms \eqref{rfdes} into
\begin{equation}\label{rfdet}
y'(t)=\omega G(y_{t}\circ s_{\omega})
\end{equation}
by $y(t):=\mathtt{y}(s_{\omega}^{-1}(t))=\mathtt{y}(\omega t)=\mathtt{y}(\mathtt{t})$. In particular, if $\mathtt{y}$ is an $\omega$-periodic solution of \eqref{rfdes}, correspondingly $y$ is a $1$-periodic solution of \eqref{rfdet} and vice versa. Recall that periodic solutions are defined on the whole line.

The state of \eqref{rfdet} should lie in $\mathbb{F}([-r,0],\mathbb{R}^{d})$ for $r:=s_{\omega}(\tau)=\tau/\omega$ unknown, so that it would change according to the concerned periodic solution of \eqref{rfdes}. To avoid this variability, we choose as a state space a set $Y\subseteq\mathbb{F}([-1,0],\mathbb{R}^{d})$, defining  $y_{t}\in Y$ as
\begin{equation}\label{statet}
y_{t}(\theta):=y(t+\theta),\quad\theta\in[-1,0].
\end{equation}
Indeed, if $\tau\leq\omega$ then $r\leq1$ and thus we deal with an enlarged state space. Otherwise, as long as $\tau$ is finite, we can always refer to a sufficiently large multiple of the period in order to fall into the previous case. Finally, with respect to \eqref{states},
\begin{equation}\label{somega}
\theta=s_{\omega}(\sigma)=\frac{\sigma}{\omega}
\end{equation}
for $\sigma\in[-\tau,0]$ as far as $\theta\in[-r,0]\subseteq[-1,0]$.
\begin{remark}
Let us anticipate that in case of numerical approximation through iterative methods requiring an initial guess of the solution (as is the case for numerical continuation, see below), if the initial guess of $\omega$ is less than or too close to $\tau$, then one can start from $k\omega$ with a suitable integer $k>1$.
\end{remark}

\bigskip
A periodic solution is usually characterized through a BVP, obtained by considering \eqref{rfdet} over one period, viz. $[0,1]$, together with a periodicity condition and a {\it phase} condition to remove translational invariance, see, e.g., \cite{elir00} again. In the case of RFDEs like \eqref{rfdet}, the evaluation of $y$ through $y_{t}$ in $G$ may regard time instants (or intervals) falling to the left of $[0,1]$. If so, one possibility is to exploit the implicitly assumed periodicity to bring the evaluation back to the desired domain. This corresponds to defining the {\color{black} {\it periodic extension} $\overline y:[-1,1]\to\mathbb{R}^d$ of $y:[0,1]\to\mathbb{R}^d$, and then the} {\it periodic state} $\overline{y_{t}}\in Y$ according to \eqref{statet}, i.e., for $t\in[0,1]$,
\begin{equation}\label{baryttheta}
\overline{y_{t}}(\theta):=\begin{cases}
y(t+\theta), & t+\theta\in[0,1], \\
y(t+\theta+1), & t+\theta\in[-1,0),
\end{cases}
\end{equation}
recalling that $\tau\leq\omega$, i.e., $\theta\in[-1,0]$. Note that in view of the fact that the right-hand side $G$ of \eqref{rfdes} acts properly on the original state space $\mathtt{Y}$, \eqref{somega} and \eqref{baryttheta} lead to considering
\begin{equation*}
\overline{y_{t}}\circ s_{\omega}(\sigma):=\begin{cases}
y(t+s_{\omega}(\sigma)), & t+s_{\omega}(\sigma)\in[0,1], \\
y(t+s_{\omega}(\sigma)+1), & t+s_{\omega}(\sigma)\in[-1,0),
\end{cases}
\end{equation*}
for $\sigma\in[-\tau,0]$. With the above device, the relevant BVP reads
\begin{equation}\label{bvp1}
\left\{\setlength\arraycolsep{0.1em}\begin{array}{ll}
y'(t)=\omega G(\overline{y_{t}}\circ s_{\omega}),&\quad t\in[0,1],\\[2mm]
y(0)=y(1)\\[2mm]
p(y)=0.
\end{array}\right.
\end{equation}
The solution $y$ of \eqref{bvp1} is intended as an element of a set $Y^{+}\subseteq\mathbb{F}([0,1],\mathbb{R}^{d})$. Moreover, $p:Y^{+}\rightarrow\mathbb{R}$ denotes the phase condition, which we assume to be linear, continuous and able to eliminate translational invariance. For example, a {\it trivial} phase condition is one of the form
$y_k(0)=\hat y$ for some $k\in\{1,\ldots,d\}$ and a given $\hat y\in\mathbb{R}$. An {\it integral} phase condition is one of the form
\begin{equation*}
\int_{0}^{1}y^{T}(t)\tilde{y}'(t)\dd t=0,
\end{equation*}
where $\tilde{y}$ is a given reference $1$-periodic solution. Either $\hat y$ or $\tilde y$ are available in the natural continuation framework where periodic solutions are usually computed \cite{doe07}: indeed, the former may be a coordinate of the equilibrium giving rise to a limit cycle through a Hopf bifurcation; the latter may be the periodic solution computed at the previous continuation step. Note that in \eqref{bvp1} the periodicity condition (i.e., the first of the boundary conditions) concerns only the values of the solution at the extrema of $[0,1]$ since the periodicity is included in the right-hand side through \eqref{baryttheta}. As such, it is a condition in $\mathbb{R}^{d}$.

\bigskip
Alternatively to \eqref{bvp1}, one can still consider a BVP for the original scaled equation \eqref{rfdet} by imposing the periodicity to the states at the extrema of the period, rather than to the solution values:
\begin{equation}\label{bvp2}
\left\{\setlength\arraycolsep{0.1em}\begin{array}{ll}
y'(t)=\omega G(y_{t}\circ s_{\omega}),&\quad t\in[0,1],\\[2mm]
y_{0}=y_{1}\\[2mm]
p(y\vert_{[0,1]})=0.
\end{array}\right.
\end{equation}
In this case the solution $y$ is intended as an element of a set $Y^{\pm}\subseteq\mathbb{F}([-1,1],\mathbb{R}^{d})$  and the periodicity condition concerns the state space $Y$.
\subsection{Literature}
\label{s_literature}
The literature on the numerical computation of periodic solutions of delay equations through relevant BVPs is rather rich (also for {\it neutral} and {\it state-dependent} problems). Let us suggest \cite[Section 1.1]{mas15NM} for a detailed account. By far most of the works concern formulation \eqref{bvp1} \cite{amr88,amr07,bad85,bk06,bel83,bel85,bz84,ed02,elir00,liu94,mas15I,mas15II,mas15NM,rt74}, while only few address formulation \eqref{bvp2} \cite{elr00,lelr97,vl05}. A short discussion on the two equivalent alternatives can be found in \cite[Section 2]{elir00}, where the name Halanay's BVP for \eqref{bvp2} is also recalled from \cite{km99}. Finally, let us note that very few papers deal with theoretical error and convergence analyses, e.g., \cite{bad85,ed02}. In particular, \cite{bad85} does not consider explicitly periodic problems or the presence of unknown parameters, while \cite{ed02} deals with linear problems and assumes the period to be known (and equal to $1$). For further references on these and other aspects see \cite{mas15I,mas15II,mas15NM}, which represent a thorough research on the subject and tackle the solution of BVPs as fixed point problems, furnishing a solid framework for the convergence analysis. The approach proposed in \cite{mas15NM} is quite abstract, while a more concrete collocation framework is illustrated in \cite{mas15I,mas15II}. However, the treatment is devoted to general BVPs, not necessarily restricted to the periodic case, which is never considered explicitly indeed.
\subsection{Aims, contributions and results}
\label{s_aims}
\hspace{-.01mm}{\color{black}The a}im of the present work is to develop a rigorous and fully-detailed analysis of error and convergence of piecewise collocation methods for the computation of periodic solutions of general RFDEs by following the abstract approach discussed in \cite{mas15NM}.

In the following sections we try to apply this general framework to both \eqref{bvp1} and \eqref{bvp2}. Note that the former formulation is the periodic instance of the  {\it side condition} considered in \cite{mas15NM} (page 526), while \eqref{bvp2} is not even mentioned therein. In spite of this, we show that only the latter is amenable of the treatment in \cite{mas15NM}, while the former fails to satisfy (some of) the required assumptions. Therefore, in what follows we give formal proofs only for \eqref{bvp2}, reserving to comment about \eqref{bvp1} up to the point in Section \ref{s_abstract} where it definitively fails to fit into \cite{mas15NM}\footnote{{\color{black}Let us remark that beyond the mentioned (technical) deficiencies, the authors are not aware of any numerical reasons for the failure of formulation \eqref{bvp1}, which is indeed the most widely used for simulations.}}.

\bigskip
Let us clarify that the contributions of this investigation are represented by the developments of proofs of the validity of the {\it theoretical} (Section \ref{s_abstract}) and {\it numerical} (Section \ref{s_discretization}) assumptions required to apply the abstract approach of \cite{mas15NM}, in the case of periodic BVPs. On the one hand, ``This task is far from trivial'' \cite[end of page 2791]{mas15I}. On the other hand, we soon anticipate that in the periodic case the period plays the role of an unknown parameter of the problem. Although unknown parameters are explicitly considered in \cite{mas15NM}, what is neglected therein is that the unknown period is linked to the course of time, and thus to the domain of the BVP. Exactly this fact is a cause of major troubles in the effort of validating the above assumptions. {\color{black}The hypotheses on the right-hand side and on the discretization under which such assumptions are validated are listed at the beginning of Section \ref{s_validation_t} and of Section \ref{s_validation_n}.}

\bigskip
The discretization considered in \cite{mas15NM} consists in the collocation of the derivative of the solution, being it devoted mainly to neutral problems. Here we keep on following this same technique even if we restrict our treatment to non-neutral equations. On the one hand, the adaptation to the periodic case is itself way far from being trivial. On the other hand, in view of our original motivation, exactly this strategy extends to the case of renewal equations by interpreting the derivative of the solution of a neutral RFDE as the solution of a corresponding renewal equation. As the analysis of the case of non-neutral RFDEs has revealed itself already complicated under this framework, we leave the extensions to neutral and renewal equations as the logical steps to be developed in the next future.

\bigskip
Concerning the method and its convergence, as the former is based on piecewise collocation (following the traditional practical approaches in both \texttt{MatCont} and \texttt{DDE-Biftool}), convergence can be potentially attained by either the {\it Finite Element Method (FEM)} or the {\it Spectral Element Method (SEM)}. It turns out that the framework of \cite{mas15NM} can be used to prove the convergence of the FEM, leading to the expected results about the order of convergence under suitable regularity assumptions. This is the main content of Section \ref{s_convergence}, namely Theorem \ref{t_epsL}. As for the SEM, although not used in practical implementations and therefore marginal to our primary interest, it is not yet clear if the current analysis can lead to prove convergence. A discussion on this aspect is contained in Section \ref{s_SEM}.

\subsection{Notations and function spaces}
\label{s_notation}
Prior to starting, let us fix some notations, mainly relevant to the choices of subsets of $\mathbb{F}([a,b],\mathbb{R}^{d})$ in view of \eqref{bvp1} and \eqref{bvp2}. In particular, we use $B^{\infty}$ in place of $\mathbb{F}$ for measurable and bounded functions and $B^{1,\infty}$ for continuous functions with measurable and bounded first derivative. Let us remark again that time derivatives are intended from the right. If $|\cdot|$ denotes a norm in finite-dimensional spaces and $\|f\|_{\infty}:=\sup_{t\in[a,b]}|f(t)|$ is the uniform norm, then $B^{\infty}([a,b],\mathbb{R}^{d})$ and $B^{1,\infty}([a,b],\mathbb{R}^{d})$ become Banach spaces respectively with 
\begin{equation}\label{normB1inf}
\|f\|_{B^{\infty}}:=\|f\|_{\infty},\qquad\|f\|_{B^{1,\infty}}:=\|f\|_{\infty}+\|f'\|_{\infty}.
\end{equation}
Occasionally, we may use also $C$ for continuous functions and $C^{1}$ for continuously differentiable ones, with $\|f\|_{C}=\|f\|_{\infty}$ and $\|f\|_{C^{1}}=\|f\|_{\infty}+\|f'\|_{\infty}$ again. Also other spaces will be temporarily introduced for a tentative analysis of \eqref{bvp1}, and in case of product spaces $U=U_{1}\times U_{2}$ we choose
\begin{equation}\label{normprod}
\|\cdot\|_{U}=\max\{\|\cdot\|_{U_{1}},\|\cdot\|_{U_{2}}\},
\end{equation}
which makes $U$ a Banach space if both $U_{1}$ and $U_{2}$ are. 

In addition, for $U,V$ normed spaces, according to \cite[Definition 1.1.5]{ampr95} we denote by $DA(u)\in\mathcal{L}(U,V)$ the Fr\'echet differential at $u\in U$ of a map $A:U\rightarrow V$, where $\mathcal{L}(U,V)$ is the set of linear bounded operators $U\rightarrow V$, equipped with the induced norm
\begin{equation}\label{normA}
\|A\|_{V\leftarrow U}=\sup_{u\in U\setminus\{0\}}\frac{\|Au\|_{V}}{\|u\|_{U}}.
\end{equation}
We denote also by $\mathcal{C}^{1}(U,V)$ the set of maps $A:U\rightarrow V$ which are continuously differentiable in the sense of Fr\'echet, i.e., their Fr\'echet derivative $DA$ is continuous as a map $U\rightarrow\mathcal{L}(U,V)$. Finally, for a Banach space $X$, $\overline{B}(x,r)$ denotes the closed ball of center $x\in X$ and radius $r>0$.

\bigskip
We close this introduction motivating the choices above about measurable and bounded functions instead of continuous ones, the latter being a (if not {\it the}) standard for RFDEs. Among the main reasons is the fact that, concerning formulation \eqref{bvp1}, for any $t\in[0,1)$ the map $\theta\mapsto\overline{y_{t}}(\theta)$ introduced in \eqref{baryttheta} is continuous if and only if $y(0)=y(1)$, a condition satisfied by solutions of \eqref{bvp1}. Otherwise, a jump discontinuity with jump $y(0)-y(1)$ appears at $\theta=-t$, in which we have continuity only from the right by virtue of \eqref{baryttheta}. Thus the classical choice $\mathtt{Y}=C([-\tau,0],\mathbb{R}^{d})$ would lead the right-hand side $G$ in \eqref{rfdes} to act outside its domain in those cases where the periodic boundary condition $y(0)=y(1)$ is not satisfied, due to the trick of recovering periodicity through \eqref{baryttheta}. In this respect, we anticipate indeed that in the following analysis of convergence such situations occur, either because boundary conditions other than the periodic one may be imposed (e.g., in Proposition \ref{p_Ax*2} below), or simply because we must deal with neighborhoods of the sought periodic solution, which by no means contain only functions satisfying the periodic boundary condition. Nevertheless, it is exactly this lack of continuity that leads to the inapplicability of the approach in \cite{mas15NM} as we show in Section \ref{s_validation_t}.

As far as formulation \eqref{bvp2} is concerned, instead, continuity is guaranteed by the usual definition of solution of RFDEs, given that an initial value problem is implicitly defined through (the yet unknown) $y_{0}$. Nevertheless, the problem illustrated above would arise for the first derivative. Indeed, the latter is not necessarily continuous at $0$ even if one chose to work with $\mathtt{Y}=C^{1}([-\tau,0],\mathbb{R}^{d})$, unless the extra condition $\psi'(0^{-})=G(\psi)$ were imposed to any $\psi\in\mathtt{Y}$. In Section \ref{s_validation_t} we show that the choice $B^{\infty}$ for the derivative is thus necessary, and that at the same time the lack of continuity for the latter is balanced by the use of right-hand derivatives with respect to time, as it is correct in the field of RFDEs.

Finally, let us remark that measurable and bounded functions can be used in the theory of RFDEs if one slightly weakens the notion of solution \cite[Section 0.2]{diekmann95}.
\section{The abstract approach towards fixed point problems}
\label{s_abstract}
We first summarize the main ingredients of the abstract approach proposed in \cite{mas15NM} to numerically treat BVPs for RFDEs, described therein for neutral problems. The backbone of the methodology consists in translating the BVP into a fixed point problem. In Section \ref{s_equivalent} we apply this translation to the two equivalent formulations \eqref{bvp1} and \eqref{bvp2}. In Section \ref{s_validation_t} we deal with the validation of the theoretical assumptions required in \cite{mas15NM} to use the framework developed therein for the relevant error and convergence analyses. As for the first formulation we show that it cannot satisfy the third of these assumptions (Proposition \ref{p_Ax*1} below), unless we restrict the relevant spaces by adding specific constraints. Nevertheless, these additional constraints immediately cause the failure of the fourth and last of these assumptions (Proposition \ref{p_Ax*2} below). As for the second formulation, instead, all the required theoretical assumptions can be satisfied under reasonable regularity hypotheses on $G$ in \eqref{rfdes}, and thus we give the relevant formal proofs. In particular, the last of these assumptions appears tricky to satisfy, to the point that its proof is among the main contributions of the present work. Eventually, the tools on which it relies are dealt with in a dedicated section, namely Section \ref{s_nongeneric}. 

\bigskip
The general BVP considered in \cite{mas15NM} has the form
\begin{equation*}
\begin{cases}
u=\mathcal{F}(\mathcal{G}(u,\alpha),u,\beta)\\
\mathcal{B}(\mathcal{G}(u,\alpha),u,\beta)=0.
\end{cases}
\end{equation*}
The first line represents the functional equation of neutral type, and the second line represents the boundary condition. $u$ is the derivative of the concerned solution $v$, the former living in a Banach space $\mathbb{U}\subseteq\mathbb{F}([a,b],\mathbb{R}^{d})$, the latter living in a normed space $\mathbb{V}\subseteq\mathbb{F}([a,b],\mathbb{R}^{d})$. The operator $\mathcal{G}:\mathbb{U}\times\mathbb{A}\rightarrow\mathbb{V}$ represents a (linear) Green operator which reconstructs the solution $v=\mathcal{G}(u,\alpha)$ given its derivative $u$ and a value $\alpha$ in a Banach space $\mathbb{A}$ containing the range of the solution; a classic example is
\begin{equation}\label{green}
\mathcal{G}(u,\alpha)(t):=\alpha+\int_{c}^{t}u(s)\dd s,\quad t\in[a,b],
\end{equation}
for some $c\in[a,b]$, in which case $\alpha=v(c)$. $\beta$ is a vector of possible parameters, usually varying together with the solution and living in a Banach space $\mathbb{B}$. The function $\mathcal{F}:\mathbb{V}\times\mathbb{U}\times\mathbb{B}\rightarrow\mathbb{U}$ is the right-hand side of the concerned equation while $\mathcal{B}:\mathbb{V}\times\mathbb{U}\times\mathbb{B}\rightarrow\mathbb{A}\times\mathbb{B}$ represents the boundary condition. The latter usually includes a proper boundary condition on the solution (the component in $\mathbb{A}$) and a further condition posing the necessary constraints on the parameters (the component in $\mathbb{B}$) .

\bigskip
Eventually, in \cite{mas15NM}, the so-called {\it Problem in Abstract Form} (PAF) consists in finding $(v^{\ast},\beta^{\ast})\in \mathbb{V}\times\mathbb{B}$ with $v^{\ast}:=\mathcal{G}(u^{\ast},\alpha^{\ast})$ and $(u^{\ast},\alpha^{\ast},\beta^{\ast})\in \mathbb{U}\times\mathbb{A}\times\mathbb{B}$ such that 
\begin{equation}\label{PAF}
(u^{\ast},\alpha^{\ast},\beta^{\ast})=\Phi(u^{\ast},\alpha^{\ast},\beta^{\ast})
\end{equation}
for $\Phi:\mathbb{U}\times\mathbb{A}\times\mathbb{B}\rightarrow\mathbb{U}\times\mathbb{A}\times\mathbb{B}$ given by
\begin{equation}\label{Phi}
\Phi(u,\alpha,\beta):=
\begin{pmatrix}
\mathcal{F}(\mathcal{G}(u,\alpha),u,\beta)\\[2mm]
(\alpha,\beta)-\mathcal{B}(\mathcal{G}(u,\alpha),u,\beta)
\end{pmatrix}.
\end{equation}
In what follows we always use the superscript $^{\ast}$ to denote quantities relevant to fixed points.
\subsection{Equivalent formulations}
\label{s_equivalent}
Let us start with formulation \eqref{bvp1}. In this case the domain of the BVP is $[a,b]=[0,1]$. We choose $\mathbb{U}=\mathbb{U}_{1}$ and $\mathbb{V}=\mathbb{V}_{1}$ for $\mathbb{U}_{1},\mathbb{V}_{1}\subseteq Y^{+}$ and $Y^{+}$ as introduced in Section \ref{s_introduction}. We choose also $\mathbb{A}=\mathbb{A}_{1}=\mathbb{R}^{d}$. The only unknown parameter is the original period, therefore we fix $\mathbb{B}=\mathbb{B}_{1}=\mathbb{R}$ and use $\omega$ in place of $\beta$ once for all (recall anyway that the sought period $\omega^{\ast}$ is assumed to be positive). The Green operator $\mathcal{G}=\mathcal{G}_{1}$ is chosen as the operator $\mathcal{G}_{1}:\mathbb{U}_{1}\times\mathbb{A}_{1}\rightarrow Y^{+}$ with action similar to \eqref{green}; in particular we define
\begin{equation}\label{G1}
\mathcal{G}_{1}(u,\alpha)(t):=\alpha+\int_{0}^{t}u(s)\dd s,\quad t\in[0,1].
\end{equation}
Then the solutions of \eqref{bvp1} are exactly the pairs $(v^{\ast},\omega^{\ast})\in\mathbb{V}_{1}\times\mathbb{B}_{1}$ with $v^{\ast}:=\mathcal{G}_{1}(u^{\ast},\alpha^{\ast})$ and $(u^{\ast},\alpha^{\ast},\omega^{\ast})\in\mathbb{U}_{1}\times\mathbb{A}_{1}\times\mathbb{B}_{1}$ the fixed points of the map $\Phi_{1}:\mathbb{U}_{1}\times\mathbb{A}_{1}\times\mathbb{B}_{1}\rightarrow\mathbb{U}_{1}\times\mathbb{A}_{1}\times\mathbb{B}_{1}$ defined by
\begin{equation*}
\Phi_{1}(u,\alpha,\omega):=
\begin{pmatrix}
\omega G(\overline{\mathcal{G}_{1}(u,\alpha)_{\cdot}}\circ s_{\omega})\\[2mm]
\mathcal{G}_{1}(u,\alpha)(1)\\[2mm]
\omega-p(\mathcal{G}_{1}(u,\alpha))
\end{pmatrix}.
\end{equation*}
\noindent Above $\alpha$ plays the role of $v(0)$, and $v_{\cdot}$ denotes the map $t\mapsto v_{t}$ according to \eqref{statet}, about which we recall also \eqref{baryttheta} and the comments closing Section \ref{s_notation}. With the above choices it follows that \eqref{bvp1} leads to an instance of \eqref{Phi} with $\mathcal{F}=\mathcal{F}_{1}:\mathbb{V}_{1}\times\mathbb{U}_{1}\times\mathbb{B}_{1}\rightarrow\mathbb{U}_{1}$ and $\mathcal{B}=\mathcal{B}_{1}:\mathbb{V}_{1}\times\mathbb{U}_{1}\times\mathbb{B}_{1}\rightarrow\mathbb{A}_{1}\times\mathbb{B}_{1}$ given respectively by
\begin{equation*}
\mathcal{F}_{1}(v,u,\omega):=\omega G(\overline{v_{\cdot}}\circ s_{\omega}),\qquad\mathcal{B}_{1}(v,u,\omega):=
\begin{pmatrix}
v(0)-v(1)\\[2mm]
p(v)
\end{pmatrix}.
\end{equation*}
Note that in our case the problem is not neutral. Moreover, the boundary operator is linear and includes both the periodicity and the phase conditions, none of which depend on $\omega$.

\bigskip
Now let us consider \eqref{bvp2}. The domain of the BVP is again $[a,b]=[0,1]$, but in this case we choose $\mathbb{U}=\mathbb{U}_{2}\subseteq Y^{+}$, $\mathbb{V}=\mathbb{V}_{2}\subseteq Y^{\pm}$ and $\mathbb{A}=\mathbb{A}_{2}\subseteq Y$ for $Y$, $Y^{+}$ and $Y^{\pm}$ as introduced in Section \ref{s_introduction}, as well as $\mathbb{B}=\mathbb{B}_{2}=\mathbb{R}$. Let us remark that in \cite{mas15NM} the treatment is restricted to the case where $\mathbb{A}$ is finite-dimensional, so that this alternative formulation brings in this novelty explicitly. Accordingly, we define the Green operator $\mathcal{G}=\mathcal{G}_{2}$ as the operator $\mathcal{G}_{2}:\mathbb{U}_{2}\times\mathbb{A}_{2}\rightarrow Y^{\pm}$ given by
\begin{equation}\label{G2}
\mathcal{G}_{2}(u,\psi)(t):=\begin{cases}
\displaystyle\psi(0)+\int_{0}^{t}u(s)\dd s,&t\in[0,1],\\[2mm]
\psi(t),&t\in[-1,0].
\end{cases}
\end{equation}
Note that $\mathcal{G}_{2}$ corresponds to the operator $V$ first introduced in \cite{bmvsinum12}. Then the solutions of \eqref{bvp2} are exactly the pairs $(v^{\ast},\omega^{\ast})\in\mathbb{V}_{2}\times\mathbb{B}_{2}$ with $v^{\ast}:=\mathcal{G}_{2}(u^{\ast},\psi^{\ast})$ and $(u^{\ast},\psi^{\ast},\omega^{\ast})\in\mathbb{U}_{2}\times\mathbb{A}_{2}\times\mathbb{B}_{2}$ the fixed points of the map $\Phi_{2}:\mathbb{U}_{2}\times\mathbb{A}_{2}\times\mathbb{B}_{2}\rightarrow\mathbb{U}_{2}\times\mathbb{A}_{2}\times\mathbb{B}_{2}$ defined by
\begin{equation}\label{Phi2}
\Phi_{2}(u,\psi,\omega):=
\begin{pmatrix}
\omega G(\mathcal{G}_{2}(u,\psi)_{\cdot}\circ s_{\omega})\\[2mm]
\mathcal{G}_{2}(u,\psi)_{1}\\[2mm]
\omega-p(\mathcal{G}_{2}(u,\psi)\vert_{[0,1]})
\end{pmatrix}.
\end{equation}
\noindent Above $\psi$ plays the role of $v_{0}$. With these choices it follows that \eqref{bvp2} leads to an instance of \eqref{Phi} with $\mathcal{F}=\mathcal{F}_{2}:\mathbb{V}_{2}\times\mathbb{U}_{2}\times\mathbb{B}_{2}\rightarrow\mathbb{U}_{2}$ and $\mathcal{B}=\mathcal{B}_{2}:\mathbb{V}_{2}\times\mathbb{U}_{2}\times\mathbb{B}_{2}\rightarrow\mathbb{A}_{2}\times\mathbb{B}_{2}$ given respectively by
\begin{equation}\label{FB2}
\mathcal{F}_{2}(v,u,\omega):=\omega G(v_{\cdot}\circ s_{\omega}),\qquad\mathcal{B}_{2}(v,u,\omega):=
\begin{pmatrix}
v_{0}-v_{1}\\[2mm]
p(v\vert_{[0,1]})
\end{pmatrix}.
\end{equation}
Again, the boundary operator is linear and independent of either $u$ or $\omega$. Finally, note that with regards to the elements of $\mathbb{A}$ we slightly modified the notation with respect to the previous one for \eqref{bvp1}, since now they are states $\psi\in\mathbb{A}_{2}\subseteq Y$ rather than solution values $\alpha\in\mathbb{A}_{1}=\mathbb{R}^{d}$.
\subsection{Validation of the theoretical assumptions}
\label{s_validation_t}
Several theoretical assumptions are required in \cite{mas15NM} to apply the convergence framework proposed therein. We state them as propositions regarding the present context, furnishing proofs of their validity for formulation \eqref{bvp2} under specific choices of the concerned spaces (and their relevant norms as indicated in Section \ref{s_notation}) and regularity properties of the right-hand side $G$ in \eqref{rfdes}. For ease of reference throughout the text, we collect below the corresponding hypotheses\footnote{\color{black}See Section \ref{s_practical} for more \emph{practical} forms of $G$.}.
\begin{enumerate}[label=(T\arabic*),ref=(T\arabic*)]
\item\label{T1} $\mathtt{Y}=B^{\infty}([-\tau,0],\mathbb{R}^{d})$, $Y=B^{\infty}([-1,0],\mathbb{R}^{d})$.
\item\label{T2} $\mathbb{U}_{2}=B^{\infty}([0,1],\mathbb{R}^{d})$, $\mathbb{V}_{2}=B^{1,\infty}([-1,1],\mathbb{R}^{d})$, $\mathbb{A}_{2}=B^{1,\infty}([-1,0],\mathbb{R}^{d})$.
\item\label{T3} $G:\mathtt{Y}\rightarrow\mathbb{R}^{d}$ is Fr\'echet-differentiable at every $\mathtt{y}\in\mathtt{Y}$.
\item\label{T4} $G\in\mathcal{C}^{1}(\mathtt{Y},\mathbb{R}^{d})$ in the sense of Fr\'echet.
\item\label{T5} There exist $r>0$ and $\kappa\geq0$ such that
\begin{equation*}
\|DG(\mathtt{y})-DG(v^{\ast}_{t}\circ s_{\omega^{\ast}})\|_{\mathbb{R}^{d}\leftarrow\mathtt{Y}}\leq\kappa\|\mathtt{y}-v^{\ast}_{t}\circ s_{\omega^{\ast}}\|_{\mathtt{Y}}
\end{equation*}
for every $\mathtt{y}\in\overline B(v^{\ast}_{t}\circ s_{\omega^{\ast}},r)$, uniformly with respect to $t\in[0,1]$.
\end{enumerate}
\noindent As far as formulation \eqref{bvp1} is concerned, instead, we reserve just to comment on possible similar proofs as anticipated in Section \ref{s_introduction}.

Let us also remark that other assumptions required in \cite{mas15NM}, this time concerning numerical aspects, are dealt with in Section \ref{s_validation_n}, after the discretization scheme is presented.

\bigskip
The first theoretical assumption in \cite{mas15NM}, viz. Assumption A$\mathfrak{F}\mathfrak{B}$ (page 534), concerns the Fr\'echet-differentiability of the operators $\mathcal{F}$ and $\mathcal{B}$ appearing in \eqref{Phi}. The latter, given also the linearity of $p$, is linear in the second of \eqref{FB2}, hence Fr\'echet-differentiable. As for the former in the first of \eqref{FB2} we prove the following, where we underline that the derivative with respect to the period is intended from the right since the period affects the course of time in the domain of the state space through \eqref{somega} and derivatives with respect to time are defined from the right as already remarked (note that $s_{\omega}(\sigma)$ is increasing with respect to $\omega$).
\begin{proposition}\label{p_Afb}
Under \ref{T1}, \ref{T2} and \ref{T3}, $\mathcal{F}_{2}$ in the first of \eqref{FB2} is Fr\'echet-differentiable, from the right with respect to $\omega$, at every point $(\hat v,\hat u,\hat\omega)\in\mathbb{V}_{2}\times\mathbb{U}_{2}\times(0,+\infty)$ and
\begin{equation}\label{DF2}
D\mathcal{F}_{2}(\hat v,\hat u,\hat\omega)(v,u,\omega)=\mathfrak{L}_{2}(\cdot;\hat v,\hat\omega)v_{\cdot}\circ s_{\hat\omega}+\omega\mathfrak{M}_{2}(\cdot;\hat v,\hat\omega)
\end{equation}
for $(v,u,\omega)\in\mathbb{V}_{2}\times\mathbb{U}_{2}\times(0,+\infty)$, where, for $t\in[0,1]$,$(\hat v,\hat u,\hat\omega)\in\mathbb{V}_{2}\times\mathbb{U}_{2}\times(0,+\infty)$
\begin{equation}\label{L2}
\mathfrak{L}_{2}(t;v,\omega):=\omega DG(v_{t}\circ s_{\omega})
\end{equation}
and
\begin{equation}\label{M2}
\mathfrak{M}_{2}(t;v,\omega):=G(v_{t}\circ s_{\omega})-\mathfrak{L}_{2}(t;v,\omega)v'_{t}\circ s_{\omega}\cdot\frac{s_{\omega}}{\omega}.
\end{equation}
\end{proposition}
\begin{proof}
According to \cite[Definition 1.1.1]{ampr95}, let us directly prove that for $D\mathcal{F}_{2}$ in \eqref{DF2} through \eqref{L2} and \eqref{M2} we get, for $\omega>0$,
\begin{equation}\label{DF20}
\setlength\arraycolsep{0.1em}\begin{array}{rcl}
\|\mathcal{F}_{2}(\hat v+v,\hat u+u,\hat\omega+\omega)&-&\mathcal{F}_{2}(\hat v,\hat u,\hat\omega)-D\mathcal{F}_{2}(\hat v,\hat u,\hat\omega)(v,u,\omega)\|_{\mathbb{U}_{2}}\\[2mm]
&=&o\left(\|(v,u,\omega)\|_{\mathbb{V}_{2}\times\mathbb{U}_{2}\times\mathbb{B}_{2}}\right).
\end{array}
\end{equation}
As for the left-hand side, by using \eqref{FB2}, the choice of $\mathbb{U}_{2}$ in \ref{T2} leads to evaluate
\begin{equation}\label{DF21}
\setlength\arraycolsep{0.1em}\begin{array}{rcl}
(\hat\omega+\omega)G((\hat v&+&v)_{t}\circ s_{\hat\omega+\omega})-\hat\omega G(\hat v_{t}\circ s_{\hat\omega})-\hat\omega DG(\hat v_{t}\circ s_{\hat\omega})v_{t}\circ s_{\hat\omega}\\[2mm]
&&-\omega G(\hat v_{t}\circ s_{\hat\omega})+\omega DG(\hat v_{t}\circ s_{\hat\omega})\hat v_{t}'\circ s_{\hat\omega}\cdot s_{\hat\omega}\\[2mm]
&=&(\hat\omega+\omega)[G((\hat v+v)_{t}\circ s_{\hat\omega+\omega})-G(\hat v_{t}\circ s_{\hat\omega})]\\[2mm]
&&-\hat\omega DG(\hat v_{t}\circ s_{\hat\omega})v_{t}\circ s_{\hat\omega}+\omega DG(\hat v_{t}\circ s_{\hat\omega})\hat v_{t}'\circ s_{\hat\omega}\cdot s_{\hat\omega}
\end{array}
\end{equation}
for $t\in[0,1]$. \ref{T3} allows to write
\begin{equation}\label{DF22}
G((\hat v+v)_{t}\circ s_{\hat\omega+\omega})-G(\hat v_{t}\circ s_{\hat\omega})=DG(\hat v_{t}\circ s_{\hat\omega})\xi^{t}+o(\|\xi^{t}\|_{\mathtt{Y}})
\end{equation}
for $\xi^{t}:=(\hat v+v)_{t}\circ s_{\hat\omega+\omega}-\hat v_{t}\circ s_{\hat\omega}$, see, e.g., \cite[(ii) on page 10]{ampr95}. So we are led to consider $\xi^{t}(\sigma)$ for every $\sigma\in[-\tau,0]$ given the choice of $\mathtt{Y}$ in \ref{T1}. Then \eqref{states} gives
\begin{equation}\label{DF23}
\setlength\arraycolsep{0.1em}\begin{array}{rcl}
\xi^{t}(\sigma)&=&\hat v(t+s_{\hat\omega+\omega}(\sigma))-\hat v(t+s_{\hat\omega}(\sigma))+v(t+s_{\hat\omega+\omega}(\sigma))\\[2mm]
&=&\hat v'(t+s_{\hat\omega}(\sigma))\eta(\sigma)+o(|\eta(\sigma)|)+v(t+s_{\hat\omega+\omega}(\sigma))
\end{array}
\end{equation}
for $\eta(\sigma):=s_{\hat\omega+\omega}(\sigma)-s_{\hat\omega}(\sigma)$, where we applied Taylor's theorem with Peano's reminder to $\hat v$ thanks to the choice of $\mathbb{V}_{2}$ in \ref{T2}. Since
\begin{equation}\label{eta}
\eta(\sigma)=\frac{\sigma}{\hat\omega+\omega}-\frac{\sigma}{\hat\omega}=-s_{\hat\omega}(\sigma)\cdot\frac{\omega}{\hat\omega+\omega}>0
\end{equation}
follows from \eqref{somega}, substitution into \eqref{DF23} leads to
\begin{equation*}
\xi^{t}=-\hat v_{t}'\circ s_{\hat\omega}\cdot s_{\hat\omega}\cdot\frac{\omega}{\hat\omega+\omega}+v_{t}\circ s_{\hat\omega+\omega}+o(\omega)
\end{equation*}
with $\|\xi^{t}\|_{\mathtt{Y}}=O(\omega+\|v\|_{\mathbb{V}_{2}})$. Substitution first into \eqref{DF22} and then into \eqref{DF21} leads to
\begin{equation*}
\setlength\arraycolsep{0.1em}\begin{array}{rcl}
(\hat\omega+\omega)G((\hat v&+&v)_{t}\circ s_{\hat\omega+\omega})-\hat\omega G(\hat v_{t}\circ s_{\hat\omega})-\hat\omega DG(\hat v_{t}\circ s_{\hat\omega})v_{t}\circ s_{\hat\omega}\\[2mm]
&&-\omega G(\hat v_{t}\circ s_{\hat\omega})+\omega DG(\hat v_{t}\circ s_{\hat\omega})\hat v_{t}'\circ s_{\hat\omega}\cdot s_{\hat\omega}\\[2mm]
&=&\displaystyle(\hat\omega+\omega)DG(\hat v_{t}\circ s_{\hat\omega})\left(-\hat v_{t}'\circ s_{\hat\omega}\cdot s_{\hat\omega}\cdot\frac{\omega}{\hat\omega+\omega}+v_{t}\circ s_{\hat\omega}\right)\\[2mm]
&&+o(\omega+\|v\|_{\mathbb{V}_{2}})-\hat\omega DG(\hat v_{t}\circ s_{\hat\omega})v_{t}\circ s_{\hat\omega}+\omega DG(\hat v_{t}\circ s_{\hat\omega})\hat v_{t}'\circ s_{\hat\omega}\cdot s_{\hat\omega}\\[2mm]
&=&o(\omega+\|v\|_{\mathbb{V}_{2}})+O(\omega\cdot\|v\|_{\mathbb{V}_{2}}).
\end{array}
\end{equation*}
The thesis now follows since $\|(v,u,\omega)\|_{\mathbb{V}_{2}\times\mathbb{U}_{2}\times\mathbb{B}_{2}}=\max\{\|v\|_{\mathbb{V}_{2}},\|u\|_{\mathbb{U}_{2}},|\omega|\}$ holds by \eqref{normprod}.
\end{proof}
\noindent As far as formulation \eqref{bvp1} is concerned, one can try to follow the proof given above for \eqref{bvp2} to eventually realize that the key step is \eqref{DF23}, where the application of Taylor's theorem is subject to the differentiability of $\overline{v_{t}}$. The latter is not guaranteed due to \eqref{baryttheta}, recall the relevant comments at the end of Section \ref{s_notation}. Yet it is still possible to obtain the result under the similar hypothesis \ref{T2} since we consider derivatives with respect to time only from the right and $\eta(\sigma)$ is indeed positive in \eqref{eta}.

\bigskip
The second theoretical assumption in \cite{mas15NM}, viz. Assumption A$\mathfrak{G}$ (page 534), concerns the boundedness of the {\color{black}Green} operator $\mathcal{G}$ appearing in \eqref{Phi}. 
\begin{proposition}\label{p_G}
Under \ref{T2}, $\mathcal{G}_{2}$ defined in \eqref{G2} is bounded.
\end{proposition}
\begin{proof}
Following \eqref{normA}, we have that
\begin{equation*}
\setlength\arraycolsep{0.1em}\begin{array}{rcl}
\displaystyle\frac{\|\mathcal{G}_{2}(u,\psi)\|_{\mathbb{V}_{2}}}{\|(u,\psi)\|_{\mathbb{U}_{2}\times\mathbb{A}_{2}}}&=&\displaystyle\frac{\max\{\|\psi(0)+\int_0^{\cdot}u(s) \dd s\|_{\infty}+\|u\|_{\infty},\|\psi\|_{\mathbb{A}_{2}}\}}{\max\{\|u\|_{\mathbb{U}_{2}},\|\psi\|_{\mathbb{A}_{2}}\}}\\[4mm]
&\leq&\displaystyle\frac{\max\{\|\psi\|_{\infty}+\|u\|_{\infty}+\|u\|_{\infty},\|\psi\|_{\infty}\}}{\max\{\|u\|_{\infty},\|\psi\|_{\infty}\}},
\end{array}
\end{equation*}
holds for all nontrivial $(u,\psi)\in\mathbb{U}_{2}\times\mathbb{A}_{2}$. Then $\|\mathcal{G}_{2}\|_{\mathbb{V}_{2}\leftarrow\mathbb{U}_{2}\times\mathbb{A}_{2}}\leq 3$ easily follows.
\end{proof}
\noindent Note, however, that the PAF requires the range of $\mathcal{G}$ to lie in $\mathbb{V}$ for the fixed point problem to be well-posed: indeed, $\mathcal{G}$ provides the first argument to $\mathcal{F}$, recall \eqref{Phi}. In this respect, it is not difficult to see that $\mathcal{G}_{2}$ verifies this requirement under \ref{T2} and by considering that derivatives with respect to time are always from the right (otherwise there would be lack of differentiability at $0$). 

As far as formulation \eqref{bvp1} is concerned, under the similar hypothesis \ref{T2}, one similarly obtains
\begin{equation*}
\displaystyle\frac{\|\mathcal{G}_1(u,\alpha)\|_{\mathbb{V}_1}}{\|(u,\alpha)\|_{\mathbb{U}_1\times\mathbb{A}_1}}=\frac{\|\alpha+\int_0^{\cdot}u(s) \dd s\|_{\infty}+\|u\|_{\infty}}{\max\{\|u\|_{\mathbb{U}_1},\|\alpha\|_{\mathbb{A}_1}\}}\leq\frac{|\alpha|+\|u\|_{\infty}+\|u\|_{\infty}}{\max\{\|u\|_{\infty},|\alpha|\}},
\end{equation*}
for all nontrivial $(u,\alpha)\in\mathbb{U}_{1}\times\mathbb{A}_{1}$ and $\mathcal{G}_{1}$ in \eqref{G1}.

Since $\mathcal{G}$ is linear, it is also Fr\'echet-differentiable. Consequently, Proposition \ref{p_Afb} guarantees the Fr\'echet-differentiability of the fixed point operator \eqref{Phi2} as stated next.
\begin{corollary}\label{c_DPhi2}
Under \ref{T1}, \ref{T2} and \ref{T3}, $\Phi_{2}$ in \eqref{Phi2} is Fr\'echet-differentiable, from the right with respect to $\omega$, at every $(u,\psi,\omega)\in\mathbb{U}_{2}\times\mathbb{A}_{2}\times(0,+\infty)$ and
\begin{equation*}
D\Phi_{2}(\hat u,\hat\psi,\hat\omega)(u,\psi,\omega)=\begin{pmatrix}
\mathfrak{L}_{2}(\cdot;\mathcal{G}_{2}(\hat u,\hat\psi),\hat\omega)\mathcal{G}_{2}(u,\psi)_{\cdot}\circ s_{\hat\omega}+\omega\mathfrak{M}_{2}(\cdot;\mathcal{G}_{2}(\hat u,\hat\psi),\hat\omega)\\[2mm]
\mathcal{G}_{2}(u,\psi)_{1}\\[2mm]
\omega-p(\mathcal{G}_{2}(u,\psi)\vert_{[0,1]})
\end{pmatrix}
\end{equation*}
for $(u,\psi,\omega)\in\mathbb{U}_{2}\times\mathbb{A}_{2}\times(0,+\infty)$, $\mathfrak{L}_{2}$ in \eqref{L2} and $\mathfrak{M}_{2}$ in \eqref{M2}.
\end{corollary}
\begin{proof}
The only nonlinear component of $\Phi_{2}$ in \eqref{Phi2} is the first one, i.e., the one in $\mathbb{U}_{2}$ given by $\mathcal{F}_{2}$ in the first of \eqref{FB2}. The result is thus provided directly by Proposition \ref{p_Afb}.
\end{proof}
\noindent It is not difficult to argue that the same result holds also for formulation \eqref{bvp1}, given that the range of $\mathcal{G}_{1}$ is in $\mathbb{V}_{1}$ if we let $\mathbb{U}_{1}=B^{\infty}([0,1],\mathbb{R}^{d})$ and $\mathbb{V}_{1}=B^{1,\infty}([0,1],\mathbb{R}^{d})$ similarly to \ref{T2}.

\bigskip
The third theoretical assumption in \cite{mas15NM}, viz. Assumption A$x^{\ast}1$ (page 536), concerns the local Lipschitz continuity of the Fr\'echet derivative of the fixed point operator at the relevant fixed points. In this respect, let $(u^{\ast},\psi^{\ast},\omega^{\ast})\in\mathbb{U}_{2}\times\mathbb{A}_{2}\times\mathbb{B}_{2}$ be a fixed point of $\Phi_{2}$ in \eqref{Phi2} and let $y^{\ast}$ be the corresponding $1$-periodic solution of \eqref{rfdes}. Recall that $\omega^{\ast}$ is meant to be positive.
\begin{proposition}\label{p_Ax*1}
Under \ref{T1}, \ref{T2}, \ref{T3} and \ref{T5}, there exist $r_{2}\in(0,\omega^{\ast})$ and $\kappa_{2}\geq0$ such that
\begin{equation*}
\setlength\arraycolsep{0.1em}\begin{array}{rcl}
\|D\Phi_{2}(u,\psi,\omega)&-&D\Phi_{2}(u^{\ast},\psi^{\ast},\omega^{\ast})\|_{\mathbb{U}_{2}\times\mathbb{A}_{2}\times\mathbb{B}_{2}\leftarrow\mathbb{U}_{2}\times\mathbb{A}_{2}\times(0,+\infty)}\\[2mm]
&\leq&\kappa_{2}\|(u,\psi,\omega)-(u^{\ast},\psi^{\ast},\omega^{\ast})\|_{\mathbb{U}_{2}\times\mathbb{A}_{2}\times\mathbb{B}_{2}}
\end{array}
\end{equation*}
for all $(u,\psi,\omega)\in\overline{B}((u^{\ast},\psi^{\ast},\omega^{\ast}),r_{2})$.
\end{proposition} 
\begin{proof}
In this proof we set for brevity $v:=\mathcal{G}_{2}(u,\psi)$, $v^{\ast}:=\mathcal{G}_{2}(u^{\ast},\psi^{\ast})$ and $\bar v:=\mathcal{G}_{2}(\bar u,\bar\psi)$.

Following \eqref{normA}, we prove that there exist $r_{2}>0$ and $\kappa_{2}\geq0$ such that
\begin{equation*}
\setlength\arraycolsep{0.1em}\begin{array}{rcl}
\|D\Phi_{2}(u,\psi,\omega)(\bar u,\bar\psi,\bar\omega)&-&D\Phi_{2}(u^{\ast},\psi^{\ast},\omega^{\ast})(\bar u,\bar\psi,\bar\omega)\|_{\mathbb{U}_{2}\times\mathbb{A}_{2}\times\mathbb{B}_{2}}\\[2mm]
&\leq&\kappa_{2}\|(\bar u, \bar\psi,\bar\omega)\|_{\mathbb{U}_{2}\times\mathbb{A}_{2}\times\mathbb{B}_{2}}\cdot\|(u,\psi,\omega)-(u^{\ast},\psi^{\ast},\omega^{\ast})\|_{\mathbb{U}_{2}\times\mathbb{A}_{2}\times\mathbb{B}_{2}}
\end{array}
\end{equation*}
for all $(u,\psi,\omega)\in\overline{B}((u^{\ast},\psi^{\ast},\omega^{\ast}),r_{2})$ and all $(\bar u, \bar\psi,\bar\omega)\in\mathbb{U}_{2}\times\mathbb{A}_{2}\times(0,+\infty)$. From Corollary \ref{c_DPhi2} it is clear that, given the linearity of both $\mathcal{G}_{2}$ and $p$, we need to monitor only the first component of $D\Phi_{2}$, i.e., the one in $\mathbb{U}_{2}$. Then, by defining
\begin{equation}\label{P}
P(t):=\omega DG(v_{t}\circ s_{\omega})\bar v_{t}\circ s_{\omega}-\omega^{\ast}DG(v^{\ast}_{t}\circ s_{\omega^{\ast}})\bar v_{t}\circ s_{\omega^{\ast}},
\end{equation}
\begin{equation}\label{Q}
Q(t):=\bar\omega[G(v_{t}\circ s_{\omega})-G(v^{\ast}_{t}\circ s_{\omega^{\ast}})]
\end{equation}
and
\begin{equation}\label{R}
R(t):=-\bar\omega[DG(v_{t}\circ s_{\omega})v_{t}'\circ s_{\omega}\cdot s_{\omega}-DG(v^{\ast}_{t}\circ s_{\omega^{\ast}}){v^{\ast}_{t}}'\circ s_{\omega^{\ast}}\cdot s_{\omega^{\ast}}]
\end{equation}
through \eqref{L2} and \eqref{M2}, we are led to bound $|P(t)+Q(t)+R(t)|$ for all $t\in[0,1]$ given the choice of $\mathbb{U}_{2}$ in \ref{T2}.

Let us start with \eqref{P}, which we schematically rewrite as
\begin{equation}\label{ABC}
\setlength\arraycolsep{0.1em}\begin{array}{rcl}
P(t)&=&(A_{1}+A_{2})(B_{1}+B_{2})(C_{1}+C_{2})-A_{2}B_{2}C_{2}\\[2mm]
&=&A_{1}B_{1}C_{1}+A_{1}B_{1}C_{2}+A_{1}B_{2}C_{1}+A_{1}B_{2}C_{2}+A_{2}B_{1}C_{1}+A_{2}B_{1}C_{2}+A_{2}B_{2}C_{1}
\end{array}
\end{equation}
for
\begin{equation*}
\setlength\arraycolsep{0.1em}\begin{array}{ll}
A_{1}:=\omega-\omega^{\ast},&\quad A_{2}:=\omega^{\ast},\\[2mm]
B_{1}:=DG(v_{t}\circ s_{\omega})-DG(v^{\ast}_{t}\circ s_{\omega^{\ast}}),&\quad B_{2}:=DG(v^{\ast}_{t}\circ s_{\omega^{\ast}}),\\[2mm]
C_{1}:=\bar v_{t}\circ s_{\omega}-\bar v_{t}\circ s_{\omega^{\ast}},&\quad C_{2}:=\bar v_{t}\circ s_{\omega^{\ast}}.
\end{array}
\end{equation*}
The plan is to bound every single term $A_{i}$, $B_{i}$ and $C_{i}$, $i=1,2$, to eventually get the desired bound for $P$. Then we proceed similarly for $R$ in \eqref{R}, while we proceed directly for $Q$ in \eqref{Q}. In doing this, quantities with the apex $^{\ast}$ are considered constant since related to the fixed point. Clearly $|A_{1}|\leq\|(u,\psi,\omega)-(u^{\ast},\psi^{\ast},\omega^{\ast})\|_{\mathbb{U}_{2}\times\mathbb{A}_{2}\times\mathbb{B}_{2}}$ follows from \eqref{normprod}, while $|A_{2}|=\omega^{\ast}$. As for $B_{1}$, under \ref{T5} we have
\begin{equation*}
\|B_{1}\|_{\mathbb{R}^{d}\leftarrow\mathtt{Y}}\leq\kappa\|v_{t}\circ s_{\omega}-v^{\ast}_{t}\circ s_{\omega^{\ast}}\|_{\mathtt{Y}}\leq\kappa\|v_{t}\circ s_{\omega}-v^{\ast}_{t}\circ s_{\omega}\|_{\mathtt{Y}}+\kappa\|v^{\ast}_{t}\circ s_{\omega}-v^{\ast}_{t}\circ s_{\omega^{\ast}}\|_{\mathtt{Y}}.
\end{equation*}
As for the first addend in the right-hand side above,
\begin{equation}\label{r/2}
\|v_{t}\circ s_{\omega}-v^{\ast}_{t}\circ s_{\omega}\|_{\mathtt{Y}}\leq2\|(u,\psi,\omega)-(u^{\ast},\psi^{\ast},\omega^{\ast})\|_{\mathbb{U}_{2}\times\mathbb{A}_{2}\times\mathbb{B}_{2}}
\end{equation}
follows directly from \eqref{G2} and \eqref{normprod} again. As for the second addend, by Lemma \ref{l_v*} we have $\|v^{\ast}_{t}\circ s_{\omega}-v^{\ast}_{t}\circ s_{\omega^{\ast}}\|_{\mathtt{Y}}\leq\|(u^{\ast},\psi^{\ast},\omega^{\ast})\|_{\mathbb{U}_{2}\times\mathbb{A}_{2}\times\mathbb{B}_{2}}\cdot\|s_{\omega}-s_{\omega^{\ast}}\|_{\infty}$.
Since
\begin{equation}\label{slip}
|s_{\omega}(\sigma)-s_{\omega^{\ast}}(\sigma)|=\left|\frac{\sigma}{\omega}-\frac{\sigma}{\omega^{\ast}}\right|\leq\frac{\tau|\omega-\omega^{\ast}|}{\omega^{\ast}\omega}
\end{equation}
holds for every $\sigma\in[-\tau,0]$, we finally get
\begin{equation*}
\|v^{\ast}_{t}\circ s_{\omega}-v^{\ast}_{t}\circ s_{\omega^{\ast}}\|_{\mathtt{Y}}\leq\frac{\tau\|(u^{\ast},\psi^{\ast},\omega^{\ast})\|_{\mathbb{U}_{2}\times\mathbb{A}_{2}\times\mathbb{B}_{2}}}{\omega^{\ast}(\omega^{\ast}-r)}\cdot\|(u,\psi,\omega)-(u^{\ast},\psi^{\ast},\omega^{\ast})\|_{\mathbb{U}_{2}\times\mathbb{A}_{2}\times\mathbb{B}_{2}}
\end{equation*}
for every $\omega\in\overline B(\omega^{\ast},r)$ and $r$ in \ref{T5}. Eventually,
\begin{equation*}
\|B_{1}\|_{\mathbb{R}^{d}\leftarrow\mathtt{Y}}\leq\kappa\left(2+\frac{\tau\|(u^{\ast},\psi^{\ast},\omega^{\ast})\|_{\mathbb{U}_{2}\times\mathbb{A}_{2}\times\mathbb{B}_{2}}}{\omega^{\ast}(\omega^{\ast}-r)}\right)\cdot\|(u,\psi,\omega)-(u^{\ast},\psi^{\ast},\omega^{\ast})\|_{\mathbb{U}_{2}\times\mathbb{A}_{2}\times\mathbb{B}_{2}}.
\end{equation*}
As for $B_{2}$, we write directly
\begin{equation}\label{K21}
\|B_{2}\|_{\mathbb{R}^{d}\leftarrow\mathtt{Y}}\leq\kappa_{2,1}:=\max_{t\in[0,1]}\|DG(v^{\ast}_{t}\circ s_{\omega^{\ast}})\|_{\mathbb{R}^{d}\leftarrow\mathtt{Y}},
\end{equation}
which makes sense since the map $t\mapsto v^{\ast}_{t}$ is uniformly continuous and so is $DG$ at the state corresponding to the fixed point under \ref{T5}.

The same arguments used above lead also to
\begin{equation*}
\|C_{1}\|_{\mathtt{Y}}\leq\frac{\tau\|(\bar u, \bar\psi,\bar\omega)\|_{\mathbb{U}_{2}\times\mathbb{A}_{2}\times\mathbb{B}_{2}}}{\omega^{\ast}(\omega^{\ast}-r)}\cdot\|(u,\psi,\omega)-(u^{\ast},\psi^{\ast},\omega^{\ast})\|_{\mathbb{U}_{2}\times\mathbb{A}_{2}\times\mathbb{B}_{2}}
\end{equation*}
and $\|C_{2}\|_{\mathtt{Y}}\leq2\|(\bar u, \bar\psi,\bar\omega)\|_{\mathbb{U}_{2}\times\mathbb{A}_{2}\times\mathbb{B}_{2}}$.

Eventually, we see that every triple $A_{i}B_{j}C_{k}$ in $P(t)$ in the last member of \eqref{ABC} contains a factor of index $1$, which is always bounded by some constant times $\|(u,\psi,\omega)-(u^{\ast},\psi^{\ast},\omega^{\ast})\|_{\mathbb{U}_{2}\times\mathbb{A}_{2}\times\mathbb{B}_{2}}$, as well as a $C$-term, whose bounds always contain $\|(\bar u, \bar\psi,\bar\omega)\|_{\mathbb{U}_{2}\times\mathbb{A}_{2}\times\mathbb{B}_{2}}$. Therefore, there exist $r_{2,P}\in(0,\omega^{\ast})$ and $\kappa_{2,P}\geq0$ such that
\begin{equation*}
\|P\|_{\mathbb{U}_{2}}\leq\kappa_{2,P}\|(\bar u, \bar\psi,\bar\omega)\|_{\mathbb{U}_{2}\times\mathbb{A}_{2}\times\mathbb{B}_{2}}\cdot\|(u,\psi,\omega)-(u^{\ast},\psi^{\ast},\omega^{\ast})\|_{\mathbb{U}_{2}\times\mathbb{A}_{2}\times\mathbb{B}_{2}}
\end{equation*}
for all $(u,\psi,\omega)\in\overline{B}((u^{\ast},\psi^{\ast},\omega^{\ast}),r_{2,P})$. Actually, it is enough to choose $r_{2,P}=r/2$ for $r$ in \ref{T5} by virtue of \eqref{r/2}, while the constant $\kappa_{2,P}$ can be recovered from the analysis above, though with some technical efforts.

The term $Q$ in \eqref{Q} can be bounded as similarly done above for $B_{1}$ once realized that one can apply the mean value theorem in a neighborhood of $v^{\ast}_{t}\circ s_{\omega^{\ast}}$ under \ref{T5}. Then, we conclude that there exist $r_{2,Q}\in(0,\omega^{\ast})$ and $\kappa_{2,Q}\geq0$ such that
\begin{equation*}
\|Q\|_{\mathbb{U}_{2}}\leq\kappa_{2,Q}\|(\bar u, \bar\psi,\bar\omega)\|_{\mathbb{U}_{2}\times\mathbb{A}_{2}\times\mathbb{B}_{2}}\cdot\|(u,\psi,\omega)-(u^{\ast},\psi^{\ast},\omega^{\ast})\|_{\mathbb{U}_{2}\times\mathbb{A}_{2}\times\mathbb{B}_{2}}
\end{equation*}
for all $(u,\psi,\omega)\in\overline{B}((u^{\ast},\psi^{\ast},\omega^{\ast}),r_{2,Q})$. Note that the factor $\|(\bar u, \bar\psi,\bar\omega)\|_{\mathbb{U}_{2}\times\mathbb{A}_{2}\times\mathbb{B}_{2}}$ here comes directly from the factor $\bar\omega$ in \eqref{Q}. Moreover we can take $r_{2,Q}=r/2$ again.

The term $R$ in \eqref{R} is treated as done for $P$ above, hence we first write
\begin{equation}\label{BDE}
\setlength\arraycolsep{0.1em}\begin{array}{rcl}
R(t)&=&-\bar\omega[(B_{1}+B_{2})(D_{1}+D_{2})(E_{1}+E_{2})-B_{2}D_{2}E_{2}]\\[2mm]
&=&-\bar\omega[B_{1}D_{1}E_{1}+B_{1}D_{1}E_{2}+B_{1}D_{2}E_{1}+B_{1}D_{2}E_{2}\\[2mm]
&&+B_{2}D_{1}E_{1}+B_{2}D_{1}E_{2}+B_{2}D_{2}E_{1}]
\end{array}
\end{equation}
for the same $B_{1}$ and $B_{2}$ above plus
\begin{equation*}
\setlength\arraycolsep{0.1em}\begin{array}{ll}
D_{1}:=v_{t}'\circ s_{\omega}-{v^{\ast}_{t}}'\circ s_{\omega^{\ast}},&\quad D_{2}:={v^{\ast}_{t}}'\circ s_{\omega^{\ast}},\\[2mm]
E_{1}:=s_{\omega}-s_{\omega^{\ast}},&\quad E_{2}:=s_{\omega^{\ast}}.
\end{array}
\end{equation*}

As for $D_{1}$ we have
\begin{equation*}
\setlength\arraycolsep{0.1em}\begin{array}{rcl}
\|D_{1}\|_{\mathtt{Y}}&\leq&\|v_{t}'\circ s_{\omega}-{v^{\ast}_{t}}'\circ s_{\omega}\|
_{\mathtt{Y}}+\|{v^{\ast}_{t}}'\circ s_{\omega}-{v^{\ast}_{t}}'\circ s_{\omega^{\ast}}\|_{\mathtt{Y}}\\[2mm]
&\leq&2\|(u,\psi,\omega)-(u^{\ast},\psi^{\ast},\omega^{\ast})\|_{\mathbb{U}_{2}\times\mathbb{A}_{2}\times\mathbb{B}_{2}}\\[2mm]
&&+\omega^{\ast}\kappa_{2,1}\|(u^*,\psi^*,\omega^*)\|_{\mathbb{U}_{2}\times\mathbb{A}_{2}\times\mathbb{B}_{2}}\cdot\|s_{\omega}-s_{\omega^{\ast}}\|_{\infty},
\end{array}
\end{equation*}
where the first addend in the right-hand side above follows from the definition of $\mathcal{G}$ in \eqref{G2} and the second addend follows from Lemma \ref{l_v*'}. Finally, by using \eqref{slip} we conclude that
\begin{equation*}
\setlength\arraycolsep{0.1em}\begin{array}{rcl}
\|D_{1}\|_{\mathtt{Y}}&\leq&\displaystyle\left(2+\frac{\tau\omega^{\ast}\kappa_{2,1}\|(u^*,\psi^*,\omega^*)\|_{\mathbb{U}_{2}\times\mathbb{A}_{2}\times\mathbb{B}_{2}}}{\omega^{\ast}(\omega^{\ast}-r)}\right)\cdot\|(u,\psi,\omega)-(u^{\ast},\psi^{\ast},\omega^{\ast})\|_{\mathbb{U}_{2}\times\mathbb{A}_{2}\times\mathbb{B}_{2}}.
\end{array}
\end{equation*}
As for $D_{2}$, instead, we directly get $\|D_{2}\|_{\mathtt{Y}}\leq2\|(u^{\ast},\psi^{\ast},\omega^{\ast})\|_{\mathbb{U}_{2}\times\mathbb{A}_{2}\times\mathbb{B}_{2}}$. Since
\begin{equation*}
\|E_{1}\|_{\infty}\leq\frac{\tau}{\omega^{\ast}(\omega^{\ast}-r)}\cdot\|(u,\psi,\omega)-(u^{\ast},\psi^{\ast},\omega^{\ast})\|_{\mathbb{U}_{2}\times\mathbb{A}_{2}\times\mathbb{B}_{2}}
\end{equation*}
and $\|E_{2}\|_{\infty}\leq\tau/\omega^{\ast}$, we eventually see that every triple $B_{i}D_{j}E_{k}$ in the last member of \eqref{BDE} contains a factor of index $1$, which are always bounded by some constant times $\|(u,\psi,\omega)-(u^{\ast},\psi^{\ast},\omega^{\ast})\|_{\mathbb{U}_{2}\times\mathbb{A}_{2}\times\mathbb{B}_{2}}$. Taking into account of the multiplying factor $\bar\omega$ in \eqref{BDE}, we conclude that there exist $r_{2,R}\in(0,\omega^{\ast})$ and $\kappa_{2,R}\geq0$ such that
\begin{equation*}
\|R\|_{\mathbb{U}_{2}}\leq\kappa_{2,R}\|(\bar u, \bar\psi,\bar\omega)\|_{\mathbb{U}_{2}\times\mathbb{A}_{2}\times\mathbb{B}_{2}}\cdot\|(u,\psi,\omega)-(u^{\ast},\psi^{\ast},\omega^{\ast})\|_{\mathbb{U}_{2}\times\mathbb{A}_{2}\times\mathbb{B}_{2}}
\end{equation*}
for all $(u,\psi,\omega)\in\overline{B}((u^{\ast},\psi^{\ast},\omega^{\ast}),r_{2,R})$ , with $r_{2,R}=r/2$ again.

The thesis eventually follows by choosing $r_{2}=r/2$ and $\kappa_{2}=\kappa_{2,P}+\kappa_{2,Q}+\kappa_{2,R}$.
\end{proof}
\begin{remark}\label{r_omega0}
Observe that the Lipschitz constant $\kappa_{2}$ grows unbounded as $\omega^{\ast}\rightarrow0$ due to the presence of the latter at the denominator of several of its terms.
\end{remark}
\noindent As far as formulation \eqref{bvp1} is concerned, it is not difficult to realize that the above proof would fail because of the analogous term $C_{1}$, i.e.,
\begin{equation*}
\overline{\mathcal{G}_{1}(\bar u,\bar\alpha)_{t}}\circ s_{\omega}-\overline{\mathcal{G}_{1}(\bar u,\bar\alpha)_{t}}\circ s_{\omega^{\ast}}.
\end{equation*}
Indeed, as already observed, the function $\overline{\mathcal{G}_{1}(\bar u,\bar\alpha)_{t}}$ is always discontinuous at $\theta=-t$, preventing the achievement of the necessary Lipschitz condition. Alternatively, a possible remedy is that of restricting to the spaces
\begin{equation}\label{A1U}
\mathbb{U}_{1}=B^{\infty}_{\pi}([0,1],\mathbb{R}^{d}):=\left\{u\in B^{\infty}([0,1],\mathbb{R}^{d})\ :\ \int_{0}^{1}u(s)\dd s=0\right\}
\end{equation}
and
\begin{equation}\label{A1V}
\mathbb{V}_{1}=B_{\pi}^{1,\infty}([0,1],\mathbb{R}^{d}):=\{v\in B^{1,\infty}([0,1],\mathbb{R}^{d})\ :\ v(0)=v(1)\}.
\end{equation}
These choices guarantee not only that $\overline{\mathcal{G}_{1}(\bar u,\bar\alpha)_{t}}$ is continuous, but also Lipschitz continuous thanks to the constraint of zero mean imposed to the derivative $u$. Note, however, that the same constraint gives $v(1)=v(0)=\alpha$ for $v=\mathcal{G}_{1}(u,\alpha)$ according to \eqref{G1}. The latter fact impedes to satisfy the next theoretical assumption as it will be evident later on.

As a final comment regarding this assumption, we note that it is not directly used in this work, even though its validity is required in Section \ref{s_validation_n} for a suitable approximation $G_{M}$ in place of $G$. Since the proof is unchanged, we prefer to give it here in full detail so as to follow the presentation in \cite{mas15NM}. Observe anyway that the comment given above about the failure of formulation \eqref{bvp1} holds unaltered since the mentioned critical step is independent of $G$ or $G_{M}$.

\bigskip
The fourth (and last) theoretical assumption in \cite{mas15NM}, viz. Assumption A$x^{\ast}2$ (page 536), concerns the well-posedness of a linear(ized) inhomogeneous version of the PAF \eqref{PAF}. Its validity can be proved under \ref{T1} and \ref{T2} again, together with \ref{T4} and an additional requirement, which is, for instance, a consequence of the {\it hyperbolicity} of the periodic solution at hands. Let us remark that the latter is a standard assumption in the context of application of the principle of linearized stability (see, e.g., \cite[Chapter XIV]{diekmann95}) or \cite[Chapter 10]{hale77}), in which one derives information on the stability of the concerned periodic solution by investigating the stability of the zero solution of \eqref{rfdet} linearized around the periodic solution itself. Let us observe that, on the one hand, stability analysis is amongst the main motivations supporting the computation of periodic solutions. On the other hand, the linearization of \eqref{rfdes} around the $\omega^{\ast}$-periodic solution $\mathtt{y}^{\ast}$ leads to considering the linear homogeneous RFDE
\begin{equation}\label{Lrfdet2}
y'(t)=\mathfrak{L}_{2}(t;v^{\ast},\omega^{\ast})y_{t}\circ s_{\omega^{\ast}}
\end{equation}
for $\mathfrak{L}_{2}$ in \eqref{L2}. Under \ref{T4} the associated initial value problem is well-posed and we denote by $T_{2}^{\ast}(t,s):Y\rightarrow Y$ the relevant (forward) evolution operator for $s\in\mathbb{R}$ and $t\geq s$. Let us note that $T_{2}^{\ast}(1,0)$ represents the corresponding {\it monodromy operator}, i.e., the operator advancing the state solution of one period. Then hyperbolicity implies the required additional hypothesis of $1$ being a simple Floquet multiplier, i.e., a simple eigenvalue of $T_{2}^{\ast}(1,0)$, besides having no other Floquet multipliers on the unit circle.

\begin{remark}\label{r_mu=1}
$1$ is always a Floquet multiplier due to linearization. Indeed, as a general fact the derivative of a solution of a nonlinear problem is always a solution of the problem obtained by linearizing around this solution. Consequently, if the latter is periodic, the linearized problem has a periodic solution.
\end{remark}

In the following we refer to Proposition \ref{p_Ax*1} also for the relevant notation. It is also convenient to introduce the abbreviations
\begin{equation}\label{L*M*}
\mathfrak{L}_{2}^{\ast}:=\mathfrak{L}_{2}(\cdot;v^{\ast},\omega^{\ast}),\qquad\mathfrak{M}_{2}^{\ast}:=\mathfrak{M}_{2}(\cdot;v^{\ast},\omega^{\ast}).
\end{equation}
Let us anticipate that the proof is not very difficult, if not for showing that a nongeneric case ($k_{1}=0$ in the proof below) is ruled out since it leads to a contradiction. Although this case is seemingly innocuous, the proof that it cannot hold is not as immediate. For this reason, we leave the treatment of this special case to the next section, giving the fact as granted in the current one. This way the deserved room can be devoted to its analysis without interrupting the main reading flow. Eventually, let us remark that this ``secondary'' proof represents a main contribution of the present work. Below $\sigma(A)$ denotes the spectrum of an operator $A$.
\begin{proposition}\label{p_Ax*2}
Under \ref{T1}, \ref{T2} and \ref{T4}, if $1\in\sigma(T_{2}^{\ast}(1,0))$ is simple, then the linear bounded operator $I_{\mathbb{U}_{2}\times\mathbb{A}_{2}\times\mathbb{B}_{2}}-D\Phi_{2}(u^{\ast},\psi^{\ast},\omega^{\ast})$ is invertible, i.e., for all $(u_{0},\psi_{0},\omega_{0})\in\mathbb{U}_{2}\times\mathbb{A}_{2}\times\mathbb{B}_{2}$ there exists a unique $(u,\psi,\omega)\in\mathbb{U}_{2}\times\mathbb{A}_{2}\times\mathbb{B}_{2}$ such that
\begin{equation}\label{Ax*20}
\left\{\setlength\arraycolsep{0.1em}\begin{array}{l}
u=\mathfrak{L}_{2}^{\ast}\mathcal{G}_{2}(u,\psi)_{\cdot}\circ s_{\omega^{\ast}}+\omega\mathfrak{M}_{2}^{\ast}+u_{0}\\[2mm]
\psi=\mathcal{G}_{2}(u,\psi)_{1}+\psi_{0}\\[2mm]
p(\mathcal{G}_{2}(u,\psi)\vert_{[0,1]})=\omega_{0}.
\end{array}
\right.
\end{equation}
\end{proposition} 
\begin{proof}
The proof is based on treating \eqref{Ax*20} as an initial value problem for $v=\mathcal{G}_{2}(u,\psi)$, i.e.,
\begin{equation}\label{Ax*211}
\left\{\setlength\arraycolsep{0.1em}\begin{array}{l}
v'(t)=\mathfrak{L}_{2}^{\ast}(t)v_{t}\circ s_{\omega^{\ast}}+\omega\mathfrak{M}_{2}^{\ast}(t)+u_{0}(t)\\[2mm]
v_{0}=\psi
\end{array}
\right.
\end{equation}
for $t\in[0,1]$, imposing then the boundary conditions in \eqref{Ax*20}. Because the RFDE in \eqref{Ax*211} is linear inhomogeneous with continuous linear part under \ref{T4}, for every $\psi\in\mathbb{A}_{2}$ there exists a unique solution $v$ whose state can be expressed through the variation of constants formula
\begin{equation*}
v_{t}=T_{2}^{\ast}(t,0)\psi+\int_{0}^{t}[T_{2}^{\ast}(t,s)X_{0}][\omega\mathfrak{M}_{2}^{\ast}(s)+u_{0}(s)]\dd s,\quad t\in[0,1],
\end{equation*}
where
\begin{equation*}
X_0(\theta):=\begin{cases}
0,&\theta\in[-1,0),\\
I_d,&\theta=0,
\end{cases}
\end{equation*}
see \cite[Section 6.2]{hale77}. The first boundary condition in \eqref{Ax*20} gives then
\begin{equation}\label{Ax*21}
\psi=T_{2}^{\ast}(1,0)\psi+\int_{0}^{1}[T_{2}^{\ast}(1,s)X_{0}][\omega\mathfrak{M}_{2}^{\ast}(s)+u_{0}(s)]\dd s+\psi_{0}.
\end{equation}
Let now $R$ and $K$ be, respectively, the range and the kernel of $I_{Y}-T_{2}^{\ast}(1,0)$. Then (see, e.g., \cite[Section 8.2]{hale77})
\begin{equation}\label{deco}
Y=R\oplus K
\end{equation}
and, by the hypothesis on the multiplier $1$, we can set $K=\Span\{\varphi\}$ for $\varphi$ an eigenfunction of the multiplier $1$ itself. Moreover, let us assume $p(v(\cdot;\varphi)|_{[0,1]})\neq0$ (see Remark \ref{r_pfi0} below), where $v(\cdot;\varphi)$ denotes the solution of \eqref{Ax*211} exiting from $\varphi$.

From \eqref{Ax*21} let us define the elements of $Y$
\begin{equation*}
\xi_{1}^{\ast}:=\int_0^{1}[T_{2}^{\ast}(1,s)X_{0}]\mathfrak{M}_{2}^{\ast}(s)\dd s,\qquad\xi_{2}^{\ast}:=\int_0^{1}[T_{2}^{\ast}(1,s)X_{0}]u_{0}(s)\dd s+\psi_{0},
\end{equation*}
so that \eqref{Ax*21} becomes
\begin{equation}\label{Ax*22}
[I_{Y}-T_{2}^{\ast}(1,0)]\psi=\omega\xi_{1}^{\ast}+\xi_{2}^{\ast}.
\end{equation}
Note that $\psi_{0}\in Y$ since $\mathbb{A}_{2}\subseteq Y$. From \eqref{deco} it follows that $\xi_{1}^{\ast}$ can be written uniquely as
\begin{equation}\label{xi1}
\xi_{1}^{\ast}=r_{1}+k_{1}\varphi,
\end{equation}
where $r_{1}\in R$ and $k_{1}\in\mathbb{R}$. Similarly, $\xi_{2}^{\ast}=r_{2}+k_{2}\varphi$. Then from \eqref{Ax*22} it must be $\omega\xi_{1}^{\ast}+\xi_{2}^{\ast}\in R$, which implies $\omega k_{1}+k_{2}=0$. Therefore, by assuming $k_{1}\neq 0$, it follows that
\begin{equation}\label{omega}
\omega=-k_{2}/k_{1}
\end{equation}
is the only possible solution. As anticipated, we show in the forthcoming section that it cannot be otherwise, since $k_{1}=0$ leads to a contradiction.

Eventually, let $\eta$ be such that $\omega\xi_{1}^{\ast}+\xi_{2}^{\ast}=\eta-T_{2}^{\ast}(1,0)\eta$. Then every $\psi$ satisfying \eqref{Ax*22} can be written as $\eta+\lambda\varphi$ for some $\lambda\in\mathbb{R}$. The value of $\lambda$ is fixed by imposing the second boundary condition in \eqref{Ax*20}, i.e., $p(v(\cdot;\eta)|_{[0,1]})+\lambda p(v(\cdot;\varphi)|_{[0,1]})=\omega_{0}$. Uniqueness follows from $p(v(\cdot;\varphi)|_{[0,1]})\neq0$.
\end{proof}
\begin{remark}\label{r_pfi0}
The condition $p(v(\cdot;\varphi)|_{[0,1]})\neq0$ is generic and not restrictive at all. In any case, it is always possible to change $p$ in order to meet the above requirement.
\end{remark}
\noindent As for \eqref{bvp1}, it is immediate to verify that a similar result cannot be obtained under the choices \eqref{A1U} and \eqref{A1V}. Indeed, for every $\alpha_{0}\in\mathbb{A}_{1}$ one should find a unique $\alpha\in\mathbb{A}_{1}$ satisfying $\alpha=\mathcal{G}_{1}(u,\alpha)(1)+\alpha_{0}$, but since $u\in\mathbb{U}_{1}$ implies that $u$ has zero mean, it must necessarily be $\mathcal{G}_{1}(u,\alpha)(1)=\alpha$, i.e., $\alpha_{0}=0$. This fact definitively proves the failure of the approach proposed in \cite{mas15NM} with respect to formulation \eqref{bvp1}, so that from now on we focus exclusively on formulation \eqref{bvp2}. In particular, in the remainder of the work we drop the use of the index $2$ to lighten the notation.
\begin{remark}\label{r_neutral}
Concerning neutral problems as considered in \cite{mas15NM}, in the periodic case the Fr\'echet-differentiability with respect to $\omega$ would require differentiability in $\mathbb{U}$ and, consequently, a larger norm. This in turn would hinder the Fr\'echet-differ\-entiability itself, recall \eqref{DF20}. We go back to this observation in Section \ref{s_concluding} in view of future extensions.
\end{remark}
\subsection{The nongeneric case}
\label{s_nongeneric}
The following content completes the proof of Pro\-position \ref{p_Ax*2} by showing that $k_{1}\neq0$ must necessarily hold in \eqref{xi1}. We proceed by contradiction assuming that $k_{1}=0$, which is equivalent to $\xi_{1}^{\ast}\in R$. Therefore, there exists $\gamma\in Y$ such that $[I_{Y}-T_{2}^{\ast}(1,0)]\gamma=\omega\xi_{1}^{\ast}$, which amounts to say that
\begin{equation}\label{yL2M2}
y'(t)=\mathfrak{L}_{2}^{\ast}(t)y_{t}\circ s_{\omega^{\ast}}+\omega\mathfrak{M}_{2}^{\ast}(t)
\end{equation}
has a $1$-periodic solution (with $y_{0}=\gamma$). The proof given next that this is not possible requires several tools.

\bigskip
As a first step, let us express the action of $\mathfrak{L}_{2}^{\ast}$ through the Riemann-Stieltjes integral
\begin{equation}\label{L2*}
\mathfrak{L}_{2}^{\ast}(t)\psi\circ s_{\omega^{\ast}}=\int_{-\tau}^{0}\dd_{\sigma}\mathtt{n}^{\ast}(t,\sigma)\psi(s_{\omega^{\ast}}(\sigma)),\quad\psi\in Y,
\end{equation}
where, for $t\in\mathbb{R}$, $\mathtt{n}^{\ast}(t,\cdot):[-\tau,0]\rightarrow\mathbb{R}^{d\times d}$ is of normalized bounded variation and, for every $\sigma\in[-\tau,0]$, $\mathtt{n}^{\ast}(\cdot,\sigma):\mathbb{R}\rightarrow\mathbb{R}^{d\times d}$ is $1$-periodic, see, e.g., \cite[Section 6.1]{hale77}. Then, by \eqref{M2}, $\mathfrak{M}_{2}^{\ast}$ in \eqref{yL2M2} becomes
\begin{equation}\label{M2*}
\setlength\arraycolsep{0.1em}\begin{array}{rcl}
\mathfrak{M}_{2}^{\ast}(t)&=&\displaystyle G(v^{\ast}_{t}\circ s_{\omega^{\ast}})-\int_{-\tau}^{0}\dd_{\sigma}\mathtt{n}^{\ast}(t,\sigma){v^{\ast}}'(t+s_{\omega^{\ast}}(\sigma))\cdot\frac{s_{\omega^{\ast}}(\sigma)}{\omega^{\ast}}\\[3mm]
&=&\displaystyle\frac{1}{\omega^{\ast}}\left({v^{\ast}}'(t)-\int_{-r}^{0}\dd_{\theta}n^{\ast}(t,\theta){v^{\ast}}'(t+\theta)\theta\right)
\end{array}
\end{equation}
for $n(t,\theta)=n(t,s_{\omega^{\ast}}(\sigma)):=\mathtt{n}(t,\sigma)$ thanks to the change of variable \eqref{somega} for $\omega=\omega^{\ast}$. Above we used also the fact that $v^{\ast}$ is the $1$-periodic solution of \eqref{rfdet} for $\omega=\omega^{\ast}$ again.

\bigskip
Next we consider tools from the theory of adjoint equations for RFDEs, see \cite[Sections 6.3 and 8.2]{hale77}. Let us start from equation \eqref{Lrfdet2}, which, through \eqref{L2*} and \eqref{somega}, reads
\begin{equation}\label{linearized}
y'(t)=\int_{-r}^{0}\dd_{\theta}n^{\ast}(t,\theta)y(t+\theta).
\end{equation}
Assume it has a solution $y$ defined on the whole line, as is the case for periodic solutions. Consider also the formal adjoint equation
\begin{equation}\label{adjoint}
z(t)+\int_{t}^{\infty}z(\theta)n^{\ast}(\theta,t-\theta)\dd\theta=\text{constant},
\end{equation}
assuming it has a solution $z$ defined as well on the whole line, and let us stress that $z(t)$ is intended as a row vector in $\mathbb{R}^{d}$. For the latter define also the state $z^{t}$ as $z^{t}(\nu):=z(t+\nu)$, $\nu\in[0,r]$. Finally, for the solutions $y$ and $z$ introduced above let us consider the bilinear form \cite[equation (2.10) p.199]{hale77}
\begin{equation}\label{bil}
(z^{t},y_{t})_{t}:=z(t)y(t)+\int_{-r}^{0}\dd_{\beta}\left[\int_{0}^{r}z(t+\xi)n^{\ast}(t+\xi,\beta-\xi)\dd\xi\right]y(t+\beta).
\end{equation}
It follows from \cite[Lemma 2.1, Section 8.2]{hale77} that
\begin{equation}\label{zyc}
(z^{t},y_{t})_{t}=c
\end{equation}
for some constant $c\in\mathbb{R}$ independently of $t$.

\bigskip
We already observed in Remark \ref{r_mu=1} that \eqref{linearized} has a $1$-periodic solution (viz. ${v^{\ast}}'$). Moreover, under the hypothesis of Proposition \ref{p_Ax*2}, this is the only $1$-periodic solution. Then also the adjoint equation \eqref{adjoint} has a unique $1$-periodic solution. Indeed, the monodromy operators of \eqref{linearized} and of its adjoint \eqref{adjoint} share the same spectrum \cite[Section 8.2]{hale77} (see \cite[Section 6.4]{hale77} for the relevant definitions and properties in the autonomous case), and it is well-known that periodic solutions correspond to having the multiplier $1$ (which is in fact assumed to be simple).

Let then $z^{\ast}$ be the $1$-periodic solution of \eqref{adjoint}. Following \cite[p.200]{hale77} and \eqref{zyc}, we conclude that 
\begin{equation}\label{zyc*}
({z^{\ast}}^{t},{v^{\ast}_{t}}')_{t}=c^{\ast}\neq0.
\end{equation}
Integrating over one period gives
\begin{equation*}
\begin{split}
c^{\ast}=\int_{0}^{1}c^{\ast}dt={}&\int_{0}^{1}z^{\ast}(t){v^{\ast}}'(t)dt\\
{}&+\int_{0}^{1}\int_{-r}^{0}\dd_{\beta}\left[\int_{0}^{r}z^{\ast}(t+\xi)n^{\ast}(t+\xi,\beta-\xi)\dd\xi\right]{v^{\ast}}'(t+\beta)\dd t.
\end{split}
\end{equation*}
As for the last integral, thanks to the periodicity of ${v^{\ast}}'$, $z^{\ast}$ and $n^{\ast}$, by exchanging the order of integration and by observing that
\begin{equation*}
\int_{a}^{b}\dd_{\theta}n^{\ast}(t,\theta){v^{\ast}}'(t+\theta)=0
\end{equation*}
whenever $a>0$ or $b<-r$ (see again \cite[Section 6.1]{hale77}), we have
\begin{equation*}
\begin{split}
\int_{0}^{1}\int_{-r}^{0}\dd_{\beta}{}&\bigg[\int_{0}^{r}z^{\ast}(t+\xi)n^{\ast}(t+\xi,\beta-\xi)\dd\xi\bigg]{v^{\ast}}'(t+\beta)\dd t\\
={}&\int_{-r}^{0}\dd_{\beta}\int_{0}^{r}\bigg[\int_{0}^{1}z^{\ast}(t+\xi)n^{\ast}(t+\xi,\beta-\xi){v^{\ast}}'(t+\beta)\dd t\bigg]\dd\xi\\
={}&\int_{-r}^{0}\dd_{\beta}\int_{0}^{r}\bigg[\int_{0}^{1}z^{\ast}(t)n^{\ast}(t,\beta-\xi){v^{\ast}}'(t+\beta-\xi)\dd t\bigg]\dd\xi\\
={}&\int_{0}^{1}z^{\ast}(t)\int_{-r}^{0}\dd_{\beta}\bigg[\int_{0}^{r}n^{\ast}(t,\beta-\xi){v^{\ast}}'(t+\beta-\xi)\dd\xi\bigg]\dd t\\
={}&\int_{0}^{1}z^{\ast}(t)\int_{0}^{r}\bigg[\int_{-r}^{0}\dd_{\beta}n^{\ast}(t,\beta-\xi){v^{\ast}}'(t+\beta-\xi)\bigg]\dd\xi\dd t\\
={}&\int_{0}^{1}z^{\ast}(t)\int_{0}^{r}\bigg[\int_{-\xi-r}^{-\xi}\dd_{\theta}n^{\ast}(t,\theta){v^{\ast}}'(t+\theta)\bigg]\dd\xi\dd t\\
={}&\int_{0}^{1}z^{\ast}(t)\int_{0}^{r}\bigg[\int_{-r}^{-\xi}\dd_{\theta}n^{\ast}(t,\theta){v^{\ast}}'(t+\theta)\bigg]\dd\xi\dd t\\
={}&\int_{0}^{1}z^{\ast}(t)\int_{-r}^{0}\dd_{\theta}\bigg[\int_{0}^{-\theta}n^{\ast}(t,\theta){v^{\ast}}'(t+\theta)\dd\xi\bigg]\dd t\\
={}&\int_{0}^{1}z^{\ast}(t)\int_{-r}^{0}\dd_{\theta}n^{\ast}(t,\theta){v^{\ast}}'(t+\theta)\bigg[\int_{0}^{-\theta}\dd\xi\bigg]\dd t\\
={}&-\int_{0}^{1}z^{\ast}(t)\int_{-r}^{0}\dd_{\theta}n^{\ast}(t,\theta){v^{\ast}}'(t+\theta)\theta\dd t.
\end{split}
\end{equation*}
By using \eqref{M2*}, we finally get
\begin{equation}\label{c*}
\begin{split}
c^{\ast}={}&\int_{0}^{1}z^{\ast}(t){v^{\ast}}'(t)dt-\int_{0}^{1}z^{\ast}(t)\int_{-r}^{0}\dd_{\theta}n^{\ast}(t,\theta){v^{\ast}}'(t+\theta)\theta\dd t\\
={}&\int_{0}^{1}z^{\ast}(t)\left({v^{\ast}}'(t)-\int_{-r}^{0}\dd_{\theta}n^{\ast}(t,\theta){v^{\ast}}'(t+\theta)\theta\right)\dd t\\
={}&\omega^{\ast}\int_{0}^{1}z^{\ast}(t)\mathfrak{M}_{2}^{\ast}(t)\dd t.
\end{split}
\end{equation}

\bigskip
Eventually, let us go back to \eqref{yL2M2}, which has a $1$-periodic solution if $k_{1}=0$. Then, \cite[Theorem 1.2, Section 9.1]{hale77} necessarily gives
\begin{equation*}
\int_{0}^{1}z^{\ast}(t)\mathfrak{M}_{2}^{\ast}(t)\dd t=0
\end{equation*}
for any $1$-periodic solution $z^{\ast}$ of \eqref{adjoint}. This means $c^{\ast}=0$ by \eqref{c*}, which is {\color{black}a contradiction} thanks to \eqref{zyc*}.
\begin{remark}
Let us note that \cite[Theorem 1.2, Section 9.1]{hale77} refers to the differential form of \eqref{adjoint} (which indeed holds in the case of periodic solutions).
\end{remark}
\begin{remark}\label{mult}
We just mention that the main results of this section could be seen in the general framework of Fredholm theory, see, e.g., \cite[Chapter 4]{kress89}. To this aim, it should be proved that the state spaces of \eqref{linearized} and \eqref{adjoint} constitute a {\it dual system} when equipped with the bilinear form \eqref{bil} (which should be nondegenerate), and their respective monodromy operators are adjoint \cite[Equation (2.12), Section 8.2]{hale77}. Thus, the first Fredholm theorem \cite[Theorem 4.14]{kress89} would tell us that \eqref{linearized} and its adjoint \eqref{adjoint} have the same number of linearly independent $1$-periodic solutions (that is just $1$ under the hypothesis of hyperbolicity). The fact that they are not orthogonal ($c^{\ast}\neq 0$) would follow from the Fredholm alternative theorem \cite[Theorem 4.17]{kress89} and the decomposition \eqref{deco}.
\end{remark}
\section{Discretization}
\label{s_discretization}
As anticipated, \cite{mas15NM} requires also other assumptions besides the theoretical ones validated in Section \ref{s_validation_t}, which concern the chosen discretization scheme for the numerical method. Such scheme is defined by the {\it primary} and the {\it secondary} discretizations. We first introduce these discretizations and then check in Section \ref{s_validation_n} the validity of the relevant numerical assumptions in \cite{mas15NM}. Recall from the end of Section \ref{s_validation_t} that we deal only with formulation \eqref{bvp2} and thus we remark again that, to lighten the notation, we drop the index 2 used up to Section \ref{s_abstract} to distinguish from formulation \eqref{bvp1}.

\bigskip
The primary discretization consists in reducing the spaces $\mathbb{U}$ and $\mathbb{A}$ to finite-dimensional spaces $\mathbb{U}_{L}$ and $\mathbb{A}_{L}$, given a level of discretization $L$. This happens by means of {\it restriction} operators $\rho_{L}^{+}:\mathbb{U}\rightarrow\mathbb{U}_{L}$, $\rho_{L}^{-}:\mathbb{A}\rightarrow\mathbb{A}_{L}$ and {\it prolongation} operators $\pi_{L}^{+}:\mathbb{U}_{L}\rightarrow\mathbb{U}$, $\pi_{L}^{-}:\mathbb{A}_{L}\rightarrow\mathbb{A}$, which extend respectively to
\begin{equation}\label{RL}
R_{L}:\mathbb{U}\times\mathbb{A}\times\mathbb{B}\rightarrow\mathbb{U}_{L}\times\mathbb{A}_{L}\times\mathbb{B},\quad R_{L}(u,\psi,\omega):=(\rho_{L}^{+}u,\rho_{L}^{-}\psi,\omega)
\end{equation}
and
\begin{equation}\label{PL}
P_{L}:\mathbb{U}_{L}\times\mathbb{A}_{L}\times\mathbb{B}\rightarrow\mathbb{U}\times\mathbb{A}\times\mathbb{B},\quad P_{L}(u_{L},\psi_{L},\omega):=(\pi_{L}^{+}u_{L},\pi_{L}^{-}\psi_{L},\omega).
\end{equation}
All of them are linear and bounded. In the following we describe the specific choices we make in this context, based on piecewise polynomial interpolation.

Starting from $\mathbb{U}$, which concerns the interval $[0,1]$, we choose the uniform {\it outer} mesh
\begin{equation}\label{outmesh+}
\Omega_{L}^{+}:=\{t_{i}^{+}=ih\ :\ i=0,1,\ldots,L,\;h=1/L\}\subset[0,1],
\end{equation}
and {\it inner} meshes
\begin{equation}\label{inmesh+}
\Omega_{L,i}^{+}:=\{t_{i,j}^{+}:=t_{i-1}^{+}+c_{j}h\ :\ j=1,\ldots,m\}\subset[t_{i-1}^{+},t_{i}^{+}],\quad i=1,\ldots,L,
\end{equation}
where $0<c_{1}<\cdots<c_{m}<1$ are given abscissae for $m$ a positive integer. Correspondingly, we define
\begin{equation}\label{UL}
\mathbb{U}_{L}:=\mathbb{R}^{(1+Lm)\times d},
\end{equation}
whose elements $u_{L}$ are indexed as
\begin{equation}\label{uL}
u_{L}:=(u_{1,0},u_{1,1},\ldots,u_{1,m},\ldots,u_{L,1}\ldots,u_{L,m})^{T}
\end{equation}
with components in $\mathbb{R}^{d}$. Finally, we define, for $u\in\mathbb{U}$,
\begin{equation}\label{rL+}
\rho_{L}^{+}u:=(u(0),u(t_{1,1}^{+}),\ldots,u(t_{1,m}^{+}),\ldots,u(t_{L,1}^{+})\ldots,u(t_{L,m}^{+}))^{T}\in\mathbb{U}_{L}
\end{equation}
and, for $u_{L}\in\mathbb{U}_{L}$, $\pi_{L}^{+}u_{L}\in\mathbb{U}$ as the unique element of the space
\begin{equation}\label{PiLm+}
\Pi_{L,m}^{+}:=\{p\in C([0,1],\mathbb{R}^d)\ :\ p|_{[t_{i-1}^{+},t_{i}^{+}]}\in\Pi_m,\;i=1,\ldots,L\}
\end{equation}
such that
\begin{equation}\label{pL+}
\pi_{L}^{+}u_{L}(0)=u_{1,0},\quad\pi_{L}^{+}u_{L}(t_{i,j}^{+})=u_{i,j},\quad j=1,\ldots,m,\;i=1,\ldots,L.
\end{equation}
Above $\Pi_{m}$ is the space of $\mathbb{R}^d$-valued polynomials having degree $m$ and, when needed, we represent $p\in\Pi_{L,m}^{+}$ through its pieces as
\begin{equation}\label{lagrange}
p|_{[t_{i-1}^{+},t_{i}^{+}]}(t)=\sum_{j=0}^{m}\ell_{m,i,j}(t)p(t_{i,j}^{+}),\quad t\in[0,1],
\end{equation}
where, for ease of notation, we implicitly set
\begin{equation}\label{t0+}
t_{i,0}^{+}:=t_{i-1}^{+},\quad i=1,\ldots,L,
\end{equation}
and $\{\ell_{m,i,0},\ell_{m,i,1},\ldots,\ell_{m,i,m}\}$ is the Lagrange basis relevant to the nodes $\{t_{i,0}^{+}\}\cup\Omega_{L,i}^{+}$. Observe that the latter is invariant with respect to $i$ as long as we fix the abscissae $c_{j}$, $j=1,\ldots,m$, defining the inner meshes \eqref{inmesh+}. Indeed, for every $i=1,\ldots,L$,
\begin{equation*}
\ell_{m,i,j}(t)=\ell_{m,j}\left(\frac{t-t_{i-1}^{+}}{h}\right),\quad t\in[t_{i-1}^{+},t_{i}^{+}],
\end{equation*}
where $\{\ell_{m,0},\ell_{m,1},\ldots,\ell_{m,m}\}$ is the Lagrange basis in $[0,1]$ relevant to the abscissae $c_{0},c_{1},\ldots,c_{m}$ with $c_{0}:=0$. Moreover, it is useful to define also the associated Lebesgue constants as
\begin{equation*}
\Lambda_{m,i}:=\max_{t\in[t_{i-1}^{+},t_{i}^{+}]}\sum_{j=0}^{m}|\ell_{m,i,j}(t)|,\quad i=1,\ldots,L,
\end{equation*}
which turn out to be independent of $i$ as well:
\begin{equation}\label{lebesgue}
\Lambda_{m,i}=\Lambda_{m}:=\max_{t\in[0,1]}\sum_{j=0}^{m}|\ell_{m,j}(t)|.
\end{equation}
Let us define also
\begin{equation}\label{lebesgue'}
\Lambda_{m,i}':=\max_{t\in[t_{i-1}^{+},t_{i}^{+}]}\sum_{j=0}^{m}|\ell_{m,i,j}'(t)|=\Lambda_{m}':=\max_{t\in[0,1]}\sum_{j=0}^{m}|\ell_{m,j}'(t)|.
\end{equation}

Similarly, for $\mathbb{A}$, which concerns the interval $[-1,0]$, we choose
\begin{equation}\label{outmesh-}
\Omega_{L}^{-}:=\{t_{i}^{-}=ih-1\ :\ i=0,1,\ldots,L,\;h=1/L\}\subset[-1,0],
\end{equation}
and
\begin{equation}\label{inmesh-}
\Omega_{L,i}^{-}:=\{t_{i,j}^{-}:=t_{i-1}^{-}+c_{j}h\ :\ j=1,\ldots,m\}\subset[t_{i-1}^{-},t_{i}^{-}],\quad i=1,\ldots,L.
\end{equation}
Correspondingly, we define
\begin{equation}\label{AL}
\mathbb{A}_{L}:=\mathbb{R}^{(1+Lm)\times d}
\end{equation}
with indexing
\begin{equation}\label{psiL}
\psi_{L}:=(\psi_{1,0},\psi_{1,1},\ldots,\psi_{1,m},\ldots,\psi_{L,1}\ldots,\psi_{L,m})^{T};
\end{equation}
for $\psi\in\mathbb{A}$,
\begin{equation}\label{rL-}
\rho_{L}^{-}\psi:=(\psi(-1),\psi(t_{1,1}^{-}),\ldots,\psi(t_{1,m}^{-}),\ldots,\psi(t_{L,1}^{-})\ldots,\psi(t_{L,m}^{-}))^{T}\in\mathbb{A}_{L}
\end{equation}
and, for $\psi_{L}\in\mathbb{A}_{L}$, $\pi_{L}^{-}\psi_{L}\in\mathbb{A}$ as the unique element of the space
\begin{equation}\label{PiLm-}
\Pi_{L,m}^{-}:=\{p\in C([-1,0],\mathbb{R}^d)\ :\ p|_{[t_{i-1}^{-},t_{i}^{-}]}\in\Pi_m,\;i=1,\ldots,L\}
\end{equation}
such that
\begin{equation}\label{pL-}
\pi_{L}^{-}\psi_{L}(-1)=\psi_{1,0},\quad\pi_{L}^{-}\psi_{L}(t_{i,j}^{-})=\psi_{i,j},\quad j=1,\ldots,m,\;i=1,\ldots,L.
\end{equation}
Elements in $\Pi_{L,m}^{-}$ are represented in the same way as those of $\Pi_{L,m}^{+}$ by suitably adapting both \eqref{lagrange} and \eqref{t0+}, so that also $\Lambda_{m}$ in \eqref{lebesgue} and $\Lambda_{m}'$ in \eqref{lebesgue'} are unchanged.
\begin{remark}\label{r_primary}
Let us note that more general choices can be made like, for instance, nonuniform outer meshes, inner meshes varying with respect to the relevant interval, different type and number of nodes and so forth. Extensions to these cases is straightforward but rather technical, so that we omit the details. Observe, however, that in practical applications {\it adaptive} meshes represent a standard, see, e.g., \cite{elir00} for delay differential equations or even \cite{acr81} for ordinary differential equations. Moreover, abscissae including the extrema of $[0,1]$ can also be considered, paying attention to put the correct constraints at the internal outer nodes, i.e., $t_{i}^{\pm}$ for $i=1,\ldots,L-1$.
\end{remark}
\begin{remark}\label{r_rangePR}
Let us underline that the range of the prolongation operator $\pi_{L}^{+}$ is contained in $C([0,1],\mathbb{R}^d)$, hence it does not cover all of $\mathbb{U}$. Nevertheless, later on and in the convergence analysis of Section \ref{s_convergence} it will be clear that $\pi_{L}^{+}\rho_{L}^{+}$ is applied only to functions which are at least continuous.
\end{remark}
\noindent In Appendix \ref{s_basic} we collect some classical results on interpolation in terms of the primary discretization just introduced, which are frequently used in the proofs of the forthcoming results. Other preparatory results are collected in Appendix \ref{s_other}.
\begin{remark}\label{r_FEMSEM}
As far as the convergence of the primary discretization is concerned later on, we soon underline that we develop all the analysis for the FEM, i.e., under the hypothesis of letting $L\rightarrow\infty$ while keeping $m$ fixed. This is also the traditional approach followed in practical implementations, as, e.g., in \texttt{MatCont} for ordinary differential equations \cite{matcont} or in \texttt{DDE-Biftool} for delay differential equations \cite{ddebiftool}, usually combined with an adaptive selection of the outer mesh. On the other hand, the convergence of the SEM, i.e., letting $m\rightarrow\infty$ for fixed $L$, is only briefly accounted for in Section \ref{s_SEM}. We anticipate that we give only some insight because this approach is out of our current primary interest, since it is not widely used in practical applications. However, it is not yet clear whether the convergence of the SEM is guaranteed under the general framework of reference for the current work.
\end{remark}

\bigskip
The secondary discretization consists in defining, for a given level of discretization $M$, an operator $\mathcal{F}_{M}$ that can be exactly computed, and is meant to be used in place of $\mathcal{F}$ in the first of \eqref{FB2}. In particular, we define $\mathcal{F}_{M}$ through an approximated version $G_{M}$ of the right-hand side $G$ of \eqref{rfdes} as $\mathcal{F}_{M}(u,\psi,\omega):=\omega G_{M}(\mathcal{G}(u,\psi)_{\cdot}\circ s_{\omega})$. Correspondingly, $\Phi_{M}$ is the operator obtained by replacing $\mathcal{F}$ in $\Phi$ in \eqref{Phi2} with its approximated version, i.e., $\Phi_{M}:\mathbb{U}\times\mathbb{A}\times\mathbb{B}\rightarrow\mathbb{U}\times\mathbb{A}\times\mathbb{B}$ defined by
\begin{equation}\label{PhiM}
\Phi_{M}(u,\psi,\omega):=
\begin{pmatrix}
\omega G_{M}(\mathcal{G}(u,\psi)_{\cdot}\circ s_{\omega})\\[2mm]
\mathcal{G}(u,\psi)_{1}\\[2mm]
\omega-p(\mathcal{G}(u,\psi)\vert_{[0,1]})
\end{pmatrix}.
\end{equation}
The need for introducing $G_{M}$ is due, for instance, to the presence in $G$ of integrals defining distributed delays, which might need indeed the application of suitable qua\-drature rules. A secondary discretization for $\mathcal{G}$ in \eqref{Phi2} is instead unnecessary, since it can be evaluated exactly in $\pi_{L}^{+}\mathbb{U}_{L}\times\pi_{L}^{-}\mathbb{A}_{L}$ according to \eqref{UL} and \eqref{AL}. Similarly, we assume that the operator $p$ defining the phase condition in \eqref{bvp2} can be evaluated exactly in $\mathcal{G}(\pi_{L}^{+}\mathbb{U}_{L},\pi_{L}^{-}\mathbb{A}_{L})|_{[0,1]})$. Let us note that in the case of integral phase conditions the latter statement translates into applying the piecewise quadrature based on the mesh of the primary discretization, which is indeed the standard approach used in practical applications.

\bigskip
The two discretizations together allow us to define the discrete version 
\begin{equation}\label{PhiLM1}
\Phi_{L,M}:=R_{L}\Phi_{M} P_{L}:\mathbb{U}_{L}\times\mathbb{A}_{L}\times\mathbb{B}\rightarrow\mathbb{U}_{L}\times\mathbb{A}_{L}\times\mathbb{B}
\end{equation}
of the fixed point operator $\Phi$ in \eqref{Phi2} as
\begin{equation*}
\Phi_{L,M}(u_{L},\psi_{L},\omega):=
\begin{pmatrix}
\omega\rho_{L}^{+}G_{M}(\mathcal{G}(\pi_{L}^{+}u_{L},\pi_{L}^{-}\psi_{L})_{\cdot}\circ s_{\omega})\\[2mm]
\rho_{L}^{-}\mathcal{G}(\pi_{L}^{+}u_{L},\pi_{L}^{-}\psi_{L})_1\\[2mm]
\omega-p(\mathcal{G}(\pi_{L}^{+}u_{L},\pi_{L}^{-}\psi_{L})\vert_{[0,1]})
\end{pmatrix}.
\end{equation*}
Fixed points $(u_{L,M}^{\ast},\psi_{L,M}^{\ast},\omega_{L,M}^{\ast})$ of $\Phi_{L,M}$ can be found by standard solvers for nonlinear systems of algebraic equations. Then, in Section \ref{s_convergence}, we consider the prolongation $P_{L}(u_{L,M}^{\ast},\psi_{L,M}^{\ast},\omega_{L,M}^{\ast})$ as an approximation of a fixed point $(u^{\ast},\psi^{\ast},\omega^{\ast})$ of $\Phi$ in \eqref{Phi2} and, correspondingly, $v_{L,M}^{\ast}:=\mathcal{G}(\pi_{L}^{+}u_{L,M}^{\ast},\pi_{L}^{-}\psi_{L,M}^{\ast})$ as an approximation of the solution $v^{\ast}=\mathcal{G}(u^{\ast},\psi^{\ast})$ of \eqref{bvp2}.
\subsection{Validation of the numerical assumptions}
\label{s_validation_n}
We now proceed to prove the validity of the numerical assumptions in \cite{mas15NM}  mentioned at the beginning of Section \ref{s_discretization}. As done in Section \ref{s_validation_t}, for ease of reference throughout the text, we collect below all the hypotheses that are used for proving the forthcoming results.
\begin{enumerate}[label=(N\arabic*),ref=(N\arabic*)]
\item\label{N1} The primary discretization of the space $\mathbb{U}$ is based on the choices \eqref{outmesh+}--\eqref{pL+}.
\item\label{N2} The primary discretization of the space $\mathbb{A}$ is based on the choices \eqref{outmesh-}--\eqref{pL-}.
\item\label{N3} For every positive integer $M$, $G_{M}$ is Fr\'echet-differentiable at every $\mathtt{y}\in\mathtt{Y}$.
\item\label{N4} For every positive integer $M$, $G_{M}\in\mathcal{C}^{1}(\mathtt{Y},\mathbb{R}^{d})$ in the sense of Fr\'echet.
\item\label{N5} There exist $r>0$ and $\kappa\geq0$ such that
\begin{equation*}
\|DG_{M}(\mathtt{y})-DG_{M}(v^{\ast}_{t}\circ s_{\omega^{\ast}})\|_{\mathbb{R}^{d}\leftarrow\mathtt{Y}}\leq\kappa\|\mathtt{y}-v^{\ast}_{t}\circ s_{\omega^{\ast}}\|_{\mathtt{Y}}
\end{equation*}
for every $\mathtt{y}\in\overline B(v^{\ast}_{t}\circ s_{\omega^{\ast}},r)$, uniformly with respect to $t\in[0,1]$ and for every positive integer $M$.
\item\label{N6} It holds
\begin{equation*}
\lim_{M\rightarrow\infty}|G_{M}(v^{\ast}_{t}\circ s_{\omega^{\ast}})-G(v^{\ast}_{t}\circ s_{\omega^{\ast}})|=0
\end{equation*}
uniformly with respect to $t\in[0,1]$.
\item\label{N7} It holds
\begin{equation*}
\lim_{M\rightarrow\infty}\|DG_{M}(v^{\ast}_{t}\circ s_{\omega^{\ast}})-DG(v^{\ast}_{t}\circ s_{\omega^{\ast}})\|_{\mathbb{R}^d\leftarrow\mathtt{Y}}=0
\end{equation*}
uniformly with respect to $t\in[0,1]$.
\end{enumerate}
\begin{remark}\label{r_N5}
The uniformity with respect to $M$ of $r$ and $\kappa$ in \ref{N5} may appear restrictive. However, as anticipated, among the main reasons to introduce $G_{M}$ is the quadrature of distributed delays. Thus, if one considers right-hand sides $G$ of the form
\begin{equation*}
G(\psi)=\int_{-\tau}^{0}H(\theta,\psi(\theta))\dd\theta
\end{equation*}
for some integration kernel $H$ with locally Lipschitz continuous derivative with respect to the second argument, then \ref{T5} is satisfied and also \ref{N5} follows from the application of any convergent interpolatory formula. The same argument holds also if
\begin{equation*}
G(\psi)=g\left(\int_{-\tau}^{0}H(\theta)\psi(\theta)\dd\theta\right)
\end{equation*}
for some $g$ with locally Lipschitz continuous derivative and any integration kernel $H$.
\end{remark}

\bigskip
The first assumption to be verified in \cite{mas15NM} is Assumption A$\mathfrak{F}_K\mathfrak{B}_K$ (page 535). Its validity is proved next.
\begin{proposition}\label{p_AfMb}
Under \ref{T1}, \ref{T2} and \ref{N3}, $\mathcal{F}_{M}$ is Fr\'echet-differentiable, from the right with respect to $\omega$, at every point $(v,u,\omega)\in\mathbb{V}\times\mathbb{U}\times(0,+\infty)$ and
\begin{equation*}
D\mathcal{F}_{M}(\hat{v},\hat u,\hat\omega)(v,u,\omega)=\mathfrak{L}_{M}(\cdot;\hat{v},\hat\omega)v_{\cdot}\circ s_{\hat\omega}+\omega \mathfrak{M}_{M}(\cdot;\hat{v},\hat\omega)
\end{equation*}
for $(v,u,\omega)\in\mathbb{V}\times\mathbb{U}\times(0,+\infty)$, where, for $t\in [0,1]$,
\begin{equation}\label{L2M}
\mathfrak{L}_{M}(t;v,\omega):=\omega DG_{M}(v_{t}\circ s_{\omega})
\end{equation}
and
\begin{equation}\label{M2M}
\mathfrak{M}_{M}(t;v,\omega):=G_{M}(v_{t}\circ s_{\omega})-\mathfrak{L}_{M}(t;\hat{v},\hat\omega)v'_{t}\circ s_{\omega}\cdot\frac{s_{\omega}}{\omega}.
\end{equation}
\end{proposition}
\begin{proof}
The proof goes as the one of Proposition \ref{p_Afb}, after replacing $G$ with $G_{M}$.
\end{proof}
\noindent As neither $\mathcal{G}$ nor $p$ is affected by the secondary discretization, Proposition \ref{p_AfMb} guarantees the Fr\'echet-differentiability of the fixed point operator $\Phi_{M}$ in \eqref{PhiM} as stated next.
\begin{corollary}\label{c_DPhiM}
Under \ref{T1}, \ref{T2} and \ref{N3}, $\Phi_{M}$ in \eqref{PhiM} is Fr\'echet-differentiable, from the right with respect to $\omega$, at every $(u,\psi,\omega)\in\mathbb{U}\times\mathbb{A}\times(0,+\infty)$ and
\begin{equation*}
D\Phi_{M}(\hat u,\hat\psi,\hat\omega)(u,\psi,\omega)=\begin{pmatrix}
\mathfrak{L}_{M}(\cdot;\mathcal{G}(\hat u,\hat\psi),\hat\omega)\mathcal{G}(u,\psi)_{\cdot}\circ s_{\hat\omega}+\omega\mathfrak{M}_{M}(\cdot;\mathcal{G}(\hat u,\hat\psi),\hat\omega)\\[2mm]
\mathcal{G}(u,\psi)_{1}\\[2mm]
\omega-p(\mathcal{G}(u,\psi)\vert_{[0,1]})
\end{pmatrix}
\end{equation*}
for $(u,\psi,\omega)\in\mathbb{U}\times\mathbb{A}\times(0,+\infty)$, $\mathfrak{L}_{M}$ in \eqref{L2M} and $\mathfrak{M}_{M}$ in \eqref{M2M}.
\end{corollary}
\begin{proof}
The proof goes as the one of Corollary \ref{c_DPhi2}, after replacing $G$ with $G_{M}$, and therefore $\Phi$ with $\Phi_{M}$.
\end{proof}

\bigskip
The other two assumptions in \cite{mas15NM}, viz. CS1 and CS2 (page 537), represent stability conditions on the chosen discretization. We stress that they concern the operator $P_{L}R_{L}\Phi_{M}$. Note that, differently from $\Phi_{L,M}$ in \eqref{PhiLM1}, $P_{L}R_{L}\Phi_{M}$ is defined on the same space of $\Phi$ in \eqref{Phi2}. Thus we use the former for computing the discrete approximations and the latter to analyze their convergence. This and the relation between all the relevant fixed points are arguments of Section \ref{s_convergence}.

In what follows it is useful to define $\Psi,\Psi_{L,M}:\mathbb{U}\times\mathbb{A}\times\mathbb{B}\rightarrow\mathbb{U}\times\mathbb{A}\times\mathbb{B}$ as
\begin{equation}\label{PsiLM}
\Psi:=I_{\mathbb{U}\times\mathbb{A}\times\mathbb{B}}-\Phi,\qquad\Psi_{L,M}:=I_{\mathbb{U}\times\mathbb{A}\times\mathbb{B}}-P_{L}R_{L}\Phi_{M}.
\end{equation}
Both are Fr\'echet-differentiable, the first thanks to Corollary \ref{c_DPhi2} and the second thanks to Corollary \ref{c_DPhiM} and the linearity of both $P_{L}$ and $R_{L}$. It is also convenient to adopt the abbreviations
\begin{equation}\label{LM*MM*}
\mathfrak{L}_{M}^{\ast}:=\mathfrak{L}_{M}(\cdot;v^{\ast},\omega^{\ast}),\qquad\mathfrak{M}_{M}^{\ast}:=\mathfrak{M}_{M}(\cdot;v^{\ast},\omega^{\ast})
\end{equation}
in accordance with \eqref{L*M*}.

\bigskip
Assumption CS1 in \cite{mas15NM} is somehow the discrete version of A$x^{\ast}1$ therein, here Proposition \ref{p_Ax*1}. It can be proved valid thanks to the following.
\begin{proposition}\label{p_CS1}
Under \ref{T1}, \ref{T2}, \ref{N1}, \ref{N2}, \ref{N3} and \ref{N5}, there exist $r_{1}\in(0,\omega^{\ast})$ and $\kappa\geq0$ such that
\begin{equation*}
\setlength\arraycolsep{0.1em}\begin{array}{rcl}
\|D\Psi_{L,M}(u,\psi,\omega)&-&D\Psi_{L,M}(u^{\ast},\psi^{\ast},\omega^{\ast})\|_{\mathbb{U}\times\mathbb{A}\times\mathbb{B}\leftarrow\mathbb{U}_{2}\times\mathbb{A}\times(0,+\infty)}\\[2mm]
&\leq&\kappa\|(u,\psi,\omega)-(u^{\ast},\psi^{\ast},\omega^{\ast})\|_{\mathbb{U}\times\mathbb{A}\times\mathbb{B}}
\end{array}
\end{equation*}
for all $(u,\psi,\omega)\in \overline{B}((u^{\ast},\psi^{\ast},\omega^{\ast}),r_{1})$ and for all positive integers $L$ and $M$.
\end{proposition}
\begin{proof}
By following the proof of Proposition \ref{p_Ax*1}, after replacing $G$ with $G_{M}$, and therefore $\Phi$ with $\Phi_{M}$, we get that there exist $r_{1}\in(0,\omega^{\ast})$ and $\kappa_{1}\geq0$ such that
\begin{equation*}
\setlength\arraycolsep{0.1em}\begin{array}{rcl}
\|D\Phi_{M}(u,\psi,\omega)&-&D\Phi_{M}(u^{\ast},\psi^{\ast},\omega^{\ast})\|_{\mathbb{U}\times\mathbb{A}\times\mathbb{B}\leftarrow\mathbb{U}\times\mathbb{A}\times(0,+\infty)}\\[2mm]
&\leq&\kappa_{1}\|(u,\psi,\omega)-(u^{\ast},\psi^{\ast},\omega^{\ast})\|_{\mathbb{U}\times\mathbb{A}\times\mathbb{B}}
\end{array}
\end{equation*}
for all $(u,\psi,\omega)\in\overline{B}((u^{\ast},\psi^{\ast},\omega^{\ast}),r_{1})$. In particular, we recall that we can choose $r_{1}=r/2$ for $r$ in \ref{N5}. By Corollary \ref{c_PLRL} in Appendix \ref{s_basic}, the thesis follows directly from the second of \eqref{PsiLM} by choosing $\kappa=\kappa_{1}\cdot\max\{\Lambda_{m}+\Lambda'_{m},1\}$.
\end{proof}
\begin{remark}\label{r_KLM}
Note that $\kappa$ is independent of $L$ thanks to \eqref{PLRL}; rather, it depends on $m$. Moreover, it is also independent of $M$ under \ref{N5}.
\end{remark}

\bigskip
Assumption CS2 in \cite{mas15NM} (page 537) is, in a certain sense, the discrete version of A$x^{\ast}2$ therein, here Proposition \ref{p_Ax*2}. As for the latter, its proof is not immediate. Hence we separate the result into several steps, the main one concerning the invertibility of $D\Psi_{L,M}(u^{\ast},\psi^{\ast},\omega^{\ast})$. In principle, one could attempt to prove it by resorting to the Banach's perturbation lemma, thus showing first that
\begin{equation*}
\lim_{L,M\rightarrow\infty}\|D\Psi_{L,M}(u^{\ast},\psi^{\ast},\omega^{\ast})-D\Psi(u^{\ast},\psi^{\ast},\omega^{\ast})\|_{\mathbb{U}\times\mathbb{A}\times\mathbb{B}\leftarrow\mathbb{U}\times\mathbb{A}\times(0,+\infty)}=0.
\end{equation*}
This in turn would require
\begin{equation*}
\lim_{L\rightarrow\infty}\|(I_{\mathbb{A}}-\pi_{L}^{-}\rho_{L}^{-})\mathcal{G}(u,\psi)_{1}\|_{\mathbb{A}}=0
\end{equation*}
through \eqref{RL}, \eqref{PL}, \eqref{PhiM} and \eqref{PsiLM}. The latter cannot hold for all $(u,\psi)\in\mathbb{U}\times\mathbb{A}$ given the choices of both $\mathbb{U}$ and $\mathbb{A}$ in \ref{T2} due to the second of \eqref{normB1inf}. However, the Banach's perturbation lemma, although of common use, represents just a sufficient criterion (yet fundamental: indeed we make use of it several times). Therefore, in the following we prove the invertibility of $D\Psi_{L,M}(u^{\ast},\psi^{\ast},\omega^{\ast})$ directly by following the lines of the proof of Proposition \ref{p_Ax*2}. To this aim, we first need to show that the initial value problem for
\begin{equation}\label{LrfdetLM}
y'(t)=[\pi_{L}^{+}\rho_{L}^{+}\mathfrak{L}_{M}^{\ast}y_{\cdot}\circ s_{\omega^{\ast}}](t)
\end{equation}
is well-posed, and thus defining an associated evolution operator $T_{L,M}^{\ast}(t,s):Y\rightarrow Y$ is meaningful (for $t,s\in[0,1]$ and $t\geq s$). In what follows it is also convenient to use the abbreviations
\begin{equation}\label{GKKM}
\setlength\arraycolsep{0.1em}\begin{array}{ll}
\mathcal{G}^{+}u:=\mathcal{G}(u,0),&\qquad\mathcal{G}^{-}\psi:=\mathcal{G}(0,\psi),\\[2mm]
\mathcal{K}^{\ast,+}u:=\mathfrak{L}^{\ast}(\mathcal{G}^{+}u)_{\cdot}\circ s_{\omega^{\ast}},&\qquad\mathcal{K}^{\ast,-}\psi:=\mathfrak{L}^{\ast}(\mathcal{G}^{-}\psi)_{\cdot}\circ s_{\omega^{\ast}},\\[2mm]
\mathcal{K}_{M}^{\ast,+}u:=\mathfrak{L}_{M}^{\ast}(\mathcal{G}^{+}u)_{\cdot}\circ s_{\omega^{\ast}},&\qquad\mathcal{K}_{M}^{\ast,-}\psi:=\mathfrak{L}_{M}^{\ast}(\mathcal{G}^{-}\psi)_{\cdot}\circ s_{\omega^{\ast}}.
\end{array}
\end{equation}
\begin{lemma}\label{l_LrfdeLM}
Under \ref{T1}, \ref{T2}, \ref{T4}, \ref{N1}, \ref{N2}, \ref{N4} and \ref{N7}, there exist positive integers $\overline L$ and $\overline M$ such that, for every $L\geq\overline L$ and every $M\geq\overline M$, the initial value problem
\begin{equation}\label{LrfdetMLIVP}
\left\{\setlength\arraycolsep{0.1em}\begin{array}{l}
y'(t)=[\pi_{L}^{+}\rho_{L}^{+}\mathfrak{L}_{M}^{\ast}y_{\cdot}\circ s_{\omega^{\ast}}](t),\quad t\in[0,1],\\[2mm]
y_{0}=\psi
\end{array}
\right.
\end{equation}
for $\psi\in Y$ has a unique solution $y_{L,M}$. 
\end{lemma}
\begin{proof}
Set $u(t):=y'(t)$ for $t\in[0,1]$ and use $y=\mathcal{G}(u,\psi)$ according to \eqref{G2}. By virtue of \eqref{GKKM}, \eqref{LrfdetMLIVP} becomes
\begin{equation*}
u=\pi_{L}^{+}\rho_{L}^{+}\mathcal{K}_{M}^{\ast,+}u+\pi_{L}^{+}\rho_{L}^{+}\mathcal{K}_{M}^{\ast,-}\psi.
\end{equation*}
Well-posedness is thus equivalent to the invertibility of $I_{\mathbb{U}}-\pi_{L}^{+}\rho_{L}^{+}\mathcal{K}_{M}^{\ast,+}:\mathbb{U}\rightarrow\mathbb{U}$, for which we resort to the Banach's perturbation lemma, since the invertibility of $I_{\mathbb{U}}-\mathcal{K}^{\ast,+}:\mathbb{U}\rightarrow\mathbb{U}$ is guaranteed by the well-posedness of the initial value problem for \eqref{Lrfdet2} under \ref{T4}. The thesis follows by \eqref{pr+KM+toK+} in Lemma \ref{l_prKMtoK}.
\end{proof} 
\begin{lemma}\label{l_TLMtoT}
Under \ref{T1}, \ref{T2}, \ref{T4}, \ref{N1}, \ref{N2}, \ref{N4} and \ref{N7},
\begin{equation}\label{TLMtoT}
\lim_{L,M\rightarrow\infty}\|T_{L,M}^{\ast}(t,s)-T^{\ast}(t,s)\|_{Y\leftarrow Y}=0
\end{equation}
uniformly with respect to $t,s\in[0,1]$, $t\geq s$. If, in addition, $1\in\sigma(T^{\ast}(1,0))$ is simple with eigenfunction $\varphi$ normalized as $\|\varphi\|_{Y}=1$ and $r>0$ is such that $1$ is the only eigenvalue of $T^{\ast}(1,0)$ in $\overline B(1,r)\subset\mathbb{C}$, then there exist positive integers $\overline L$ and $\overline M$ such that, for every $L\geq\overline L$ and every $M\geq\overline M$, $T_{L,M}^{\ast}(1,0)$ has only a simple eigenvalue $\mu_{L,M}$ in $\overline B(1,r)$ and, moreover,
\begin{equation}\label{muLMphiLM}
\lim_{L,M\rightarrow\infty}|\mu_{L,M}-1|=0,\qquad\lim_{L,M\rightarrow\infty}\|\varphi_{L,M}-\varphi\|_{Y}=0,
\end{equation}
where $\varphi_{L,M}$ is the eigenfunction associated to $\mu_{L,M}$ normalized as $\|\varphi_{L,M}\|_{Y}=1$.
\end{lemma}
\begin{proof}
We give the proof for $s=0$, the extension to $s\in(0,1)$ being straightforward.

Let $\mathcal{G}(u,\psi)$ be the solution of \eqref{Lrfdet2} exiting from a given $\psi\in Y$, where $u$ satisfies $u=\mathfrak{L}^{\ast}\mathcal{G}(u,\psi)_{\cdot}\circ s_{\omega^{\ast}}$. Correspondingly, thanks to Lemma \ref{l_LrfdeLM}, let $\mathcal{G}(u_{L,M},\psi)$ be {\it the} solution of \eqref{LrfdetLM} exiting from the same $\psi$, where $u_{L,M}$ satisfies
\begin{equation*}
u_{L,M}=\pi_{L}^{+}\rho_{L}^{+}\mathfrak{L}_{M}^{\ast}\mathcal{G}(u_{L,M},\psi)_{\cdot}\circ s_{\omega^{\ast}}.
\end{equation*}
The relevant evolution operators are defined, for $t\in[0,1]$, respectively by
\begin{equation*}
T^{\ast}(t,0)\psi=\mathcal{G}(u,\psi)_{t},\qquad T_{L,M}^{\ast}(t,0)\psi=\mathcal{G}(u_{L,M},\psi)_{t}.
\end{equation*}
By recalling that $\mathcal{G}$ in \eqref{G2} is linear we get $T_{L,M}^{\ast}(t,0)\psi-T^{\ast}(t,0)\psi=\mathcal{G}(u_{L,M}-u,0)_{t}$. Therefore, \eqref{TLMtoT} is equivalent to showing that
\begin{equation}\label{eLM}
\lim_{L,M\rightarrow\infty}\|e_{L,M}\|_{\mathbb{U}}=0
\end{equation}
for $e_{L,M}:=u_{L,M}-u$. By using \eqref{GKKM} we have $u_{L,M}=\pi_{L}^{+}\rho_{L}^{+}\mathcal{K}_{M}^{\ast,+}u_{L,M}+\pi_{L}^{+}\rho_{L}^{+}\mathcal{K}_{M}^{\ast,-}\psi$ and $u=\mathcal{K}^{\ast,+}u+\mathcal{K}^{\ast,-}\psi$. Therefore $e_{L,M}=\pi_{L}^{+}\rho_{L}^{+}\mathcal{K}_{M}^{\ast,+}e_{L,M}+r_{L,M}^{+}+r_{L,M}^{-}$, where $r_{L,M}^{+}:=(\pi_{L}^{+}\rho_{L}^{+}\mathcal{K}_{M}^{\ast,+}-\mathcal{K}^{\ast,+})u$ and $r_{L,M}^{-}:=(\pi_{L}^{+}\rho_{L}^{+}\mathcal{K}_{M}^{\ast,-}-\mathcal{K}^{\ast,-})\psi$. We already showed in the proof of Lemma \ref{l_LrfdeLM} that $I_{\mathbb{U}}-\pi_{L}^{+}\rho_{L}^{+}\mathcal{K}_{M}^{\ast,+}$ is invertible through the Banach's perturbation lemma. By the latter it is also possible to show that $\|(I_{\mathbb{U}}-\pi_{L}^{+}\rho_{L}^{+}\mathcal{K}_{M}^{\ast,+})^{-1}\|_{\mathbb{U}\leftarrow\mathbb{U}}\leq2\|(I_{\mathbb{U}}-\mathcal{K}^{\ast,+})^{-1}\|_{\mathbb{U}\leftarrow\mathbb{U}}$ holds for $L$ and $M$ sufficiently large. Now \eqref{eLM} follows since both $\|r_{L,M}^{+}\|_{\mathbb{U}}\leq\|\pi_{L}^{+}\rho_{L}^{+}\mathcal{K}_{M}^{\ast,+}-\mathcal{K}^{\ast,+}\|_{\mathbb{U}\leftarrow\mathbb{U}}\|u\|_{\mathbb{U}}$ and
\begin{equation*}
\|r_{L,M}^{-}\|_{U}\leq\|\pi_{L}^{+}\rho_{L}^{+}\mathcal{K}_{M}^{\ast,-}-\mathcal{K}^{\ast,-}\|_{\mathbb{U}\leftarrow\mathbb{A}}\|\psi\|_{\mathbb{A}}
\end{equation*}
vanish by Lemma \ref{l_prKMtoK}.

The second part follows from standard results on spectral approximation of linear operators. In particular, \eqref{TLMtoT} implies strongly-stable convergence of $\mu I_{Y}-T_{L,M}^{\ast}(t,0)$ to $\mu I_{Y}-T^{\ast}(t,0)$ for every finite eigenvalue $\mu$ of $T^{\ast}(t,0)$ \cite[Example 3.8 and Theorem 5.22]{chat11} and the latter implies the final statement by \cite[Proposition 5.6 and Theorem 6.7]{chat11}.
\end{proof}

\bigskip
We are now in the position to prove the invertibility of $D\Psi_{L,M}(u^{\ast},\psi^{\ast},\omega^{\ast})$, which represents the first part of CS2 in \cite{mas15NM}. The second part is proved as the final result of this section.
\begin{proposition}\label{p_Ax*2LM}
Under \ref{T1}, \ref{T2}, \ref{T4}, \ref{N1}, \ref{N2}, \ref{N4}, \ref{N6} and \ref{N7}, there exist positive integers $\overline L$ and $\overline M$ such that, for every $L\geq\overline L$ and every $M\geq\overline M$, $D\Psi_{L,M}(u^{\ast},\psi^{\ast},\omega^{\ast})$ is invertible, i.e., for all $(u_{0},\psi_{0},\omega_{0})\in\mathbb{U}\times\mathbb{A}\times\mathbb{B}$ there exists a unique $(u_{L,M},\psi_{L,M},\omega_{L,M})\in\mathbb{U}\times\mathbb{A}\times\mathbb{B}$ such that
\begin{equation}\label{Ax*20LM}
\left\{\setlength\arraycolsep{0.1em}\begin{array}{l}
u_{L,M}=\pi_{L}^{+}\rho_{L}^{+}\mathfrak{L}_{M}^{\ast}\mathcal{G}(u_{L,M},\psi_{L,M})_{\cdot}\circ s_{\omega^{\ast}}+\omega_{L,M}\pi_{L}^{+}\rho_{L}^{+}\mathfrak{M}_{M}^{\ast}+u_{0}\\[2mm]
\psi_{L,M}=\pi_{L}^{-}\rho_{L}^{-}\mathcal{G}(u_{L,M},\psi_{L,M})_{1}+\psi_{0}\\[2mm]
p(\mathcal{G}(u_{L,M},\psi_{L,M})\vert_{[0,1]})=\omega_{0}.
\end{array}
\right.
\end{equation}
\end{proposition}
\begin{proof}
We follow the lines of the proof of Proposition \ref{p_Ax*2}. Thus let us treat \eqref{Ax*20LM} as an initial value problem for $v_{L,M}:=\mathcal{G}(u_{L,M},\psi_{L,M})$, i.e.,
\begin{equation}\label{vLM}
\left\{\setlength\arraycolsep{0.1em}\begin{array}{l}
v_{L,M}'(t)=(\pi_{L}^{+}\rho_{L}^{+}\mathfrak{L}_{M}^{\ast}v_{L,M,\cdot}\circ s_{\omega^{\ast}})(t)+\omega_{L,M}\pi_{L}^{+}\rho_{L}^{+}\mathfrak{M}_{M}^{\ast}(t)+u_{0}(t)\\[2mm]
v_{L,M,0}=\psi_{L,M}
\end{array}
\right.
\end{equation}
for $t\in[0,1]$\footnote{$v_{L,M,t}$ is a shortcut for $(v_{L,M})_{t}$, the latter according to \eqref{statet}.}. The variation of constants formula gives, for $t\in[0,1]$,
\begin{equation*}
v_{L,M,t}=T_{L,M}^{\ast}(t,0)\psi_{L,M}+\int_{0}^{t}[T_{L,M}^{\ast}(t,s)X_{0}][\omega_{L,M}\pi_{L}^{+}\rho_{L}^{+}\mathfrak{M}_{M}^{\ast}(s)+u_{0}(s)]\dd s,
\end{equation*}
and the first boundary condition in \eqref{Ax*20LM} returns
\begin{equation*}
\setlength\arraycolsep{0.1em}\begin{array}{rcl}
\psi_{L,M}&=&\pi_{L}^{-}\rho_{L}^{-}T_{L,M}^{\ast}(1,0)\psi_{L,M}\\[2mm]
&&\displaystyle+\pi_{L}^{-}\rho_{L}^{-}\int_{0}^{1}[T_{L,M}^{\ast}(1,s)X_{0}][\omega_{L,M}\pi_{L}^{+}\rho_{L}^{+}\mathfrak{M}_{M}^{\ast}(s)+u_{0}(s)]\dd s+\psi_{0}.
\end{array}
\end{equation*}
For $L$ and $M$ sufficiently large, let $\mu_{L,M}$ be the multiplier of $T_{L,M}^{\ast}(1,0)$ in Lemma \ref{l_TLMtoT} and rewrite the last equation as
\begin{equation}\label{Ax*22LM}
\setlength\arraycolsep{0.1em}\begin{array}{rcl}
\mu_{L,M}\psi_{L,M}&=&T_{L,M}^{\ast}(1,0)\psi_{L,M}\\[2mm]
&&\displaystyle+\pi_{L}^{-}\rho_{L}^{-}\int_{0}^{1}[T_{L,M}^{\ast}(1,s)X_{0}][\omega_{L,M}\pi_{L}^{+}\rho_{L}^{+}\mathfrak{M}_{M}^{\ast}(s)+u_{0}(s)]\dd s+\psi_{0}+\nu_{L,M}
\end{array}
\end{equation}
for
\begin{equation}\label{nuLM}
\nu_{L,M}:=(\pi_{L}^{-}\rho_{L}^{-}-I_{\mathbb{A}})T_{L,M}^{\ast}(1,0)\psi_{L,M}+(\mu_{L,M}-1)\psi_{L,M},
\end{equation}
where we note that under \ref{T2} $T_{L,M}^{\ast}(1,0)\psi_{L,M}=\mathcal{G}(u_{L,M},\psi_{L,M})_{1}\in\mathbb{A}$ since $\mathcal{G}(\mathbb{U},\mathbb{A})$$\subseteq\mathbb{V}$ as already observed after Proposition \ref{p_G}.

Now we have
\begin{equation}\label{decoLM}
Y=R_{L,M}\oplus K_{L,M}
\end{equation}
for $R_{L,M}$ and $K_{L,M}$ the range and the kernel of $\mu_{L,M}I_{Y}-T_{L,M}^{\ast}(1,0)$, respectively. Since $\mu_{L,M}$ is simple, we can set $K_{L,M}=\Span\{\varphi_{L,M}\}$ for $\varphi_{L,M}$ an eigenfunction of the multiplier $\mu_{L,M}$. We can also assume  $p(v(\cdot;\varphi_{L,M})|_{[0,1]})\neq0$ for $v(\cdot;\varphi_{L,M})$ the solution of \eqref{vLM} exiting from $\varphi_{L,M}$ thanks to the linearity of $p$ and to the second of \eqref{muLMphiLM} in Lemma \ref{l_TLMtoT} (recall Remark \ref{r_pfi0}).

From \eqref{Ax*22LM} let us define the elements of $Y$
\begin{equation}\label{xiLM}
\setlength\arraycolsep{0.1em}\begin{array}{rcl}
\xi_{L,M,1}^{\ast}&:=&\displaystyle\pi_{L}^{-}\rho_{L}^{-}\int_0^{1}[T_{L,M}^{\ast}(1,s)X_{0}]\pi_{L}^{+}\rho_{L}^{+}\mathfrak{M}_{M}^{\ast}(s)\dd s,\\[2mm]
\xi_{L,M,2}^{\ast}&:=&\displaystyle\pi_{L}^{-}\rho_{L}^{-}\int_0^{1}[T_{L,M}^{\ast}(1,s)X_{0}]u_{0}(s)\dd s+\psi_{0},
\end{array}
\end{equation}
so that \eqref{Ax*22LM} becomes
\begin{equation}\label{Ax*23LM}
[\mu_{L,M}I_{Y}-T_{L,M}^{\ast}(1,0)]\psi_{L,M}=\omega_{L,M}\xi_{L,M,1}^{\ast}+\xi_{L,M,2}^{\ast}+\nu_{L,M}.
\end{equation}
From \eqref{decoLM} we can write uniquely
\begin{equation}\label{xinuLM}
\setlength\arraycolsep{0.1em}\begin{array}{rcl}
\xi_{L,M,1}^{\ast}&=&r_{L,M,1}+k_{L,M,1}\varphi_{L,M},\\[2mm]
\xi_{L,M,2}^{\ast}&=&r_{L,M,2}+k_{L,M,2}\varphi_{L,M},\\[2mm]
\nu_{L,M}&=&s_{L,M}+h_{L,M}\varphi_{L,M}
\end{array}
\end{equation}
for $r_{L,M,1},r_{L,M,2},s_{L,M}\in R_{L,M}$ and $k_{L,M,1},k_{L,M,2},h_{L,M}\in\mathbb{R}$. Then from \eqref{Ax*23LM} it must be $\omega_{L,M}\xi_{L,M,1}^{\ast}+\xi_{L,M,2}^{\ast}+\nu_{L,M}\in R_{L,M}$, which implies $\omega_{L,M}k_{L,M,1}+k_{L,M,2}+h_{L,M}=0$. Therefore, by assuming $k_{L,M,1}\neq 0$, it follows that
\begin{equation}\label{omegaLM}
\omega_{L,M}=-\frac{k_{L,M,2}+h_{L,M}}{k_{L,M,1}}
\end{equation}
is the only possible solution. Eventually, let $\eta_{L,M}$ be such that
\begin{equation*}
[\mu_{L,M}I_{Y}-T_{L,M}^{\ast}(1,0)]\eta_{L,M}=\omega_{L,M}\xi_{L,M,1}^{\ast}+\xi_{L,M,2}^{\ast}+\nu_{L,M}.
\end{equation*}
Then, every $\psi_{L,M}$ satisfying \eqref{Ax*23LM} can be written as $\eta_{L,M}+\lambda_{L,M}\varphi_{L,M}$ for some $\lambda_{L,M}\in\mathbb{R}$. The value of the latter can be fixed uniquely by imposing the phase condition, i.e., $p(v(\cdot;\eta_{L,M})|_{[0,1]})+\lambda_{L,M}p(v(\cdot;\varphi_{L,M})|_{[0,1]})=\omega_{0}$.

It is left to prove that $k_{L,M,1}\neq 0$ for $L$ and $M$ sufficiently large. By \eqref{k1LMtok1} in Proposition \ref{p_Ax*2LM-Ax*2}, $k_{L,M,1}\rightarrow k_{1}$ for $k_{1}$ in the proof of Proposition \ref{p_Ax*2}. As the latter is proved to be different from $0$ in Section \ref{s_nongeneric}, the same holds for $k_{L,M,1}$ for $L$ and $M$ sufficiently large.
\end{proof}

Eventually, to complete the proof of the validity of CS2 in Proposition \ref{p_CS2b} below, we show next that the inverse of $D\Psi_{L,M}(u^{\ast},\psi^{\ast},\omega^{\ast})$ is bounded uniformly with respect to $L$ and $M$.
\begin{lemma}\label{l_DPsiLMinv}
Under \ref{T1}, \ref{T2}, \ref{T4}, \ref{N1}, \ref{N2}, \ref{N4}, \ref{N6} and \ref{N7}, the inverse of $D\Psi_{L,M}(u^{\ast},\psi^{\ast},\omega^{\ast})$ is uniformly bounded with respect to both $L$ and $M$.
\end{lemma}
\begin{proof}
Proposition \ref{p_Ax*2LM} guarantees that, given $(u_{0},\psi_{0},\omega_{0})\in\mathbb{U}\times\mathbb{A}\times\mathbb{B}$, there exists a unique $(u_{L,M},\psi_{L,M},\omega_{L,M})\in\mathbb{U}\times\mathbb{A}\times\mathbb{B}$ satisfying
\begin{equation}\label{DPdiinv}
D\Psi_{L,M}(u^{\ast},\psi^{\ast},\omega^{\ast})(u_{L,M},\psi_{L,M},\omega_{L,M})=(u_{0},\psi_{0},\omega_{0}).
\end{equation}
We thus need to show that $\|(u_{L,M},\psi_{L,M},\omega_{L,M})\|_{\mathbb{U}\times\mathbb{A}\times\mathbb{B}}$ is bounded uniformly with respect to both $L$ and $M$. To this aim we prove that $(u_{L,M},\psi_{L,M},\omega_{L,M})$ is related to the solution of the collocation of (the secondary discretization of) an equivalent version of \eqref{Ax*20} according to the primary discretization under \ref{N1} and \ref{N2}. Indeed, we first need to rearrange the terms of \eqref{Ax*20} to give a proper sense to the collocation problem since, in general, $u$ is not continuous therein (because of $u_{0}$), while the range of $\pi_{L}^{+}\rho_{L}^{+}$ contains only continuous functions, Remark \ref{r_rangePR}. Consider then
\begin{equation}\label{Ax*23}
\left\{\setlength\arraycolsep{0.1em}\begin{array}{l}
z=\mathfrak{L}^{\ast}\mathcal{G}(z,\gamma)_{\cdot}\circ s_{\omega^{\ast}}+\omega\mathfrak{M}^{\ast}+\mathfrak{L}^{\ast}\mathcal{G}(u_{0},\psi_{0})_{\cdot}\circ s_{\omega^{\ast}}\\[2mm]
\gamma=\mathcal{G}(z,\gamma)_{1}+\mathcal{G}(u_{0},\psi_{0})_{1}\\[2mm]
p(\mathcal{G}(z,\gamma)\vert_{[0,1]})=\omega_{0}-p(\mathcal{G}(u_{0},\psi_{0})\vert_{[0,1]})
\end{array}
\right.
\end{equation}
obtained from \eqref{Ax*20} by setting $z:=u-u_{0}$ and $\gamma:=\psi-\psi_{0}$. Let us observe that $z$ is continuous as it follows from the first equation in \eqref{Ax*23} under \ref{T4}. Similarly, we rewrite \eqref{Ax*20LM} as
\begin{equation}\label{Ax*24LM}
\left\{\setlength\arraycolsep{0.1em}\begin{array}{l}
z_{L,M}=\pi_{L}^{+}\rho_{L}^{+}\mathfrak{L}_{M}^{\ast}\mathcal{G}(z_{L,M},\gamma_{L,M})_{\cdot}\circ s_{\omega^{\ast}}+\omega_{L,M}\pi_{L}^{+}\rho_{L}^{+}\mathfrak{M}_{M}^{\ast}\\[2mm]
\hspace{12mm}+\pi_{L}^{+}\rho_{L}^{+}\mathfrak{L}_{M}^{\ast}\mathcal{G}(u_{0},\psi_{0})_{\cdot}\circ s_{\omega^{\ast}}\\[2mm]
\gamma_{L,M}=\pi_{L}^{-}\rho_{L}^{-}\mathcal{G}(z_{L,M},\gamma_{L,M})_{1}+\pi_{L}^{-}\rho_{L}^{-}\mathcal{G}(u_{0},\psi_{0})_{1}\\[2mm]
p(\mathcal{G}(z_{L,M},\gamma_{L,M})\vert_{[0,1]})=\omega_{0}-p(\mathcal{G}(u_{0},\psi_{0})\vert_{[0,1]})
\end{array}
\right.
\end{equation}
for $z_{L,M}:=u_{L,M}-u_{0}$ and $\gamma_{L,M}:=\psi_{L,M}-\psi_{0}$. It follows that
\begin{equation}\label{uLMpsiLM}
u_{L,M}=e_{L,M}^{+}+u,\qquad\psi_{L,M}=e_{L,M}^{-}+\psi,
\end{equation}
where $e_{L,M}^{+}:=z_{L,M}-z$ and $e_{L,M}^{-}:=\gamma_{L,M}-\gamma$ are the collocation errors of the components in $\mathbb{U}$ and $\mathbb{A}$, respectively, given that $(z_{L,M},\gamma_{L,M},\omega_{L,M})$ is the collocation solution of the secondary discretization of \eqref{Ax*23} according to \ref{N1} and \ref{N2}. By subtracting \eqref{Ax*23} from \eqref{Ax*24LM} we get
\begin{equation}\label{eLM0}
\left\{\setlength\arraycolsep{0.1em}\begin{array}{l}
e_{L,M}^{+}=\pi_{L}^{+}\rho_{L}^{+}\mathfrak{L}_{M}^{\ast}\mathcal{G}(e_{L,M}^{+},e_{L,M}^{-})_{\cdot}\circ s_{\omega^{\ast}}+\varepsilon_{\omega,L,M}+\varepsilon_{L,M}^{+}\\[2mm]
e_{L,M}^{-}=\pi_{L}^{-}\rho_{L}^{-}\mathcal{G}(e_{L,M}^{+},e_{L,M}^{-})_{1}+\varepsilon_{L,M}^{-}\\[2mm]
p(\mathcal{G}(e_{L,M}^{+},e_{L,M}^{-})\vert_{[0,1]})=0
\end{array}
\right.
\end{equation}
for
\begin{equation}\label{epsLM}
\setlength\arraycolsep{0.1em}\begin{array}{rcl}
\varepsilon_{\omega,L,M}&:=&\omega_{L,M}\pi_{L}^{+}\rho_{L}^{+}\mathfrak{M}_{M}^{\ast}-\omega\mathfrak{M}^{\ast},\\[2mm]
\varepsilon_{L,M}^{+}&:=&\pi_{L}^{+}\rho_{L}^{+}\mathfrak{L}_{M}^{\ast}\mathcal{G}(u_{0},\psi_{0})_{\cdot}\circ s_{\omega^{\ast}}-\mathfrak{L}^{\ast}\mathcal{G}(u_{0},\psi_{0})_{\cdot}\circ s_{\omega^{\ast}},\\[2mm]
\varepsilon_{L,M}^{-}&:=&(\pi_{L}^{-}\rho_{L}^{-}-I_{\mathbb{A}})\mathcal{G}(u_{0},\psi_{0})_{1}.
\end{array}
\end{equation}
By using \eqref{GKKM} we rewrite the first two equations of \eqref{eLM0} as
\begin{equation}\label{eLM1}
\left\{\setlength\arraycolsep{0.1em}\begin{array}{l}
e_{L,M}^{+}=\pi_{L}^{+}\rho_{L}^{+}\mathcal{K}_{M}^{\ast,+}e_{L,M}^{+}+\pi_{L}^{+}\rho_{L}^{+}\mathcal{K}_{M}^{\ast,-}e_{L,M}^{-}+\varepsilon_{\omega,L,M}+\varepsilon_{L,M}^{+}\\[2mm]
e_{L,M}^{-}=\pi_{L}^{-}\rho_{L}^{-}\mathcal{G}_{1}^{+}e_{L,M}^{+}+\pi_{L}^{-}\rho_{L}^{-}\mathcal{G}_{1}^{-}e_{L,M}^{-}+\varepsilon_{L,M}^{-},
\end{array}
\right.
\end{equation}
where we also used $\mathcal{G}(e_{L,M}^{+},e_{L,M}^{-})_{1}=\mathcal{G}_{1}^{+}e_{L,M}^{+}+\mathcal{G}_{1}^{-}e_{L,M}^{-}$ for
\begin{equation}\label{G1+-}
(\mathcal{G}_{1}^{+}e_{L,M}^{+})(t):=\int_{0}^{1+t}e_{L,M}^{+}(s)\dd s,\qquad(\mathcal{G}_{1}^{-}e_{L,M}^{-})(t)=e_{L,M}^{-}(0)
\end{equation}
and $t\in[-1,0]$ according to the definition of $\mathcal{G}$ in \eqref{G2}. Allowing for a block-wise definition of operators in $\mathbb{U}\times\mathbb{A}$, which should be self-explaining in the following, \eqref{eLM1} becomes
\begin{equation*}
\begin{pmatrix}
e_{L,M}^{+}\\[2mm]
e_{L,M}^{-}
\end{pmatrix}
=
\begin{pmatrix}
\pi_{L}^{+}\rho_{L}^{+}\mathcal{K}_{M}^{\ast,+}&\pi_{L}^{+}\rho_{L}^{+}\mathcal{K}_{M}^{\ast,-}\\[2mm]
\pi_{L}^{-}\rho_{L}^{-}\mathcal{G}_{1}^{+}&\pi_{L}^{-}\rho_{L}^{-}\mathcal{G}_{1}^{-}
\end{pmatrix}
\begin{pmatrix}
e_{L,M}^{+}\\[2mm]
e_{L,M}^{-}
\end{pmatrix}
+
\begin{pmatrix}
\varepsilon_{\omega,L,M}+\varepsilon_{L,M}^{+}\\[2mm]
\varepsilon_{L,M}^{-}
\end{pmatrix}.
\end{equation*}
Now we look for a bound on the collocation error, and we are allowed to search for a bound on $\|(e_{L,M}^{+},e_{L,M}^{-})\|_{C([0,1],\mathbb{R}^{d})\times\mathbb{A}}$. Indeed, it is crucial to observe that $e_{L,M}^{+}$ is continuous since so is $z$ as already observed and $z_{L,M}\in\Pi_{L,m}^{+}$. Moreover, also the first two in \eqref{epsLM} are continuous under \ref{T4} and \ref{N4}. Let us also set for brevity
\begin{equation}\label{C+}
C^{+}:=C([0,1],\mathbb{R}^{d}).
\end{equation}

Note that existence and uniqueness of $(e_{L,M}^{+},e_{L,M}^{-})$ follows already from Propositions \ref{p_Ax*2LM} and \ref{p_Ax*2}, so that 
the invertibility of the operator
\begin{equation*}
\begin{pmatrix}
I_{C^{+}}&0\\[2mm]
0&I_{\mathbb{A}}
\end{pmatrix}
-
\begin{pmatrix}
\pi_{L}^{+}\rho_{L}^{+}\mathcal{K}_{M}^{\ast,+}&\pi_{L}^{+}\rho_{L}^{+}\mathcal{K}_{M}^{\ast,-}\\[2mm]
\pi_{L}^{-}\rho_{L}^{-}\mathcal{G}_{1}^{+}&\pi_{L}^{-}\rho_{L}^{-}\mathcal{G}_{1}^{-}
\end{pmatrix}:C^{+}\times\mathbb{A}\rightarrow C^{+}\times\mathbb{A}
\end{equation*}
is already proved. Anyway, if we manage to prove that
\begin{equation*}
\lim_{L,M\rightarrow\infty}
\left\|
\begin{pmatrix}
\pi_{L}^{+}\rho_{L}^{+}\mathcal{K}_{M}^{\ast,+}&\pi_{L}^{+}\rho_{L}^{+}\mathcal{K}_{M}^{\ast,-}\\[2mm]
\pi_{L}^{-}\rho_{L}^{-}\mathcal{G}_{1}^{+}&\pi_{L}^{-}\rho_{L}^{-}\mathcal{G}_{1}^{-}
\end{pmatrix}
-
\begin{pmatrix}
\mathcal{K}^{\ast,+}&\mathcal{K}^{\ast,-}\\[2mm]
\mathcal{G}_{1}^{+}&\mathcal{G}_{1}^{-}
\end{pmatrix}
\right\|_{C^{+}\times\mathbb{A}\leftarrow C^{+}\times\mathbb{A}}=0,
\end{equation*}
then we can apply the Banach's perturbation lemma to recover the bound
\begin{equation}\label{boundeLM}
\setlength\arraycolsep{0.1em}\begin{array}{l}
\left\|
\left[
\begin{pmatrix}
I_{C^{+}}&0\\[2mm]
0&I_{\mathbb{A}}
\end{pmatrix}
-
\begin{pmatrix}
\pi_{L}^{+}\rho_{L}^{+}\mathcal{K}_{M}^{\ast,+}&\pi_{L}^{+}\rho_{L}^{+}\mathcal{K}_{M}^{\ast,-}\\[2mm]
\pi_{L}^{-}\rho_{L}^{-}\mathcal{G}_{1}^{+}&\pi_{L}^{-}\rho_{L}^{-}\mathcal{G}_{1}^{-}
\end{pmatrix}
\right]^{-1}
\right\|_{C^{+}\times\mathbb{A}\leftarrow C^{+}\times\mathbb{A}}\\[4mm]
\qquad\leq2
\left\|
\left[
\begin{pmatrix}
I_{C^{+}}&0\\[2mm]
0&I_{\mathbb{A}}
\end{pmatrix}
-
\begin{pmatrix}
\mathcal{K}^{\ast,+}&\mathcal{K}^{\ast,-}\\[2mm]
\mathcal{G}_{1}^{+}&\mathcal{G}_{1}^{-}
\end{pmatrix}
\right]^{-1}
\right\|_{C^{+}\times\mathbb{A}\leftarrow C^{+}\times\mathbb{A}},
\end{array}
\end{equation}
for sufficiently large $L$ and $M$, which is also uniform with respect to both $L$ and $M$. Indeed, from Proposition \ref{p_Ax*2} we already know that the operator
\begin{equation*}
\begin{pmatrix}
I_{C^{+}}&0\\[2mm]
0&I_{\mathbb{A}}
\end{pmatrix}
-
\begin{pmatrix}
\mathcal{K}^{\ast,+}&\mathcal{K}^{\ast,-}\\[2mm]
\mathcal{G}_{1}^{+}&\mathcal{G}_{1}^{-}
\end{pmatrix}:C^{+}\times\mathbb{A}\rightarrow C^{+}\times\mathbb{A}
\end{equation*}
is invertible. \eqref{pr+KM+toK+} in Lemma \ref{l_prKMtoK} holds also if we replace $\mathbb{U}$ with $C^{+}$ since the norm is the same. The same holds for \eqref{pr+KM-toK-}. Therefore, thanks to Lemma \ref{l_prG1toG}, \eqref{boundeLM} holds and we obtain
\begin{equation*}
\|(e_{L,M}^{+},e_{L,M}^{-})\|_{C^{+}\times\mathbb{A}}\leq\kappa\|(\varepsilon_{\omega,L,M}+\varepsilon_{L,M}^{+},\varepsilon_{L,M}^{-})\|_{C^{+}\times\mathbb{A}}
\end{equation*}
for some constant $\kappa$ independent of $L$ and $M$. Above, from \eqref{epsLM} we have
\begin{equation}\label{epsLM+}
\varepsilon_{L,M}^{+}=\pi_{L}^{+}\rho_{L}^{+}(\mathfrak{L}_{M}^{\ast}-\mathfrak{L}^{\ast})\mathcal{G}(u_{0},\psi_{0})_{\cdot}\circ s_{\omega^{\ast}}+(\pi_{L}^{+}\rho_{L}^{+}-I_{\mathbb{U}})\mathfrak{L}^{\ast}\mathcal{G}(u_{0},\psi_{0})_{\cdot}\circ s_{\omega^{\ast}},
\end{equation}
so that $\varepsilon_{L,M}^{+}$ vanishes as $L,M\rightarrow\infty$ under \ref{T4} and \ref{N7} by \eqref{p+r+} of Lemma \ref{l_p+r+-I} in Appendix \ref{s_basic} and \eqref{LMtoL} of Lemma \ref{l_prLMMtoLM} (the first addend) and by \eqref{p+r+} of Lemma \ref{l_p+r+-I} again (the second addend). On the other hand, $\varepsilon_{L,M}^{-}$ does not necessarily vanish but is anyway bounded uniformly with respect to $L$ and $M$ since $u_{0}$ is the derivative of $\mathcal{G}(u_{0},\psi_{0})_{1}$ and, as such, it is not necessarily continuous, even though it is bounded. Consequently, it is not difficult to argue that $\|\varepsilon_{L,M}^{-}\|_{\mathbb{A}}\leq\|\psi_{0}\|_{\mathbb{A}}+2\|u_{0}\|_{\mathbb{U}}$ by taking into account for possible jumps in $u_{0}$. It is left to prove that $\varepsilon_{\omega,L,M}$ either vanishes or is uniformly bounded. From \eqref{epsLM} we have
\begin{equation*}
\varepsilon_{\omega,L,M}=\omega_{L,M}\pi_{L}^{+}\rho_{L}^{+}(\mathfrak{M}_{M}^{\ast}-\mathfrak{M}^{\ast})+\omega_{L,M}(\pi_{L}^{+}\rho_{L}^{+}-I_{\mathbb{U}})\mathfrak{M}^{\ast}+(\omega_{L,M}-\omega)\mathfrak{M}^{\ast},
\end{equation*}
in which the third addend at the right-hand side vanishes since $\omega_{L,M}\rightarrow\omega$ thanks to Proposition \ref{p_Ax*2LM-Ax*2} and, therefore, the first and the second addends vanish as well since $\omega_{L,M}$ is uniformly bounded, thanks to the same arguments adopted for \eqref{epsLM+} under \ref{T4}, \ref{N6} and \ref{N7} and also thanks to \eqref{p+r+} in Lemma \ref{l_p+r+-I} in Appendix \ref{s_basic}.

In the proof of Proposition \ref{p_Ax*2LM-Ax*2}, it is also shown that $\psi_{L,M}$ is bounded uniformly in both $L$ and $M$. Finally, we obtain that $\|(u_{L,M},\psi_{L,M},\omega_{L,M})\|_{\mathbb{U}\times\mathbb{A}\times\mathbb{B}}$ is bounded uniformly in both $L$ and $M$ thanks to \eqref{uLMpsiLM} and Proposition \ref{p_Ax*2}.
\end{proof}

\bigskip
We now conclude by proving the validity of the second part of CS2 in \cite{mas15NM}.
\begin{proposition}\label{p_CS2b}
Under \ref{T1}, \ref{T2}, \ref{T4}, \ref{N1}, \ref{N2}, \ref{N4}, \ref{N6} and \ref{N7},
\begin{equation}\label{DPsiLMlim}
\setlength\arraycolsep{0.1em}\begin{array}{rcl}
\lim_{L,M\rightarrow\infty}&&\displaystyle\frac{1}{r_{2}(L,M)}\|[D\Psi_{L,M}(u^{\ast},\psi^{\ast},\omega^{\ast})]^{-1}\|_{\mathbb{U}\times\mathbb{A}\times\mathbb{B}\leftarrow\mathbb{U}\times\mathbb{A}\times\mathbb{B}}\\[3mm]
&&\cdot\|\Psi_{L,M}(u^{\ast},\psi^{\ast},\omega^{\ast})\|_{\mathbb{U}\times\mathbb{A}\times\mathbb{B}}=0,
\end{array}
\end{equation}
where
\begin{equation*}
r_{2}(L,M):=\min\left\{r_{1},\frac{1}{2\kappa\|[D\Psi_{L,M}(u^{\ast},\psi^{\ast},\omega^{\ast})]^{-1}\|_{\mathbb{U}\times\mathbb{A}\times\mathbb{B}\leftarrow\mathbb{U}\times\mathbb{A}\times\mathbb{B}}}\right\}
\end{equation*}
with $r_{1}$ and $\kappa$ as in Proposition \ref{p_CS1}.
\end{proposition}
\begin{proof}
Thanks to Lemma \ref{l_DPsiLMinv} and to the fact that $r_{1}$ and $\kappa$ in Proposition \ref{p_CS1} are independent of $L$ and $M$ (recall indeed Remark \ref{r_KLM}), it remains to prove that the norm of $\Psi_{L,M}(u^{\ast},\psi^{\ast},\omega^{\ast})$ vanishes. We have
\begin{equation}\label{PsiLMbound}
\setlength\arraycolsep{0.1em}\begin{array}{rcl}
\|\Psi_{L,M}(u^{\ast},\psi^{\ast},\omega^{\ast})\|_{\mathbb{U}\times\mathbb{A}\times\mathbb{B}}\leq\|(I_{\mathbb{U}\times\mathbb{A}\times\mathbb{B}}&-&P_{L}R_{L})(u^{\ast},\psi^{\ast},\omega^{\ast})\|_{\mathbb{U}\times\mathbb{A}\times\mathbb{B}}\\[2mm]
+\|P_{L}R_{L}[\Phi_{M}(u^{\ast},\psi^{\ast},\omega^{\ast})&-&\Phi(u^{\ast},\psi^{\ast},\omega^{\ast})]\|_{\mathbb{U}\times\mathbb{A}\times\mathbb{B}}
\end{array}
\end{equation}
since $\Phi(u^{\ast},\psi^{\ast},\omega^{\ast})=(u^{\ast},\psi^{\ast},\omega^{\ast})$. The second addend in the right-hand side above vanishes under \ref{N6} and \ref{N7} and thanks to \eqref{PLRL} of Corollary \ref{c_PLRL} in Appendix \ref{s_basic}. The first addend vanishes as well by Lemma \ref{l_v*'}, which shows in particular that $u^{\ast}$ and ${\psi^{\ast}}'$ are continuous.
\end{proof}
\section{Convergence analysis}
\label{s_convergence}
In Sections \ref{s_abstract} and \ref{s_discretization} we have proved the validity of all the assumptions needed to use the method in \cite{mas15NM} as applied to \eqref{bvp2}. In this section, we first state two theorems which eventually ensure the convergence of the method in view of Remark \ref{r_FEMSEM}. For their proof we refer to the corresponding \cite[Theorems 1 and 2]{mas15NM}. Then we comment about the rate of convergence, which is elaborated in subsequent sections. {\color{black} Section \ref{s_practical} includes some observations on the requirements \ref{T3}-\ref{T5} and \ref{N3}-\ref{N7} for concrete, specific instances of the right-hand side. }Finally we give a brief account of the SEM in Section \ref{s_SEM}.

\bigskip
\begin{theorem}[\protect{\cite[Theorem 1, page 538]{mas15NM}}]\label{t1}
Under \ref{T1}, \ref{T2}, \ref{T4}, \ref{N1}, \ref{N2}, \ref{N4}, \ref{N5}, \ref{N6} and \ref{N7}, there exists a positive integer $\overline{N}$ such that, for every $L,M\geq \overline{N}$, $P_{L}R_{L}\Phi_{M}$ has a unique fixed point $(\tilde{u}_{L,M}^{\ast},\tilde{\psi}_{L,M}^{\ast},\tilde{\omega}_{L,M}^{\ast})$ in $\overline{B}((u^{\ast},\psi^{\ast},\omega^{\ast}),$ $r_{2}(L,M))$ and
\begin{equation*}
\setlength\arraycolsep{0.1em}\begin{array}{rcl}
\|(\tilde{u}_{L,M}^{\ast},\tilde{\psi}_{L,M}^{\ast},\tilde{\omega}_{L,M}^{\ast})&-&(u^{\ast},\psi^{\ast},\omega^{\ast})\|_{\mathbb{U}\times\mathbb{A}\times\mathbb{B}}\\[2mm]
&\leq&2\|[D\Psi_{L,M}(u^{\ast},\psi^{\ast},\omega^{\ast})]^{-1}\|_{\mathbb{U}\times\mathbb{A}\times\mathbb{B}\leftarrow\mathbb{U}\times\mathbb{A}\times\mathbb{B}}\\[2mm]
&&\cdot\|\Psi_{L,M}(u^{\ast},\psi^{\ast},\omega^{\ast})\|_{\mathbb{U}\times\mathbb{A}\times\mathbb{B}}
\end{array}
\end{equation*}
holds for $r_{2}(L,M)$ defined as in Proposition \ref{p_CS2b}. Moreover, we have the expansion
\begin{equation*}
\setlength\arraycolsep{0.1em}\begin{array}{rcl}
(\tilde{u}_{L,M}^{\ast},\tilde{\psi}_{L,M}^{\ast},\tilde{\omega}_{L,M}^{\ast})&-&(u^{\ast},\psi^{\ast},\omega^{\ast})\\[2mm]
&=&-[D\Psi_{L,M}(u^{\ast},\psi^{\ast},\omega^{\ast})]^{-1}\Psi_{L,M}(u^{\ast},\psi^{\ast},\omega^{\ast})+\delta_{L,M},
\end{array}
\end{equation*}
where
\begin{equation*}
\setlength\arraycolsep{0.1em}\begin{array}{rcl}
\|\delta_{L,M}\|_{\mathbb{U}\times\mathbb{A}\times\mathbb{B}}&\leq&4\kappa\cdot\|[D\Psi_{L,M}(u^{\ast},\psi^{\ast},\omega^{\ast})]^{-1}\|^3_{\mathbb{U}\times\mathbb{A}\times\mathbb{B}\leftarrow\mathbb{U}\times\mathbb{A}\times\mathbb{B}}\\[2mm]
&&\cdot\|\Psi_{L,M}(u^{\ast},\psi^{\ast},\omega^{\ast})\|^2_{\mathbb{U}\times\mathbb{A}\times\mathbb{B}}
\end{array}
\end{equation*}
for $\kappa$ defined as in Proposition \ref{p_CS1}.
\end{theorem}
\begin{theorem}[\protect{\cite[Theorem 2, page 539]{mas15NM}}] Under \ref{T1}, \ref{T2}, \ref{T4}, \ref{N1}, \ref{N2}, \ref{N4}, \ref{N5}, \ref{N6} and \ref{N7}, there exists a positive integer $\hat{N}$ such that, for all $L,M\geq\hat{N}$, the operator $R_{L}\Phi_{M}P_{L}$ has a fixed point $(u_{L,M}^{\ast},\psi_{L,M}^{\ast},\omega_{L,M}^{\ast})$ and
\begin{equation*}
\setlength\arraycolsep{0.1em}\begin{array}{rcl}
\|P_{L}(u_{L,M}^{\ast},\psi_{L,M}^{\ast},\omega_{L,M}^{\ast})&-&(u^{\ast},\psi^{\ast},\omega^{\ast})\|_{\mathbb{U}\times\mathbb{A}\times\mathbb{B}}\\[2mm]
&\leq2&\|[D\Psi_{L,M}(u^{\ast},\psi^{\ast},\omega^{\ast})]^{-1}\|_{\mathbb{U}\times\mathbb{A}\times\mathbb{B}\leftarrow\mathbb{U}\times\mathbb{A}\times\mathbb{B}}\\[2mm]
&&\cdot\|\Psi_{L,M}(u^{\ast},\psi^{\ast},\omega^{\ast})\|_{\mathbb{U}\times\mathbb{A}\times\mathbb{B}}
\end{array}
\end{equation*}
and
\begin{equation*}
\setlength\arraycolsep{0.1em}\begin{array}{rcl}
P_{L}(u_{L,M}^{\ast},\psi_{L,M}^{\ast},\omega_{L,M}^{\ast})&-&(u^{\ast},\psi^{\ast},\omega^{\ast})\\[2mm]
&=&-[D\Psi_{L,M}(u^{\ast},\psi^{\ast},\omega^{\ast})]^{-1}\Psi_{L,M}(u^{\ast},\psi^{\ast},\omega^{\ast})+\delta_{L,M},
\end{array}
\end{equation*}
where $\delta_{L,M}$ is bounded as in Theorem \ref{t1}. Moreover, if $(\hat u_{L,M}^{\ast},\hat\psi_{L,M}^{\ast},\hat\omega_{L,M}^{\ast})$ is another fixed point of $R_{L}\Phi_{M}P_{L}$, then
\begin{equation*}
\|P_{L}(\hat u_{L,M}^{\ast},\hat\psi_{L,M}^{\ast},\hat\omega_{L,M}^{\ast})-(u^{\ast},\psi^{\ast},\omega^{\ast})\|_{\mathbb{U}\times\mathbb{A}\times\mathbb{B}}>r_{2}(L,M)
\end{equation*}
and
\begin{equation*}
\setlength\arraycolsep{0.1em}\begin{array}{rcl}
\|(\hat u_{L,M}^{\ast},\hat\psi_{L,M}^{\ast},\hat\omega_{L,M}^{\ast})&-&(u_{L,M}^{\ast},\psi_{L,M}^{\ast},\omega_{L,M}^{\ast})\|_{\mathbb{U}_{L}\times\mathbb{A}_{L}\times\mathbb{B}}>\frac{r_{2}(L,M)}{2\cdot\max\{\|\pi_{L}^{+}\|_{\mathbb{U}\leftarrow\mathbb{U}_{L}},\|\pi_{L}^{-}\|_{\mathbb{A}\leftarrow\mathbb{A}_{L}},1\}}
\end{array}
\end{equation*}
for $r_{2}(L,K)$ defined as in Proposition \ref{p_CS2b}. Finally,
\begin{equation}\label{converr}
\setlength\arraycolsep{0.1em}\begin{array}{rcl}
\|(v_{L,M}^{\ast},\omega_{L,M}^{\ast})&-&(v^{\ast},\omega^{\ast})\|_{\mathbb{V}\times\mathbb{B}}\leq2\cdot\max\{\|\mathcal{G}\|_{\mathbb{V}\leftarrow\mathbb{U}\times\mathbb{A}},1\}\\[2mm]
&&\cdot\|[D\Psi_{L,M}(u^{\ast},\psi^{\ast},\omega^{\ast})]^{-1}\|_{\mathbb{U}\times\mathbb{A}\times\mathbb{B}\leftarrow\mathbb{U}\times\mathbb{A}\times\mathbb{B}}\\[2mm]
&&\cdot\|\Psi_{L,M}(u^{\ast},\psi^{\ast},\omega^{\ast})\|_{\mathbb{U}\times\mathbb{A}\times\mathbb{B}}.
\end{array}
\end{equation}
\end{theorem}
\noindent Recall that Proposition \ref{p_G} holds for the second factor in the right-hand side of \eqref{converr}. More importantly, thanks to Lemma \ref{l_DPsiLMinv}, the error on $(v^{\ast},\omega^{\ast})$ is determined by the last factor, namely the \emph{consistency error}. For the latter \eqref{PsiLMbound} in the proof of Proposition \ref{p_CS2b} holds and, in view of Corollary \ref{c_PLRL} in Appendix \ref{s_basic}, we can write
\begin{equation}\label{epsLepsM}
\|\Psi_{L,M}(u^{\ast},\psi^{\ast},\omega^{\ast})\|_{\mathbb{U}\times\mathbb{A}\times\mathbb{B}}\leq\varepsilon_{L}+\max\{\Lambda_{m}+\Lambda'_m,1\}\varepsilon_{M},
\end{equation}
where the important terms are
\begin{equation}\label{epsL}
\varepsilon_{L}:=\|(I_{\mathbb{U}\times\mathbb{A}\times\mathbb{B}}-P_{L}R_{L})(u^{\ast},\psi^{\ast},\omega^{\ast})\|_{\mathbb{U}\times\mathbb{A}\times\mathbb{B}}
\end{equation}
and
\begin{equation}\label{epsM}
\varepsilon_{M}:=\|\Phi_{M}(u^{\ast},\psi^{\ast},\omega^{\ast})-\Phi(u^{\ast},\psi^{\ast},\omega^{\ast})\|_{\mathbb{U}\times\mathbb{A}\times\mathbb{B}}.
\end{equation}
We call these contributions respectively {\it primary} and {\it secondary} consistency errors, and we analyze them separately in the following sections.
\subsection{Primary consistency error}
\label{s_primary}
The error term $\varepsilon_{L}$ in \eqref{epsL} concerns only the primary discretization and, according to \eqref{RL}, \eqref{PL} and \eqref{normprod} we have
\begin{equation*}
\varepsilon_{L}\leq\max\{\|u^{\ast}-\pi_{L}^{+}\rho_{L}^{+}u^{\ast}\|_{\mathbb{U}},\|\psi^{\ast}-\pi_{L}^{-}\rho_{L}^{-}\psi^{\ast}\|_{\mathbb{A}}\}.
\end{equation*}
Therefore a bound on $\varepsilon_{L}$ depends on the regularity of both $u^{\ast}$ and $\psi^{\ast}$, so that we prove the following result.
\begin{theorem}\label{t_epsL}
Let $G\in\mathcal{C}^{p}(\mathtt{Y},\mathbb{R}^{d})$ for some integer $p\geq1$. Then, Under \ref{T1}, \ref{T2}, \ref{N1} and \ref{N2}, it holds that $u^{\ast}\in C^{p}([0,1],\mathbb{R}^{d})$, $\psi^{\ast}\in C^{p+1}([-1,0],\mathbb{R}^{d})$, $v^{\ast}\in C^{p+1}([-1,1],\mathbb{R}^{d})$ and
\begin{equation}\label{epsLh}
\varepsilon_{L}=O\left(h^{\min\{m,p\}}\right).
\end{equation}
\end{theorem}
\begin{proof}
Recall that $v^{\ast}=\mathcal{G}(u^{\ast},\psi^{\ast})$ satisfies \eqref{bvp2}, hence its periodic extension to $[-1,\infty]$ is a periodic solution of \eqref{rfdet}. Given the existence of this solution, if $G$ is only continuous, then $u^{\ast}$ is continuous and $v^{\ast}$ is continuously differentiable in $[0,+\infty)$. It follows that $\psi^{\ast}$ is continuously differentiable by periodicity and, moreover, ${\psi^{\ast}}'(0)=u^{\ast}(0)$ follows again by periodicity since ${v^{\ast}}'$ is continuous at $1$. This means that $v^{\ast}$ is continuously differentiable in $[-1,1]$. As a consequence, if $p=1$, $u^{\ast}$ becomes continuously differentiable and the whole reasoning can be repeated, proving the first part of the result. This is a consequence of the well-known {\it smoothing effect} of RFDEs.

To prove \eqref{epsLh}, we observe first that
\begin{equation}\label{epsLu1}
\|u^{\ast}-\pi_{L}^{+}\rho_{L}^{+}u^{\ast}\|_{\mathbb{U}}\leq\frac{\|{u^{\ast}}^{(m+1)}\|_{\infty}}{(m+1)!}\cdot h^{m+1}
\end{equation}
holds if $p\geq m+1$, while
\begin{equation}\label{epsLu2}
\|u^{\ast}-\pi_{L}^{+}\rho_{L}^{+}u^{\ast}\|_{\mathbb{U}}\leq(1+\Lambda_{m})\left(\frac{h}{2}\right)^{p}\frac{c_{p}}{m^{p}}\cdot\|{u^{\ast}}^{(p)}\|_{\infty}
\end{equation}
holds if $p\leq m+1$, with $c_{p}$ a positive constant independent of $m$. \eqref{epsLu1} is a direct consequence of the standard Cauchy interpolation reminder, see, e.g., \cite[Section 6.1, Theorem 2]{kc02}. \eqref{epsLu2} is a direct consequence of Jackson's theorem on best uniform approximation, see, e.g., \cite[(2.9) and (2.11)]{mas15I}.

Secondly, similar results can be obtained for the component in $\mathbb{A}$, by recalling that $\|\cdot\|_{\mathbb{A}}$ is given by the second of \eqref{normB1inf}. Indeed, on the one hand, based on the same arguments used above for \eqref{epsLu1} and \eqref{epsLu2},
\begin{equation*}
\|\psi^{\ast}-\pi_{L}^{-}\rho_{L}^{-}\psi^{\ast}\|_{\infty}\leq\frac{\|{\psi^{\ast}}^{(m+1)}\|_{\infty}}{(m+1)!}\cdot h^{m+1}
\end{equation*}
holds if $p\geq m$, while
\begin{equation*}
\|\psi^{\ast}-\pi_{L}^{-}\rho_{L}^{-}\psi^{\ast}\|_{\infty}\leq(1+\Lambda_{m})\left(\frac{h}{2}\right)^{p+1}\frac{c_{p}'}{m^{p+1}}\cdot\|{\psi^{\ast}}^{(p+1)}\|_{\infty}
\end{equation*}
holds if $p\leq m$, with $c_{p}'$ a positive constant independent of $m$. On the other hand,
\begin{equation}\label{epsLpsi1'}
\|(\psi^{\ast}-\pi_{L}^{-}\rho_{L}^{-}\psi^{\ast})'\|_{\infty}\leq\frac{\|{\psi^{\ast}}^{(m+1)}\|_{\infty}}{m!}\cdot h^{m}
\end{equation}
holds if $p\geq m$, while
\begin{equation}\label{epsLpsi2'}
\|(\psi^{\ast}-\pi_{L}^{-}\rho_{L}^{-}\psi^{\ast})'\|_{\infty}\leq\Lambda_{m}\left(\frac{h}{2}\right)^{p}\frac{c_{p}''}{m^{p-1}}\cdot\|{\psi^{\ast}}^{(p+1)}\|_{\infty}
\end{equation}
holds if $p\leq m$, with $c_{p}''$ a positive constant independent of $m$. In particular, \eqref{epsLpsi1'} follows by adapting the classical proof of the Cauchy interpolation reminder to the first derivative of the reminder itself, while \eqref{epsLpsi2'} follows similarly to \eqref{p-r--I'} in the proof of Lemma \ref{l_p-r--I} in Appendix \ref{s_basic} thanks to \cite[page 331]{mo01} and \cite[Corollary 1.4.1]{riv81}.
\end{proof}
\noindent Let us note that even in the case that the periodic solution is smooth enough, i.e., $p>m+1$, and assuming the absence of a secondary discretization, it turns out that the consistency error in \eqref{epsLepsM} is $O(h^{m})$, in contrast to $O(h^{m+1})$ as obtained in \cite{mas15I} (see in particular the conclusions therein). According to formulation \eqref{bvp2}, it is clear from the proof of Theorem \ref{t_epsL} that this difference is due to the need of discretizing also the infinite-dimensional space $\mathbb{A}$, a circumstance that is only mentioned in \cite{mas15NM}, rather than being concretely elaborated, and, simultaneously, to the fact that functions in $\mathbb{A}$ must be differentiable due to the need of differentiating with respect to the period, as already remarked several times. After all, formulation \eqref{bvp1}, in which $\mathbb{A}$ is finite-dimensional, does not even satisfy all the required assumptions to develop this convergence analysis.
\subsection{Secondary consistency error}
\label{s_secondary}
The error term $\varepsilon_{M}$ in \eqref{epsM} concerns only the secondary discretization and, according to \eqref{Phi2} and \eqref{PhiM}, it reduces to
\begin{equation}\label{epsMU}
\varepsilon_{M}:=\omega^{\ast}\|G_{M}(v^{\ast}_{\cdot}\circ s_{\omega^{\ast}})-G(v^{\ast}_{\cdot}\circ s_{\omega^{\ast}})\|_{\mathbb{U}}.
\end{equation}
Of course this error is absent in case a secondary discretization is not needed. Conversely, as already remarked, the latter is necessary when the equation contains distributed delays, in which case it is determined by applying suitable quadrature rules to approximate the concerned integrals. Thus we can safely think that \eqref{epsMU} is basically a quadrature error, and in this respect we can assume to choose a formula that guarantees at least the same order of the primary consistency error (as far as $M$ varies proportionally to $L$). Alternatively, we can assume that \eqref{epsMU} falls below a given tolerance, say ${\rm TOL}$, and accept the fact that the consistency error decays down to ${\rm TOL}$ as fast as the primary consistency error.
{\color{black}\subsection{Regularity hypotheses in concrete cases}
\label{s_practical}
In realistic delay models it is frequent to encounter right-hand sides of the form
\begin{equation*}
G(\psi)=g(\psi(0),\psi(-\tau)).
\end{equation*}
It is easy to check that, in this case, \ref{T3} and \ref{T4} hold whenever both partial derivatives of $g$ exist and are continuous. Moreover, \ref{T5} holds as well if such derivatives are Lipschitz-continuous. Most frequently $g$ has an exact definition and does not need to be discretized but, if it did, then $g_M$ would need to fulfill the same regularity requirements for \ref{N3}-\ref{N5} to hold. Finally, \ref{N7} holds if the corresponding convergence condition holds for both the partial derivatives. The above observations can be extended to the case of multiple discrete delays, i.e., $G(\psi)=g(\psi(0),\psi(-\tau_1),\ldots,\psi(-\tau_n))$.

Right-hand sides featuring distributed delays are also common in literature. Remark \ref{r_N5} already includes some observations on the regularity requirements in order to satisfy the conditions up to \ref{T5} and \ref{N5}. Note that using a convergent interpolatory formula allows to satisfy conditions \ref{N6}-\ref{N7} as well.}
\subsection{Convergence of the spectral element method}
\label{s_SEM}
Let us recall from Section \ref{s_introduction} and Remark \ref{r_FEMSEM} that two methods can be considered as far as the convergence of the proposed piecewise collocation strategy is concerned. In particular, with reference to the primary discretization under \ref{N1} and \ref{N2}, the FEM consists in letting $L\rightarrow\infty$ while keeping $m$ fixed, while the SEM consists in letting $m\rightarrow\infty$ while keeping $L$ fixed. The analysis carried out in Section \ref{s_validation_n} and in Section \ref{s_primary} is presented for the FEM. Under this framework the final convergence result guarantees an error of magnitude $O(L^{-m})$ under suitable regularity conditions, see Theorem \ref{t_epsL}.

Unfortunately, as anticipated in Remark \ref{r_FEMSEM}, we are not able to guarantee the convergence of the SEM under this framework. Indeed, there are several points of the analysis either in Section \ref{s_validation_n} or in Section \ref{s_primary} (as well as in Appendix \ref{s_appendix}) which fails for the SEM based on \ref{N1} and \ref{N2}. Partial remedies can be advanced for some of these points by refining the requirements of regularity, yet some others seem not amenable of a definitive solution, or at least to a simple one. Below we comment on this and related aspects, but first it is remarkable to observe that some numerical experiments run by the authors indicate that indeed the SEM converges, so that it is our conclusion that an error analysis different from the one proposed in \cite{mas15I,mas15II,mas15NM} is necessary for the periodic case. We may investigate this issue in the future, recalling anyway that the FEM is preferred (and used) in practical implementations.

\bigskip
Concerning the analysis in Section \ref{s_validation_n}, the first point where the SEM fails is in the proof of Lemma \ref{l_LrfdeLM}, in particular due to \eqref{pr+KM+toK+} in Lemma \ref{l_prKMtoK}. In the case of the SEM indeed,
\begin{equation*}
\lim_{m,M\rightarrow\infty}\|(\pi_{L}^{+}\rho_{L}^{+}-I_{\mathbb{U}})\mathcal{K}_{M}^{\ast,+}\|_{\mathbb{U}\leftarrow\mathbb{U}}=0
\end{equation*}
should hold, but functions in the range of $\mathcal{K}_{M}^{\ast,+}$ are only continuous under \ref{N4} as observed in the proof of Lemma \ref{l_prKMtoK}. Nevertheless, the problem can be easily overcome by assuming \ref{N5}, which guarantees the functions in the range of $\mathcal{K}_{M}^{\ast,+}$ to be Lipschitz continuous (since $\mathcal{G}^{+}$ already sends to Lipschitz continuous functions) and thus we have as a classic result the uniform convergence of the interpolant with respect to increasing the polynomial degree $m$ as far as the collocation abscissae are of Chebyshev-type. The same point arises in the proofs of both Lemmas \ref{l_TLMtoT} and \ref{l_DPsiLMinv}.

A different issue emerges in the convergence of $\varepsilon_{L,M}^{+}$ in \eqref{epsLM+}, always in the proof of Lemma \ref{l_DPsiLMinv}, with specific reference to the first addend in the right-hand side of the latter (as for the second addend apply the arguments on the first point above). Indeed, it is well-known that $\Lambda_{m}$ in \eqref{p+r+} of Lemma \ref{l_p+r+-I} in Appendix \ref{s_basic} grows unbounded independently of the choice of the collocation abscissae, at least as $O(\log m)$ (and at most with the same order in case of Chebyshev-type nodes). Therefore, in order to ensure convergence one should assume to balance this growth with the rate of convergence of the secondary discretization, being the attention focused on the term $\pi_{L}^{+}\rho_{L}^{+}(\mathfrak{L}_{M}^{\ast}-\mathfrak{L}^{\ast})$. Since primary and secondary discretizations can be chosen independently, we can consider this balance a reasonable option. Of course, if a secondary discretization is not required, the term is not even present and the issue becomes meaningless. Let us note that similar questions appear in the convergence of $\varepsilon_{\omega,L,M}$ in \eqref{epsLM}. In the latter also the convergence of $\omega_{L,M}$ to $\omega$ is required, see Proposition \ref{p_Ax*2LM-Ax*2}, whose validity for the SEM we discuss next.

The proof of \eqref{xiLM1toxi1} in Proposition \ref{p_Ax*2LM-Ax*2} holds for the FEM. The convergence of $\xi_{L,M,1}^{\ast}$ to $\xi_{1}^{\ast}$ depends on the four terms at the right-hand side of \eqref{xiLM1toxi1'}. In particular, for the first and the third ones, the same balance between primary and secondary discretizations mentioned above has to be considered. The second addend could be made vanishing by ensuring that the interpolation error $\|(\pi_{L}^{+}\rho_{L}^{+}-I_{\mathbb{U}})\mathfrak{M}_{M}^{\ast}\|_{\mathbb{U}}$ decays as fast as to override the growth of $\pi_{L}^{-}\rho_{L}^{-}$. The latter, according to \eqref{p-r-} of Lemma \ref{l_p-r--I} in Appendix \ref{s_basic}, grows at best as $O(m\log m)$ for Chebyshev-type nodes. Consequently, $\mathfrak{M}_{M}^{\ast}$ should be at least continuously differentiable with Lipschitz continuous first derivative, thus guaranteeing that the above interpolation error is $O(\log m/m^{2})$ and the sought balance is scored. We are thus left with the fourth term, which concerns the interpolation error in $\mathbb{A}$ of the function $\int_0^{1}[T^{\ast}(1,s)X_{0}]\mathfrak{M}^{\ast}(s)\dd s$. As this is the state at $1$ of \eqref{ivpvLM*}, it is not difficult to argue  that the above map has a Lipschitz continuous first derivative under \ref{T5}, so that also this last term may vanish. For this to happen we must choose the abscissae $c_{1},\ldots,c_{m-1}$ corresponding to Chebyshev-type zeros, as well as $c_{m}=1$. Only in this case in fact, we can apply \cite[Theorem 4.2.11]{mn08} to replace \eqref{p-r--I'} in Lemma \ref{l_p-r--I} in Appendix \ref{s_basic} with $\|(\pi_{L}^{-}\rho_{L}^{-}\psi-\psi)'\|_{\infty}\leq c\Lambda_{m}E_{m-1}(\psi')$, otherwise another factor $m$ in front appears, thus requiring a degree of regularity that cannot be obtained when the same analysis is carried-out for $\xi_{L,M,2}^{\ast}$ (see \eqref{err_u0} and the relevant comments).

Despite the above efforts, to complete the proof of Proposition \ref{p_Ax*2LM-Ax*2} boundedness of $\psi_{L,M}$ is required, and the latter becomes mandatory for Lemma \ref{l_DPsiLMinv} to hold as evident from \eqref{DPdiinv}. But it is equally evident that such boundedness is not guaranteed by \eqref{psiLMbounded} because of \eqref{p-r-} of Lemma \ref{l_p-r--I} in Appendix \ref{s_basic}. Yet it could well be that the norm of $[D\Psi_{L,M}(u^{\ast},\psi^{\ast},\omega^{\ast})]^{-1}$ grows with $m$, but not as fast as the (square root of the) consistency error $\|\Psi_{L,M}(u^{\ast},\psi^{\ast},\omega^{\ast})\|_{\mathbb{U}\times\mathbb{A}\times\mathbb{B}}$ decays, recall in fact \eqref{DPsiLMlim} in Proposition \ref{p_CS2b}. The hypothesis is not far from being reasonable, given that the consistency error depends on the regularity of the periodic solution at hands, in view of the similar result of Theorem \ref{t_epsL} for the SEM. In any case, all this requires a much more focused analysis, that, as already remarked, is out of the scopes of the present work. Let us anyway confirm that experimental tests seem to validate the possibility of (a spectrally accurate) convergence of the SEM.
\section{Concluding remarks}
\label{s_concluding}
Computing periodic solutions is a key issue in the dynamical analysis of systems. Piecewise orthogonal collocation is perhaps the most used technique, especially in a continuation framework. This paper is an attempt to furnish a fully-detailed and complete error analysis to prove the convergence of this method in the context of RFDEs.

The main result is given in terms of the FEM method, whose error behaves as $O(L^{-m})$ for $m$ the (fixed) degree of the piecewise polynomial and $L$ the (increasing) number of mesh intervals. Although this was largely expected, given the abundance of experimental results in the literature, it is nowhere proved for a general equation in this class. To close this gap we followed the abstract approach proposed in \cite{mas15I,mas15II,mas15NM}, converting the BVP into an operator fixed point problem. The effort consisted in furnishing proofs of the validity of the (theoretical and numerical) assumptions required to reach the final convergence result in \cite{mas15NM}. Along the way, a main difficulty was represented by the period of the concerned solution, showing up as an unknown parameter linked to the course of time. The need for differentiating with respect to parameters led to additional smoothness requirements for the functional spaces involved in the analysis. Among the several consequences, the obtaining of one order less than what was proved in \cite{mas15I} is perhaps the most evident (viz. $m$ instead of $m+1$). Moreover, (some of) these smoothness requirements caused also the impossibility of applying the abstract framework of \cite{mas15NM} to the classical BVP formulation (i.e., \eqref{bvp1}), thus requiring to work under periodic constraints formulated on the state space (i.e., \eqref{bvp2}).

\bigskip
If the problem for RFDEs can be (optimistically) regarded as closed, this is far from holding true for either the neutral case, or the case of renewal (Volterra integral) equations. Let us remark that precisely the latter class inspired the current research (mainly driven by models of population dynamics), as it is not even considered in continuation packages from the computational standpoint -- not to speak about convergence. At a first sight the differentiability properties above mentioned seem to pose serious obstacles in both these cases, and perhaps a substantial effort is required in this direction. Nevertheless, the present work offers a first, solid background to start elaborating a succeeding strategy towards the proof of convergence. The authors plan to make this effort in the immediate future, also to substantiate the encouraging experimental results already obtained by extending the piecewise collocation to renewal equations, even coupled to RFDEs (for the target class of realistic models we have in mind see, e.g., \cite{bdgsv16,bresisc15,ssg16}).

\bigskip
As a final remark, we observe that the current contents are almost exclusively devoted to the theoretical analysis. We have intentionally neglected to report on implementation issues: on the one hand, the literature is not lacking from this point of view; on the other hand, the collocation proposed here is definitely more meaningful for neutral problems (as one collocates the derivative of the solution rather than the solution itself). Nevertheless, the analysis is restricted to the retarded case to focus on the peculiar aspects of periodic problems, as to understand how they affect convergence. It is anyway important to observe that exactly this collocation strategy is the natural candidate to treat renewal equations, whose solution can be seen somehow as the derivative of (the solution of) a neutral RFDE. Eventually, this extension would also coincide with that of \cite{elir00}, the latter representing perhaps the first choice to implement (and indeed the one we adopted to perform the numerical tests).
\appendix
\section{Auxiliary results}
\label{s_appendix}
In the following we keep on avoiding the use of the index $2$ to refer to formulation \eqref{bvp2} even though some of the results are already used in Section \ref{s_validation_t} where the latter is not yet discharged.
\subsection{Basic results on the primary discretization}
\label{s_basic}
\begin{lemma}\label{l_p+r+-I}
Let $\rho_{L}^{+}$, $\pi_{L}^{+}$ and $\Lambda_{m}$ be defined respectively in \eqref{rL+}, \eqref{pL+} and \eqref{lebesgue} under \ref{N1} and let $C^{+}$ be defined in \eqref{C+}. Then, under \ref{T2},
\begin{equation}\label{p+r+}
\|\pi_{L}^{+}\rho_{L}^{+}\|_{\mathbb{U}\leftarrow\mathbb{U}}\leq\Lambda_{m}
\end{equation}
holds for all positive integers $L$ and
\begin{equation}\label{p+r+-I}
\lim_{L\rightarrow\infty}\|\pi_{L}^{+}\rho_{L}^{+}u-u\|_{\mathbb{U}}=0
\end{equation}
holds for all $u\in C^{+}$.
\end{lemma}
\begin{proof}
According to the notation of Section \ref{s_discretization},
\begin{equation*}
\pi_{L}^{+}\rho_{L}^{+}u(t)=\sum_{j=0}^m\ell_{m,i,j}(t)u(t_{i,j}^{+})
\end{equation*}
holds for $u\in\mathbb{U}$ and $t\in[t_{i-1}^+,t_i^+]$, $i=1,\ldots,L$. Then \eqref{p+r+} follows from
\begin{equation*}
\|\pi_{L}^{+}\rho_{L}^{+}u\|_{\mathbb{U}}\leq\max_{i=1,\ldots,L}\max_{t\in[t_{i-1}^{+},t_{i}^{+}]}\sum_{j=0}^m|\ell_{m,i,j}(t)|\|u\|_{\mathbb{U}}=\Lambda_{m}\|u\|_{\mathbb{U}}
\end{equation*}
thanks to \eqref{lebesgue}. As for \eqref{p+r+-I},
\begin{equation*}
\setlength\arraycolsep{0.1em}\begin{array}{rcl}
\pi_{L}^{+}\rho_{L}^{+}u(t)-u(t)&=&\displaystyle\sum_{j=0}^m\ell_{m,i,j}(t)u(t_{i,j}^{+})-\sum_{j=0}^m\ell_{m,i,j}(t)u(t)+\sum_{j=0}^m\ell_{m,i,j}(t)u(t)-u(t)\\[2mm]
&=&\displaystyle\sum_{j=0}^m\ell_{m,i,j}(t)[u(t_{i,j}^{+})-u(t)]+\left(\sum_{j=0}^m\ell_{m,i,j}(t)-1\right)u(t)\\[2mm]
&=&\displaystyle\sum_{j=0}^m\ell_{m,i,j}(t)[u(t_{i,j}^{+})-u(t)]
\end{array}
\end{equation*}
holds always for $t\in[t_{i-1}^{+},t_{i}^{+}]$, $i=1,\ldots,L$. Therefore $\|\pi_{L}^{+}\rho_{L}^{+}u-u\|_{\mathbb{U}}\leq\Lambda_{m}\omega(u;h)$, where $\omega$ denotes the modulus of continuity. The latter vanishes as $h\rightarrow0$ only if $u$ is at least continuous.
\end{proof}
\begin{lemma}\label{l_p-r--I}
Let $\rho_{L}^{-}$, $\pi_{L}^{-}$, $\Lambda_{m}$ and $\Lambda_{m}'$ be defined respectively in \eqref{rL-}, \eqref{pL-}, \eqref{lebesgue} and \eqref{lebesgue'} under \ref{N2} and let $C^{1,-}:=C^{1}([-1,0],\mathbb{R}^{d})$. Then, under \ref{T2},
\begin{equation}\label{p-r-}
\|\pi_{L}^{-}\rho_{L}^{-}\|_{\mathbb{A}\leftarrow\mathbb{A}}\leq\Lambda_{m}+\Lambda_{m}'
\end{equation}
holds for all positive integers $L$ and
\begin{equation}\label{p-r--I}
\lim_{L\rightarrow\infty}\|\pi_{L}^{-}\rho_{L}^{-}\psi-\psi\|_{\mathbb{A}}=0
\end{equation}
holds for all $\psi\in C^{1,-}$.
\end{lemma}
\begin{proof}
The proof of \eqref{p-r-} is analogous to that of \eqref{p+r+} in Lemma \ref{l_p+r+-I} once considered that $\|\cdot\|_{\mathbb{A}}$ is given by the second of \eqref{normB1inf}. As for \eqref{p-r--I}, $\|\pi_{L}^{-}\rho_{L}^{-}\psi-\psi\|_{\infty}\leq\Lambda_{m}\omega(\psi;h)$, follows similarly by the proof of \eqref{p+r+-I} in Lemma \ref{l_p+r+-I}, while
\begin{equation}\label{p-r--I'}
\|(\pi_{L}^{-}\rho_{L}^{-}\psi-\psi)'\|_{\infty}\leq c(m+1)\Lambda_{m}E_{m-1}(\psi')
\end{equation}
holds for some positive constant $c$ thanks to \cite[page 331]{mo01}, where $E_{m}(f)$ denotes the best uniform approximation error of $f$ with (piecewise) polynomials of degree $m$. As for the latter $E_{m-1}(\psi')\leq6\omega(\psi',h/2m)$ holds thanks to \cite[Corollary 1.4.1]{riv81}, so that it vanishes as $h\rightarrow0$ only if $\psi$ is at least continuously differentiable.
\end{proof}
\begin{corollary}\label{c_PLRL}
Let $R_{L}$, $P_{L}$, $\Lambda_{m}$ and $\Lambda_{m}'$ be defined respectively in \eqref{RL}, \eqref{PL}, \eqref{lebesgue} and \eqref{lebesgue'} under \ref{N1} and \ref{N2}. Then
\begin{equation}\label{PLRL}
\|P_{L}R_{L}\|_{\mathbb{U}\times\mathbb{A}\times\mathbb{B}\leftarrow\mathbb{U}\times\mathbb{A}\times\mathbb{B}}\leq\max\{\Lambda_{m}+\Lambda'_m,1\}
\end{equation}
holds for all positive integers $L$.
\end{corollary}
\begin{proof}
It holds
\begin{equation*}
\setlength\arraycolsep{0.1em}\begin{array}{rcl}
\|P_{L}R_{L}\|_{\mathbb{U}\times\mathbb{A}\times\mathbb{B}\leftarrow\mathbb{U}\times\mathbb{A}\times\mathbb{B}}&=&\max\{\|\pi_{L}^{+}\rho_{L}^{+}\|_{\mathbb{U}\leftarrow\mathbb{U}},\|\pi_{L}^{-}\rho_{L}^{-}\|_{\mathbb{A}\leftarrow\mathbb{A}},\|I_{\mathbb{B}}\|_{\mathbb{B}\leftarrow\mathbb{B}}\}\\[2mm]
&\leq&\max\{\Lambda_{m},\Lambda_{m}+\Lambda'_m,1\}\\[2mm]
&=&\max\{\Lambda_{m}+\Lambda'_m,1\}
\end{array}
\end{equation*}
thanks to \eqref{normprod}, \eqref{p+r+} in Lemma \ref{l_p+r+-I} and \eqref{p-r-} in Lemma \ref{l_p-r--I}.
\end{proof}
\subsection{Other preparatory results}
\label{s_other}
\begin{lemma}\label{l_v*}
Let $(u^{\ast},\psi^{\ast},\omega^{\ast})\in\mathbb{U}\times\mathbb{A}\times\mathbb{B}$ be a fixed point of $\Phi$ in \eqref{Phi2}. Then, under \ref{T2}, $v^{\ast}:=\mathcal{G}(u^{\ast},\psi^{\ast})$ for $\mathcal{G}$ in \eqref{G2} is Lipschitz continuous, in particular
\begin{equation*}
|v^{\ast}(t_{1})-v^{\ast}(t_{2})|\leq\|(u^{\ast},\psi^{\ast},\omega^{\ast})\|_{\mathbb{U}\times\mathbb{A}\times\mathbb{B}}\cdot|t_{1}-t_{2}|
\end{equation*}
holds for all $t_{1},t_{2}\in[-1,1]$.
\end{lemma}
\begin{proof}
By the choice of $\mathbb{V}_{2}$ in \ref{T2}, the derivative ${v^{\ast}}'$ of $v^{\ast}$ is bounded, hence $v^{\ast}$ is Lipschitz continuous with Lipschitz constant $\|{v^{\ast}}'\|_{\infty}$. By the choice of $\mathbb{A}$ in \ref{T2}, ${v^{\ast}}'$ reads
\begin{equation*}
{v^{\ast}}'(t):=\begin{cases}
\displaystyle u^{\ast}(t),&t\in[0,1],\\[2mm]
{\psi^{\ast}}'(t),&t\in[-1,0)
\end{cases}
\end{equation*}
and hence
\[
\|{v^{\ast}}'\|_{\infty}=\max\{\|u^{\ast}\|_{\infty},\|{\psi^{\ast}}'\|_{\infty}\}\leq\max\{\|u^{\ast}\|_{\mathbb{U}},\|{\psi^{\ast}}\|_{\mathbb{A}}\}\leq\|(u^{\ast},\psi^{\ast},\omega^{\ast})\|_{\mathbb{U}\times\mathbb{A}\times\mathbb{B}}.
\]
\end{proof}
\begin{lemma}\label{l_v*'}
Let $(u^{\ast},\psi^{\ast},\omega^{\ast})\in\mathbb{U}\times\mathbb{A}\times\mathbb{B}$ be a fixed point of $\Phi$ in \eqref{Phi2} and $v^{\ast}:=\mathcal{G}(u^{\ast},\psi^{\ast})$ for $\mathcal{G}$ in \eqref{G2}. Then, under \ref{T2} and \ref{T5}, $u^{\ast}$, ${\psi^{\ast}}'$ and ${v^{\ast}}'$ are Lipschitz continuous, in particular
\begin{equation*}
|{v^{\ast}}'(t_{1})-{v^{\ast}}'(t_{2})|\leq\omega^{\ast}\kappa_{2,1}\|(u^*,\psi^*,\omega^*)\|_{\mathbb{U}\times\mathbb{A}\times\mathbb{B}}\cdot|t_{1}-t_{2}|
\end{equation*}
holds for all $t_{1},t_{2}\in[-1,1]$ and $\kappa_{2,1}$ in \eqref{K21}.
\end{lemma}
\begin{proof}
Thanks to \ref{T5}, the constant $\kappa_{2,1}$ in \eqref{K21} is well-defined, and, for $t_1,t_{2}\in[0,1]$, from \eqref{rfdet} we get
\begin{equation*}
\setlength\arraycolsep{0.1em}\begin{array}{rcl}
|u^{\ast}(t_{1})-u^{\ast}(t_{2})|&\leq&\omega^{\ast}\kappa_{2,1}\|v^{\ast}_{t_{1}}\circ s_{\omega^{\ast}}-v^{\ast}_{t_{2}}\circ s_{\omega^{\ast}}\|_{\mathtt{Y}}.
\end{array}
\end{equation*}
Therefore, $u^{\ast}$ is Lipschitz continuous since so is $v^{\ast}$ by Lemma \ref{l_v*}, the Lipschitz constant being $\omega^{\ast}\kappa_{2,1}\|(u^*,\psi^*,\omega^*)\|_{\mathbb{U}\times\mathbb{A}\times\mathbb{B}}$ thanks to the latter. The same holds for ${\psi^{\ast}}'$ by periodicity and, in turn, for ${v^{\ast}}'$ being it the continuous junction of two Lipschitz continuous functions with the same Lipschitz constant.
\end{proof}
\begin{lemma}\label{l_prKMtoK}
Let $\rho_{L}^{+}$ and $\pi_{L}^{+}$ be defined respectively in \eqref{rL+} and \eqref{pL+} under \ref{N1} and $\mathcal{K}^{\ast,+}$, $\mathcal{K}^{\ast,-}$, $\mathcal{K}_{M}^{\ast,+}$ and $\mathcal{K}_{M}^{\ast,-}$ be defined in \eqref{GKKM}. Then, under \ref{T2}, \ref{N4} and \ref{N7},
\begin{equation}\label{pr+KM+toK+}
\lim_{L,M\rightarrow\infty}\|\pi_{L}^{+}\rho_{L}^{+}\mathcal{K}_{M}^{\ast,+}-\mathcal{K}^{\ast,+}\|_{\mathbb{U}\leftarrow\mathbb{U}}=0,
\end{equation}
and
\begin{equation}\label{pr+KM-toK-}
\lim_{L,M\rightarrow\infty}\|\pi_{L}^{+}\rho_{L}^{+}\mathcal{K}_{M}^{\ast,-}-\mathcal{K}^{\ast,-}\|_{\mathbb{U}\leftarrow\mathbb{A}}=0.
\end{equation}
\end{lemma}
\begin{proof}
As for \eqref{pr+KM+toK+} we have
\begin{equation*}
\|\pi_{L}^{+}\rho_{L}^{+}\mathcal{K}_{M}^{\ast,+}-\mathcal{K}^{\ast,+}\|_{\mathbb{U}\leftarrow\mathbb{U}}\leq\|(\pi_{L}^{+}\rho_{L}^{+}-I_{\mathbb{U}})\mathcal{K}_{M}^{\ast,+}\|_{\mathbb{U}\leftarrow\mathbb{U}}+\|\mathcal{K}_{M}^{\ast,+}-\mathcal{K}^{\ast,+}\|_{\mathbb{U}\leftarrow\mathbb{U}}.
\end{equation*}
The second addend in the right-hand side above vanishes thanks to \ref{N7}. It thus also follows that $\mathcal{K}_{M}^{\ast,+}$ is uniformly bounded with respect to $M$. This in turn makes the first addend vanish as well thanks to \eqref{p+r+-I} of Lemma \ref{l_p+r+-I}, given that $\mathcal{K}_{M}^{\ast,+}\mathbb{U}\subseteq C^{+}$ as it follows from \eqref{G2} through the definition of $\mathcal{G}^{+}$ in \eqref{GKKM} and the continuity of $\mathfrak{L}_{M}^{\ast}$ under \ref{N4}. Similar arguments hold for \eqref{pr+KM-toK-}.
\end{proof}
\begin{lemma}\label{l_prLMMtoLM}
Let $\mathfrak{L}^{\ast}$ and $\mathfrak{M}^{\ast}$ be defined in \eqref{L*M*} and $\mathfrak{L}_{M}^{\ast}$ and $\mathfrak{M}_{M}^{\ast}$ be defined in \eqref{LM*MM*}. Then, under \ref{N6} and \ref{N7},
\begin{equation}\label{LMtoL}
\lim_{M\rightarrow\infty}\|\mathfrak{L}_{M}^{\ast}-\mathfrak{L}^{\ast}\|_{\mathcal{L}(\mathtt{Y},\mathbb{R}^{d})\leftarrow[0,1]}=0
\end{equation}
and
\begin{equation}\label{MMtoM}
\lim_{M\rightarrow\infty}\|\mathfrak{M}_{M}^{\ast}-\mathfrak{M}^{\ast}\|_{\infty}=0.
\end{equation}
\end{lemma}
\begin{proof}
Straightforward from the definitions.
\end{proof}
\begin{proposition}\label{p_Ax*2LM-Ax*2}
Let $\omega$ and $\omega_{L,M}$ be given as in \eqref{omega} and \eqref{omegaLM}, respectively. Then, under \ref{T4}, \ref{N4}, \ref{N6} and \ref{N7},
$$
\lim_{L,M\rightarrow\infty}\omega_{L,M}=\omega.
$$
\end{proposition}
\begin{proof}
Let $\xi_{1}^{\ast}$ and $\xi_{2}^{\ast}$ be
as in the proof of Proposition \ref{p_Ax*2} and $\xi_{L,M,1}^{\ast}$, $\xi_{L,M,2}^{\ast}$ and $\nu_{L,M}$ be as in \eqref{xinuLM}. First we show that
\begin{equation}\label{xiLM1toxi1}
\lim_{L,M\rightarrow\infty}\|\xi_{L,M,1}^{\ast}-\xi_{1}^{\ast}\|_{Y}=0
\end{equation}
for $\xi_{1}^{\ast}$ in \eqref{xi1}. From the latter and the first of \eqref{xiLM} we have
\begin{equation}\label{xiLM1toxi1'}
\setlength\arraycolsep{0.1em}\begin{array}{rcl}
\xi_{L,M,1}^{\ast}-\xi_{1}^{\ast}&=&\displaystyle\pi_{L}^{-}\rho_{L}^{-}\int_0^{1}[T_{L,M}^{\ast}(1,s)X_{0}]\pi_{L}^{+}\rho_{L}^{+}\mathfrak{M}_{M}^{\ast}(s)\dd s\\[4mm]
&&\displaystyle-\int_0^{1}[T^{\ast}(1,s)X_{0}]\mathfrak{M}^{\ast}(s)\dd s\\[4mm]
&=&\displaystyle\pi_{L}^{-}\rho_{L}^{-}\int_0^{1}[(T_{L,M}^{\ast}(1,s)-T^{\ast}(1,s))X_{0}]\pi_{L}^{+}\rho_{L}^{+}\mathfrak{M}_{M}^{\ast}(s)\dd s\\[4mm]
&&+\displaystyle\pi_{L}^{-}\rho_{L}^{-}\int_0^{1}[T^{\ast}(1,s)X_{0}][\pi_{L}^{+}\rho_{L}^{+}-I_{\mathbb{U}}]\mathfrak{M}_{M}^{\ast}(s)\dd s\\[4mm]
&&+\displaystyle\pi_{L}^{-}\rho_{L}^{-}\int_0^{1}[T^{\ast}(1,s)X_{0}](\mathfrak{M}_{M}^{\ast}(s)-\mathfrak{M}^{\ast}(s))\dd s\\[4mm]
&&+\displaystyle(\pi_{L}^{-}\rho_{L}^{-}-I_{\mathbb{A}})\int_0^{1}[T^{\ast}(1,s)X_{0}]\mathfrak{M}^{\ast}(s)\dd s.
\end{array}
\end{equation}
The third addend in the right-hand side of the last equality above vanishes thanks to \eqref{p-r-} of Lemma \ref{l_p-r--I} and \eqref{MMtoM} of Lemma \ref{l_prLMMtoLM}. The latter implies also that $\mathfrak{M}_{M}^{\ast}$ is uniformly bounded, so that the first addend vanishes as well thanks to Lemma \ref{l_TLMtoT}, \eqref{p+r+} of Lemma \ref{l_p+r+-I} and \eqref{p-r-} of Lemma \ref{l_p-r--I}. The second addend vanishes similarly thanks to \eqref{p+r+-I} of Lemma \ref{l_p+r+-I} since $\mathfrak{M}_{M}^{\ast}$ is continuous under \ref{N4} as it follows from \eqref{M2}, the second of \eqref{L*M*} and thanks to Lemma \ref{l_v*'}. Finally, as for the last addend, note that $\int_0^{1}[T^{\ast}(1,s)X_{0}]\mathfrak{M}^{\ast}(s)\dd s$ is the state solution at $1$ of
\begin{equation}\label{ivpvLM*}
\left\{\setlength\arraycolsep{0.1em}\begin{array}{l}
v'(t)=\mathfrak{L}^{\ast}(t)v_{t}\circ s_{\omega^{\ast}}+\mathfrak{M}^{\ast}(t),\quad t\in[0,1],\\[2mm]
v_{0}=0,
\end{array}
\right.
\end{equation}
as it can be seen by applying the variation of constants formula as done for \eqref{vLM}. As such it is continuously differentiable, being the right-hand side of the RFDE continuous under \ref{T4} similarly as already observed above for $\mathfrak{M}_{M}^{\ast}$. Therefore, also this addend vanishes thanks to \eqref{p-r--I} of Lemma \ref{l_p-r--I}. Since \eqref{xiLM1toxi1} holds, the second of \eqref{muLMphiLM} in Lemma \ref{l_TLMtoT} gives also
\begin{equation}\label{k1LMtok1}
\lim_{L,M\rightarrow\infty}k_{L,M,1}=k_{1}.
\end{equation}
Note that the definition of $\xi_1^*$ does not change if one considers system \eqref{Ax*23} in place of \eqref{Ax*20}. Similarly, the definition of $\xi_{1,L,M}$ does not change if one considers system \eqref{Ax*24LM} in place of \eqref{Ax*20LM}.

Now, by trying a similar reasoning to prove that
\begin{equation}\label{xiLM2toxi2}
\lim_{L,M\rightarrow\infty}\|\xi_{L,M,2}^{\ast}-\xi_{2}^{\ast}\|_{Y}=0,
\end{equation}
one ends up with the term $(\pi_{L}^{-}\rho_{L}^{-}-I_{\mathbb{A}})\int_0^{1}[T^{\ast}(1,s)X_{0}]u_{0}(s)\dd s$ in place of the fourth addend analyzed above, where $\int_0^{1}[T^{\ast}(1,s)X_{0}]u_{0}(s)\dd s$ is the state solution at $1$ of
\begin{equation*}
\left\{\setlength\arraycolsep{0.1em}\begin{array}{l}
v'(t)=\mathfrak{L}^{\ast}(t)v_{t}\circ s_{\omega^{\ast}}+u_{0}(t),\quad t\in[0,1],\\[2mm]
v_{0}=0.
\end{array}
\right.
\end{equation*}
Unfortunately, the latter is not necessarily continuously differentiable since $u_{0}\in\mathbb{U}$ is not necessarily continuous, so that \eqref{p-r--I} of Lemma \ref{l_p-r--I} cannot be applied. Nevertheless, it is not difficult to argue that the proof of \eqref{xiLM2toxi2} can be accomplished by using the same arguments adopted in the proof of Lemma \ref{l_DPsiLMinv}, i.e., by repeating the proof of Proposition \ref{p_Ax*2} for \eqref{Ax*23} in place of \eqref{Ax*20}, as well as the proof of Proposition \ref{p_Ax*2LM} for \eqref{Ax*24LM} in place of \eqref{Ax*20LM}. Going this way would give rise to the term
\begin{equation}\label{err_u0}
(\pi_{L}^{-}\rho_{L}^{-}-I_{\mathbb{A}})\int_0^{1}[T^{\ast}(1,s)X_{0}](\mathfrak{L}^{\ast}\mathcal{G}(u_{0},\psi_{0})_{\cdot}\circ s_{\omega^{\ast}})(s)\dd s,
\end{equation}
where now $\mathfrak{L}^{\ast}\mathcal{G}(u_{0},\psi_{0})_{\cdot}\circ s_{\omega^{\ast}}$ is continuous under \ref{T4}. We omit the details since they would not add any novelty to what already elaborated in the proofs mentioned above. Obviously, \eqref{xiLM2toxi2} guarantees also that
\begin{equation}\label{k2LMtok2}
\lim_{L,M\rightarrow\infty}k_{L,M,2}=k_{2}
\end{equation}
thanks to \eqref{muLMphiLM} in Lemma \ref{l_TLMtoT}.

Eventually,
\begin{equation}\label{hLMto0}
\lim_{L,M\rightarrow\infty}h_{L,M}=0
\end{equation}
follows since in the third of \eqref{xinuLM} $\varphi_{L,M}$ converges to $\varphi$ thanks to Lemma \ref{l_TLMtoT} again and $\nu_{L,M}$ in \eqref{nuLM} vanishes. The latter statement is a consequence of $\psi_{L,M}$ being bounded (see below), $\mu_{L,M}\rightarrow1$ from Lemma \ref{l_TLMtoT} and that $(\pi_{L}^{-}\rho_{L}^{-}-I_{\mathbb{A}})T_{L,M}^{\ast}(1,0)$ vanishes since
\begin{equation*}
(\pi_{L}^{-}\rho_{L}^{-}-I_{\mathbb{A}})T_{L,M}^{\ast}(1,0)=(\pi_{L}^{-}\rho_{L}^{-}-I_{\mathbb{A}})[T_{L,M}^{\ast}(1,0)-T^{\ast}(1,0)]+(\pi_{L}^{-}\rho_{L}^{-}-I_{\mathbb{A}})T^{\ast}(1,0).
\end{equation*}
Indeed, the right-hand side above vanishes under \ref{N2} thanks to \eqref{p-r--I} of Lemma \ref{l_p-r--I}, Lemma \ref{l_TLMtoT} again and to the fact that the range of $T^{\ast}(1,0)$ contains only continuously differentiable functions. Finally, that $\psi_{L,M}$ is bounded follows from
\begin{equation}\label{psiLMbounded}
\|\psi_{L,M}\|_{\mathbb{A}}\leq(\Lambda_{m}+\Lambda_{m}')\|\mathcal{G}(u_{L,M},\psi_{L,M})_{1}\|_{\infty}+\|\psi_{0}\|_{\mathbb{A}},
\end{equation}
which holds from the second of \eqref{Ax*20LM}, and where boundedness of $\|\mathcal{G}(u_{L,M},\psi_{L,M})_{1}\|_{\infty}$ $=\|\mathcal{G}(u_{L,M},\psi_{L,M})\vert_{[0,1]}\|_{\infty}$ follows from the third of \eqref{Ax*20LM} and the continuity of $p$. 

\bigskip
In conclusion, by \eqref{k1LMtok1}, \eqref{k2LMtok2} and \eqref{hLMto0},
\[
\setlength\arraycolsep{0.1em}\begin{array}{rcl}
\displaystyle\lim_{L,M\to\infty}(\omega_{L,M}-\omega)&=&\displaystyle\lim_{L,M\to\infty}\left(-\frac{k_{L,M,2}+h_{L,M}}{k_{L,M,1}}+\frac{k_{2}}{k_{1}}\right)\\[4mm]
&=&\displaystyle\lim_{L,M\to\infty}\frac{-k_{L,M,2}k_{1}-h_{L,M}k_{1}+k_{2}k_{L,M,1}}{k_{L,M,1}k_{1}}\\[4mm]
&=&\displaystyle\frac{-k_{2}k_{1}-0\cdot k_{1}+k_{2}k_{1}}{k_{1}k_{1}}=0.
\end{array}
\]
\end{proof}
\begin{lemma}\label{l_prG1toG}
Let $\rho_{L}^{-}$ and $\pi_{L}^{-}$ be defined respectively in \eqref{rL-} and \eqref{pL-} under \ref{N2} and $\mathcal{G}_{1}^{+}$ and $\mathcal{G}_{1}^{+}$ be defined in \eqref{G1+-}. Then, under \ref{T2},
\begin{equation}\label{pr-G1+toG1+}
\lim_{L,M\rightarrow\infty}\|\pi_{L}^{-}\rho_{L}^{-}\mathcal{G}_{1}^{+}-\mathcal{G}_{1}^{+}\|_{\mathbb{A}\leftarrow C^{+}}=0
\end{equation}
and
\begin{equation}\label{pr-G1-toG1}
\lim_{L,M\rightarrow\infty}\|\pi_{L}^{-}\rho_{L}^{-}\mathcal{G}_{1}^{-}-\mathcal{G}_{1}^{-}\|_{\mathbb{A}\leftarrow\mathbb{A}}=0.
\end{equation}
\end{lemma}
\begin{proof}
\eqref{pr-G1+toG1+} follows from \eqref{p-r--I} of Lemma \ref{l_p-r--I} and the fact that $\mathcal{G}_{1}^{+}C^{+}$ contains only continuously differentiable functions by the first of \eqref{G1+-}. Similarly, \eqref{pr-G1-toG1} follows from the fact that $\mathcal{G}_{1}^{-}\mathbb{A}$ contains only constant functions by the second of \eqref{G1+-}.
\end{proof}
\section*{Acknowledgments}
The authors wish to express their sincere gratitude to Stefano Maset (University of Trieste) for a continuous discussion on the subject, to Duccio Papini and Fabio Zanolin (University of Udine) for their precious advices on duality and Fredholm theory and to Rossana Vermiglio (University of Udine) for a careful reading of the manuscript.

\end{document}